\documentclass[11pt, twoside]{article}
\usepackage{mathrsfs}
\usepackage{amssymb}
\usepackage{amsmath}
\usepackage{mathrsfs}
\usepackage{amsthm}
\usepackage{amsfonts}
\usepackage{color}
\usepackage{latexsym}
\usepackage{txfonts}
\usepackage{indentfirst}
\usepackage{soul}
\usepackage[normalem]{ulem}
\usepackage{enumitem}

\usepackage{tocbasic}
\DeclareTOCStyleEntry[
  beforeskip=.56em plus 1pt,
  pagenumberformat=\textbf
]{tocline}{section}

\allowdisplaybreaks

\def\red{\color{red}}

\def\rn{{\mathbb{R}^n}}

\def\nn{{\mathbb N}}
\def\zz{{\mathbb Z}}

\def\vi{\varphi}
\def\fz{\infty }

\def\lf{\left}
\def\r{\right}
\def\ls{\lesssim}
\def\noz{\nonumber}

\def\XXint#1#2#3{{\setbox0=\hbox{$#1{#2#3}{\int}$ }
		\vcenter{\hbox{$#2#3$ }}\kern-.6\wd0}}

\DeclareMathOperator{\diam}{diam}

\newtheorem{theorem}{Theorem}[section]
\newtheorem{lemma}[theorem]{Lemma}
\newtheorem{corollary}[theorem]{Corollary}
\newtheorem{proposition}[theorem]{Proposition}

\theoremstyle{definition}
\newtheorem{remark}[theorem]{Remark}
\newtheorem{definition}[theorem]{Definition}

\newtheorem{convention}[theorem]{Convention}
\renewcommand{\appendix}{\par
	\setcounter{section}{0}%
	\setcounter{subsection}{0}%
	\setcounter{subsubsection}{0}%
	\gdef\thesection{\@Alph\c@section}%
	\gdef\thesubsection{\@Alph\c@section.\@arabic\c@subsection}%
	\gdef\theHsection{\@Alph\c@section.}%
	\gdef\theHsubsection{\@Alph\c@section.\@arabic\c@subsection}%
	\csname appendixmore\endcsname
}

\numberwithin{equation}{section}

\def\mvint_#1{\mathchoice
	{\mathop{\vrule width 6pt height 3 pt depth -2.5pt
			\kern -9pt \intop}\limits_{\kern -3pt #1}}%
	{\mathop{\vrule width 5pt height 3 pt depth -2.6pt
			\kern -6pt \intop}\nolimits_{#1}}%
	{\mathop{\vrule width 5pt height 3 pt depth -2.6pt
			\kern -6pt \intop}\nolimits_{#1}}%
	{\mathop{\vrule width 5pt height 3 pt depth -2.6pt
			\kern -6pt \intop}\nolimits_{#1}}}

\title{\bf A Measure Characterization of Embedding and Extension Domains for Sobolev, Triebel--Lizorkin, and Besov Spaces on Spaces of Homogeneous Type
\footnotetext{\hspace{-0.35cm} 2020 {\it Mathematics Subject Classification}.
Primary 46E36, 46E35; Secondary  43A85, 42B35, 30L99.
\endgraf {\it Key words and phrases.}  quasi-metric measure space, space of homogeneous type,
doubling measure,
Sobolev space, Besov space, Triebel--Lizorkin space, extension, embedding, the measure density condition.
\endgraf This project is partially supported by the National
Key Research and Development Program of China (Grant No. 2020YFA0712900)
and the National Natural Science Foundation of China (Grant Nos. 11971058,
12071197, 12122102  and 11871100).}}
\author{Ryan Alvarado\footnote{Corresponding
author, E-mail: \texttt{rjalvarado@amherst.edu}/{\red February 14, 2022}/Final version.},\ \
Dachun Yang and Wen Yuan}
\date{}

\begin{document}
	
\maketitle
\vspace{-0.8cm}
	
\begin{center}
\begin{minipage}{13cm}
{\small {\bf Abstract.}\quad In this article, for an optimal range of the smoothness parameter $s$ that depends (quantitatively) on the geometric makeup of the underlying space, the authors identify purely measure theoretic conditions that fully characterize embedding and extension domains for the scale of Haj{\l}asz--Triebel--Lizorkin spaces $M^s_{p,q}$ and Haj{\l}asz--Besov spaces $N^s_{p,q}$ in  general spaces of homogeneous type. Although stated in the context of quasi-metric spaces, these characterizations improve related work even in the metric setting. In particular, as a corollary of the main results in this article, the authors obtain a new characterization for Sobolev embedding and extension domains in the context of general doubling metric measure spaces.}
\end{minipage}
\end{center}

\tableofcontents	
\vspace{0.5cm}

\section{Introduction}

The extendability of a given class of functions satisfying
certain regularity properties (for instance, functions belonging to Sobolev, Besov or Triebel--Lizorkin
spaces)
from a subset of an ambient to the entire space,
while retaining the regularity properties in question,
is a long-standing topic in analysis that has played a fundamental role in both theoretical and applied branches of mathematics.
The literature on this topic is vast; however, we refer the reader to, for instance,  \cite{j81,T89,sd93,k98,r99,r00,S06,hkt08,hajlaszkt2,Z15,HIT16} and  the references
therein, for a series of  studies on extension problems for function spaces including
Sobolev, Besov, and Triebel--Lizorkin spaces.
In this article, we investigate the relationship between the extension (as well as embedding) properties of these functions spaces and the regularity of the underlying measure in the most general geometric and measure theoretic context that these questions still have a meaningful formulation.

It is known that the extension  and embedding properties of  Sobolev, Besov, and Triebel--Lizorkin
spaces are closely related to the so-called measure density condition.
To recall this condition,  let $(X, \rho,\mu)$ be a \emph{quasi-metric measure space},
namely, $X$ is a set of cardinality $\geq2$, $\rho$  a quasi-metric on $X$, and
$\mu$  a Borel measure such that,
  for any $x\in X$ and $r\in(0,\fz)$,  the {\it $\rho$-ball}
$B_\rho(x,r):=\{y\in X:\ \rho(x,y)<r\}$
is $\mu$-measurable and $\mu(B_\rho(x,r))\in(0,\fz)$.
Then, a $\mu$-measurable set $\Omega\subset X$ is said to satisfy
the \emph{measure density condition} if there exists a positive constant $C_\mu$ such that
\begin{equation}
\label{measdens-INT-a}
\mu(B_\rho(x,r))\leq C_\mu\,\mu(B_\rho(x,r)\cap\Omega)
\quad\mbox{for any $x\in\Omega$ and $r\in(0,1]$.}
\end{equation}
The upper bound  1 for the radii in \eqref{measdens-INT-a} is not essential and can be replaced by any strictly positive and finite threshold. From a geometric perspective, the measure density condition implies, among other things, that $\Omega$ is somewhat ``thick" nearby its boundary, and it is easy to see from definitions that the category of domains satisfying \eqref{measdens-INT-a} encompasses not only the classes of Lipschitz and $(\varepsilon,\delta)$ domains, but also fractal sets such as the `fat' Cantor and Sierpi\'nski carpet~sets \cite{jw84,S07}.

In the Euclidean setting, that is, $X=\mathbb{R}^n$ equipped with the Euclidean distance and
the Lebesgue measure,
Haj\l asz et al. \cite{hkt08}
proved that if $\Omega\subset \rn$ is an extension domain of  the classical
Sobolev space $W^{k,p}$ for some positive integer $k$ and $p\in[1,\fz)$, that is, if every
function in $W^{k,p}(\Omega)$ can be extended to a function in $W^{k,p}(\rn)$, then $\Omega$ necessarily  satisfies the measure density condition (see also \cite{k90}). Although there are domains satisfying \eqref{measdens-INT-a} that do not have the extension property (for instance, the `slit disk' example is $\mathbb{R}^2$), it was shown in \cite[Theorem~5]{hkt08} that \eqref{measdens-INT-a} together with the demand that the space $W^{k,p}(\Omega)$ admits a characterization via
some sharp maximal functions in spirit of Calder\'on \cite{c72}, is enough to characterize all $W^{k,p}$-extension domains when $p\in(1,\fz)$. In this case, there actually exists a bounded and linear operator $$\mathscr{E}: W^{k,p}(\Omega) \to W^{k,p}(\rn)$$
satisfying $\mathscr{E}u|_\Omega=u$ for any $u\in W^{k,p}(\Omega)$, which is a variant of the Whitney--Jones extension operator; see \cite{S06}.
Zhou \cite[Theorem~1.1]{Z15} considered the fractional
Slobodeckij--Sobolev space $W^{s,p}$ with $s\in(0,1)$ and $p\in(0,\fz)$ (which is known to coincide with the classical Besov and Triebel--Lizorkin spaces $B^s_{p,p}$ and $F^s_{p,p}$), and proved
that a domain $\Omega\subset \rn$ with $n\ge 2$ is a $W^{s,p}$-extension domain
if and only if $\Omega$ satisfies the  measure density condition. Going further, Zhou \cite[Theorem~1.2]{Z15} showed that the measure density condition is also equivalent to the existence of certain Sobolev-type embeddings for the space $W^{s,p}$ on $\Omega$. Remarkably, no further conditions beyond \eqref{measdens-INT-a} are needed to characterize $W^{s,p}$-embedding and extension domains; however, the extension operator  in \cite[Theorem~1.1]{Z15} is a modified Whitney-type operator based on \textit{median values} that may fail to be linear if $p<1$. We shall return to this point below.

Characterizations of extension domains  for certain classes of function spaces have been further studied in much more general  settings than the Euclidean one, namely, in the setting of \textit{metric} measure spaces $(X,\rho,\mu)$, where the measure
$\mu$ satisfies  the following  {\it doubling condition}:  there exists a positive
constant  $C_{D}$ such that
\begin{equation}\label{doub}
\mu(2B)\leq C_{D}\,\mu(B)\quad\,\mbox{ for any $\rho$-ball\,\,$B\subset X$;}
\end{equation}
here and thereafter, for any $\lambda\in(0,\fz)$ and  $\rho$-ball $B$ in $X$,
$\lambda B$ denotes the ball with the same center as $B$ and $\lambda$-times its radius; see \cite{hajlaszkt2,S07,h2001,bs07}.
When $\rho$ is a quasi-metric and $\mu$ is doubling, the triplet $(X,\rho,\mu)$ is called a \textit{space of homogeneous type} in the sense of Coifman and Weiss \cite{CoWe71,CoWe77}, and a \emph{doubling metric measure space} is a space of homogeneous type equipped with a metric.

For doubling metric measure spaces $(X,\rho,\mu)$,  Haj\l asz et al.
\cite[Theorem~6]{hajlaszkt2} proved that,
if a set $\Omega\subset X$ satisfies the measure density condition
\eqref{measdens-INT-a} then, for any given $p\in[1,\infty)$, there exists a bounded and linear Whitney-type extension operator from $M^{1,p}(\Omega)$ into $M^{1,p}(X)$, where $M^{1,p}$ denotes the Haj\l asz-Sobolev space (introduced in \cite{hajlasz2}) defined via pointwise Lipschitz-type inequalities. Under the assumption that $s\in(0,1)$, the corresponding Whitney extension theorems for
the Haj\l asz--Triebel--Lizorkin space $M^s_{p,q}$ and the Haj\l asz--Besov space
$N^s_{p,q}$ were   obtained by Heikkinen et al. \cite{HIT16}. More specifically, they proved in \cite[Theorem~1.2]{HIT16} that, if  $(X,\rho,\mu)$ is a doubling metric measure space
and $\Omega\subset X$ is a $\mu$-measurable set satisfying \eqref{measdens-INT-a} then, for any $s\in(0,1)$, $p\in(0,\fz)$, and $q\in(0,\fz]$,
there exists a bounded extension operator from $M^s_{p,q}(\Omega)$ [resp., $N^s_{p,q}(\Omega)$] into $M^s_{p,q}(X)$ [resp., $N^s_{p,q}(X)$]. Similar to \cite[Theorem~1.1]{Z15}, the extension operator in \cite[Theorem~1.2]{HIT16} may fail to be linear for small values of $p$.
The definitions and some basic properties of these function spaces can be found in Section~\ref{s-func} below. In particular, it is known that $M^{1,p}$, $M^s_{p,q}$, and $N^s_{p,q}$ coincide with the classical Sobolev, Triebel--Lizorkin, and Besov spaces on $\mathbb{R}^n$, respectively, and that
${M}^1_{p,\fz}=M^{1,p}$ for any $p\in(0,\fz)$ in general metric spaces; see \cite{KYZ11}.

The first main result of this work is to
provide a uniform approach to
extending both \cite[Theorem~6]{hajlaszkt2} and
\cite[Theorem~1.2]{HIT16} simultaneously  from metric spaces
to general \emph{quasi-metric} spaces, without compromising the quantitative
aspects of the theory pertaining to the range of the smoothness parameter $s$.
To better elucidate this latter point, we remark here that the upper bound  1
for the range of the smoothness parameter $s$ of the function spaces considered
in \cite{HIT16} is related to the fact
that one always has nonconstant Lipschitz functions (that is, H\"older continuous
functions of order 1)  in metric spaces. However, in a general quasi-metric space,
there is no guarantee that nonconstant Lipschitz functions exist. From this perspective,
it is crucial for the work being undertaken here to clarify precisely what should replace
the number 1 in the range for $s$ in the more general setting of quasi-metric spaces. As it turns out, the answer
is rather subtle and is, roughly speaking, related to the nature of the `best' quasi-metric on
$X$ which is bi-Lipschitz equivalent to $\rho$. More specifically, the upper bound on the
range of $s\in(0,1)$ that was considered in \cite{HIT16} for metric spaces should be
replaced by the following `index':
\begin{align}
\label{index-INT}
{\rm ind}\,(X,\rho):=&\sup_{\varrho\approx\rho}\left(\log_2C_\varrho\right)^{-1}
:=\sup_{\varrho\approx\rho}\left\{\log_2\left[\sup_{\substack{x,\,y,\,z\in X\\\mbox{\scriptsize{not all equal}}}}
\frac{\varrho(x,y)}{\max\{\varrho(x,z),\varrho(z,y)\}}\right]\right\}^{-1}
\in\,(0,\infty],
\end{align}
where   the first supremum is taken over all
the quasi-metrics $\varrho$ on $X$ which are  bi-Lipschitz equivalent to $\rho$,
and the second supremum is taken over all the points $x,\,y,\,z$ in $X$ which are not all
equal; see  \eqref{C-RHO.111} below for a more formal definition of $C_\varrho\in[1,\infty)$.
This index was introduced in \cite{MMMM13}, and its value encodes information about the
geometry of the underlying space, as evidenced by the following examples
(see  Section~\ref{section:preliminaries} for additional examples highlighting this fact):
\begin{itemize}[itemsep=1pt]
\item {${\rm ind}\,(\mathbb{R}^n,|\cdot-\cdot|)=1$ and  ${\rm ind}\,([0,1]^n,|\cdot-\cdot|)=1$,
where $|\cdot-\cdot|$ denotes the Euclidean distance;}

\item {${\rm ind}\,(\mathbb{R}^n,|\cdot-\cdot|^{1/\varepsilon})=\varepsilon$ and ${\rm ind}\,([0,1]^n,|\cdot-\cdot|^{1/\varepsilon})=\varepsilon$\, for any $\varepsilon\in(0,\infty)$;}

\item
{${\rm ind}\,(X,\rho)\geq 1$ if there exists a genuine distance on $X$ which is
pointwise equivalent to $\rho$;}

\item {$(X,\rho)$ cannot be bi-Lipschitzly embedded into some ${\mathbb{R}}^n$ with
$n\in{\mathbb{N}}$, whenever ${\rm ind}\,(X,\rho)<1$;}

\item  {${\rm ind}\,(X,\rho)=1$ if $(X,\rho)$ is a metric space that is equipped with a doubling measure  and supports a weak $(1, p)$-Poincar\'e inequality with $p>1$.}

\item {${\rm ind}\,(X,\rho)=\infty$ if there exists an ultrametric \footnote{Recall
that a metric $\rho$ on the set $X$ is called an \textit{ultrametric} provided that, in place of the triangle-inequality, $\rho$ satisfies the stronger condition $\rho(x,y)\leq\max\{\rho(x,z),\rho(z,y)\}$ for any $x,\,y,\,z\in X.$}
on $X$ which is pointwise equivalent to $\rho$.}
\end{itemize}
From these examples, we can deduce that the range of $s\in(0,1)$  is not optimal, even in the metric setting, as the number 1 fails to fully reflect the geometry of the underlying space.

To facilitate the statement of the main theorems in this work, we make the following notational convention: Given a quasi-metric space $(X,\rho)$ and  fixed numbers $s\in(0,\infty)$ and $q\in(0,\infty]$, we will understand by $s\preceq_q{\rm ind}\,(X,\rho)$ that $s\leq{\rm ind}\,(X,\rho)$ and that the value $s={\rm ind}\,(X,\rho)$ is only permissible when $q=\infty$ and the supremum defining the number ${\rm ind}\,(X,\rho)$ in \eqref{index-INT} is attained. Also, recall that  a measure $\mu$ is doubling (in the sense of \eqref{doub}) if and
only if $\mu$ satisfies the following $Q$-{\it doubling} property:
there exist positive constants $Q$ and $\kappa$ such that
\begin{equation}\label{Doub-2}
\kappa\left(\frac{r}{R}\right)^{Q}\leq\frac{\mu(B_\rho(x,r))}{\mu(B_\rho(y,R))},
\end{equation}
whenever $x,\,y\in X$ and $0<r\leq R<\infty$ satisfy  $B_\rho(x,r)\subset B_\rho(y,R)$.
Although \eqref{doub} and \eqref{Doub-2} are equivalent, the advantage of \eqref{Doub-2} over \eqref{doub}, from the perspective of the work in this article, is that the exponent $Q$ in \eqref{Doub-2} provides some notion of dimension for the space $X$. With this convention and piece of terminology in hand, the first main result of this article which simultaneously generalizes \cite[Theorem~6]{hajlaszkt2} and  \cite[Theorem~1.2]{HIT16}  from metric spaces to general quasi-metric spaces for an
optimal range of $s$ is as follows (see also Theorem~\ref{measext} below).

\begin{theorem}
\label{measext-INT}
Let $(X,\rho,\mu)$ be a space of homogeneous type,
where $\mu$ is a Borel regular \footnote{In the sense that every $\mu$-measurable
set is contained in a Borel set of equal measure.}
measure satisfying \eqref{Doub-2} for some $Q\in(0,\infty)$, and  fix exponents $s,\,p\in(0,\infty)$ and $q\in(0,\infty]$ such that
$s\preceq_q{\rm ind}\,(X,\rho).$
Also, suppose that $\Omega\subset X$ is a nonempty $\mu$-measurable set that satisfies the measure density condition \eqref{measdens-INT-a}.
Then any function belonging to $M^s_{p,q}(\Omega,\rho,\mu)$
can be extended to the entire space $X$ with preservation of smoothness
 while retaining control of the associated `norm'.
More precisely, there exists a positive constant $C$ such that,
\begin{eqnarray*}
\begin{array}{c}
\mbox{for any }\,u\in M^s_{p,q}(\Omega,\rho,\mu),
\,\,\mbox{ there exists a }\ \widetilde{u}\in M^s_{p,q}(X,\rho,\mu)
\\[6pt]
\mbox{ satisfying }\,\,\, u=\widetilde{u}|_{\Omega}\,\,
\mbox{ and }\,\,\,\,\|\widetilde{u}\|_{M^s_{p,q}(X,\rho,\mu)}
\leq C\|u\|_{M^s_{p,q}(\Omega,\rho,\mu)}.
\end{array}
\end{eqnarray*}
Consequently,
\begin{equation*}
M^s_{p,q}(\Omega,\rho,\mu)
=\left\{\widetilde{u}|_{\Omega}:\,\widetilde{u}\in M^s_{p,q}(X,\rho,\mu)\right\}.
\end{equation*}
Furthermore, if $p,\,q>Q/(Q+s)$, then there exists a  bounded linear  operator
$$\mathscr{E}\colon M^s_{p,q}(\Omega,\rho,\mu)\to M^s_{p,q}(X,\rho,\mu)$$
such that $\big(\mathscr{E}u\big)|_{\Omega}=u$ for any $u\in M^s_{p,q}(\Omega,\rho,\mu)$.
In addition, if $s<{\rm ind}\,(X,\rho)$, then all  of the statements above remain valid with $M^s_{p,q}$ replaced by $N^s_{p,q}$.
\end{theorem}

The distinguishing feature of Theorem~\ref{measext-INT} is the range of $s$ for which the conclusions of this result hold true, because it is the largest range of this type to be identified and it turns out to be in the nature of best possible.
More specifically, if the underlying space is $\mathbb{R}^n$
equipped with the Euclidean distance, then ${\rm ind}\,(\mathbb{R}^n, |\,\cdot-\cdot\,|)=1$
and Theorem~\ref{measext-INT} is valid for the Haj\l asz--Sobolev space
${M}^1_{p,\infty}={M}^{1,p}$, as well as the Haj\l asz--Triebel--Lizorkin space ${M}^s_{p,q}$ and
the Haj\l asz--Besov space ${N}^s_{p,q}$  whenever $s\in(0,1)$.
Therefore, we recover the expected range for $s$ in the Euclidean setting. Similar considerations also hold whenever the underlying space is a set $S\subset\rn$ or, more generally, a metric measure space. Consequently, the trace spaces $W^{1,p}(\rn)|_\Omega=M^{1,p}(\rn)|_\Omega$, $F^s_{p,q}(\rn)|_\Omega=M^s_{p,q}(\rn)|_\Omega$, and $B^s_{p,q}(\rn)|_\Omega=N^s_{p,q}(\rn)|_\Omega$ can be identified with the pointwise spaces $M^{1,p}(\Omega)$, $M^s_{p,q}(\Omega)$, and $N^s_{p,q}(\Omega)$, respectively, for $s\in(0,1)$, where $\Omega\subset\rn$ is as in Theorem~\ref{measext-INT}, and $F^s_{p,q}$ and $B^s_{p,q}$ denote the classical Triebel--Lizorkin and Besov spaces in $\rn$. Remarkably, there are environments where the range of $s$ is strictly larger than what it would be in the Euclidean setting. For example, if the underlying  space $(X,\rho)$ is an ultrametric space (like a Cantor-type set), then ${\rm ind}\,(X,\rho)=\infty$ and, in this case,  Theorem~\ref{measext-INT} is valid for the  spaces  ${M}^s_{p,q}$ and ${N}^s_{p,q}$ for {\it all} $s\in(0,\infty)$. To provide yet another example, for the metric space $(\mathbb{R},|\cdot-\cdot|^{1/2})$, one has that ${\rm ind}\,(\mathbb{R},|\cdot-\cdot|^{1/2})=2$ and hence, Theorem~\ref{measext-INT} is
valid for the spaces ${M}^s_{p,q}$ and ${N}^s_{p,q}$ for any $s\in(0,2)$. As these examples illustrate, Theorem~\ref{measext-INT} not only extends to full generality the related work in \cite[Theorem~1.2]{HIT16} and \cite[Theorem~6]{hajlaszkt2}, but also sharpens their work even in the metric setting by identifying an optimal range of the smoothness parameter $s$ for which these extension results hold true. In particular, the existence of an $M^{1,p}$-extension operator for  $p<1$ as given by  Theorem~\ref{measext-INT} is new in the context of doubling metric measure spaces.

Loosely speaking, the extension operator in Theorem~\ref{measext-INT} is of Whitney-type
in   that it is constructed via
gluing various `averages' of a function together with using
a partition of unity that is sufficiently smooth (relative to the parameter $s$).
A key issue in this regard is that functions in $M^s_{p,q}$ and $N^s_{p,q}$
may not be locally integrable for small values of $p$, which, in turn,
implies that the usual integral average of such functions will not be well defined.
It turns out that one suitable replacement for the integral average
in this case is the so-called median value of a function
(see, for instance, \cite{F91}, \cite{Z15}, \cite{HIT16}, and also Definition~\ref{median} below). Although the Whitney extension operator based on median values has the distinct attribute that it extends functions in  $M^s_{p,q}(\Omega)$ and $N^s_{p,q}(\Omega)$ for any given $p\in(0,\infty)$, there is no guarantee that such an operator is linear. However, when $p$ is not too small, functions in $M^s_{p,q}$ and $N^s_{p,q}$ are locally integrable and we can consider a \textit{linear} Whitney-type extension operator based on integral averages.

While our approach for dealing with Theorem~\ref{measext-INT} is
related to the work in \cite{hajlaszkt2} and \cite{HIT16},
where metric spaces have been considered, the geometry of
quasi-metric spaces can be significantly more intricate,
which brings a number of obstacles. For example, as mentioned above, the extension
operators constructed in \cite{hajlaszkt2} and \cite{HIT16} rely on
the existence of a Lipschitz partition of unity;
however, depending on the quasi-metric space, nonconstant Lipschitz functions may not exist.
To cope with this fact,
we employ a partition of unity consisting of
H\"older continuous functions developed in \cite[Theorem 6.3]{AMM13} (see also \cite[Theorem~2.5]{AM15}) that exhibit a maximal amount of
smoothness (measured on the H\"older scale) that the geometry of
the underlying space can permit.

It was also shown in \cite[Theorem~1.3]{HIT16} that the measure density condition fully characterizes the $M^s_{p,q}$ and $N^s_{p,q}$ extension domains for any given $s\in(0,1)$, provided that metric measure space is \textit{geodesic} (that is, the metric
space has the property that any two points in it can be  joined by a curve whose length equals
the distance between these two points) and the measure satisfies the following $Q$-\textit{Ahlfors regularity} condition on $X$:
$$
\kappa_1r^Q\leq\mu(B_\rho(x,r))\leq\kappa_2r^Q\quad\mbox{for any $x\in X$ and finite $r\in(0,{\rm diam}_\rho(X)]$,}
$$
where $Q,\kappa_1,\kappa_2\in(0,\infty)$ and ${\rm diam}_\rho(X):=\sup\{\rho(x,y): x,\,y\in X\}$. Clearly, every $Q$-Ahlfors regular measure is $Q$-doubling; however, there are non-Ahlfors regular doubling measures. The corresponding result for Haj\l asz--Sobolev spaces can be found in \cite[Theorem~5]{hajlaszkt2}.
The geodesic assumption is a very strong connectivity condition that precludes many settings in which the spaces $M^{1,p}$, $M^s_{p,q}$, and $N^s_{p,q}$ have a rich theory, such as  Cantor-type sets and the fractal-like metric space $(\mathbb{R},|\cdot-\cdot|^{1/2})$. In fact, one important virtue that the Haj\l asz--Sobolev, Haj\l asz--Besov, and Haj\l asz--Triebel--Lizorkin spaces possess over other definitions of these spaces in the quasi-metric setting, is that there is a fruitful theory for this particular brand of function spaces without needing to assume that the space is geodesic (or even connected) or that the measure is Ahlfors regular.

In this article, we sharpen the results in \cite[Theorem~1.3]{HIT16} and \cite[Theorem~5]{hajlaszkt2}
by eliminating the $Q$-Ahlfors regularity and the geodesic assumptions   on the underlying spaces,
without compromising the main conclusions of these results and, in particular, the quantitative aspects of the theory pertaining to the optimality of the range
for the smoothness parameter $s$. Moreover, in the spirit of Zhou \cite{Z15}, we go further and show that the measure density condition also fully characterizes the existence of certain Sobolev-type embeddings for the spaces $M^s_{p,q}$ and $N^s_{p,q}$ on domains, which is a brand new result even in the Euclidean setting.
For the clarity of exposition in this introduction, we state in the
following theorem a simplified summary of these remaining principal results which,
to some degree, represents the central theorem in this article.
The reader is directed to Theorems~\ref{measdens-ext-sob} and \ref{measdens-ext-sob-besov} below
for stronger and more informative formulations of this result.

\begin{theorem}
\label{measdens-ext-sob-INT}
Let $(X,\rho,\mu)$ be a space of homogeneous type, where $\mu$ is a Borel regular measure satisfying \eqref{Doub-2} for some $Q\in(0,\infty)$, and suppose that $\Omega\subset X$ is a $\mu$-measurable locally uniformly perfect set in the sense of \eqref{U-perf}. Then  the following statements are equivalent.
\begin{enumerate}[label=\rm{(\alph*)}]
\item $\Omega$ satisfies the  measure density condition \eqref{measdens-INT-a}.

\item $\Omega$ is an $M^s_{p,q}$-extension domain for some (or all)  $p\in(0,\infty)$, $q\in(0,\infty]$, and $s\in(0,\infty)$ satisfying $s\preceq_q{\rm ind}\,(X,\rho)$.

\item $\Omega$ is an $N^s_{p,q}$-extension domain for some (or all)  $p\in(0,\infty)$, $q\in(0,\infty]$, and $s\in(0,\infty)$ satisfying $s<{\rm ind}\,(X,\rho)$.
		
\item $\Omega$ is a local $\dot{M}^s_{p,q}$-embedding domain for some (or all)  $p\in(0,\infty)$, $q\in(0,\infty]$, and $s\in(0,\infty)$ satisfying $s\preceq_q{\rm ind}\,(\Omega,\rho)$.

\item $\Omega$ is a local $\dot{N}^s_{p,q}$-embedding domain for some (or all) $p\in(0,\infty)$, $q\in(0,p]$, and $s\in(0,\infty)$ satisfying $s\preceq_q{\rm ind}\,(\Omega,\rho)$.
\end{enumerate}
If the measure $\mu$  is actually $Q$-Ahlfors regular on $X$,  then the following statements are also equivalent to each of
{\rm (a)}-{\rm (e)}.
\begin{enumerate}[label={\rm(\alph*)}]\addtocounter{enumi}{5}
\item $\Omega$ is a global $M^s_{p,q}$-embedding domain for some (or all)  $p\in(0,\infty)$, $q\in(0,\infty]$, and $s\in~\!\!(0,\infty)$ satisfying $s\preceq_q{\rm ind}\,(\Omega,\rho)$.

\item $\Omega$ is a global $N^s_{p,q}$-embedding domain for some (or all)  $p\in(0,\infty)$, $q\in(0,p]$, and $s\in(0,\infty)$ satisfying $s\preceq_q{\rm ind}\,(\Omega,\rho)$.
\end{enumerate}
\end{theorem}

The additional demand that $\Omega$ is a locally uniformly perfect set is only used in proving that each of the statements in (b)-(g) imply the measure density condition in (a).
Note that this assumption on $\Omega$ is optimal. Indeed, when $\mu$ is $Q$-Ahlfors regular on $X$, the measure density condition in Theorem~\ref{measdens-ext-sob-INT}(a) implies that $\Omega$ is a locally Ahlfors regular \textit{set}, and it is known that such sets are necessarily locally uniformly perfect; see, for instance \cite[Lemma~4.7]{DaSe97}.

Regarding how Theorem~\ref{measdens-ext-sob-INT} fits into the existing literature, since the metric space $(\mathbb{R},|\cdot-\cdot|^{1/2})$ is not geodesic, one cannot appeal to \cite{hajlaszkt2} and \cite{HIT16} in order to conclude that the measure density condition
holds true for ${M}^{1,p}$, ${M}^s_{p,q}$ and ${N}^s_{p,q}$-extension domains
as their results are not applicable in such settings. However, we have ${\rm ind}\,(\mathbb{R},|\cdot-\cdot|^{1/2})=2$ and hence, Theorem~\ref{measdens-ext-sob-INT} in this work is valid for the spaces ${M}^s_{p,q}$ and ${N}^s_{p,q}$ for any $s\in(0,2)$ and, in particular, for the space $M^{1,p}$. In this vein, we also wish to mention that for $q=\infty$, Theorem~\ref{measdens-ext-sob-INT} is
valid for any $s\in(0,\infty)$ satisfying $s\leq(\log_2C_\rho)^{-1}$,
where $C_\rho\in[1,\infty)$ is as in \eqref{C-RHO.111}.
On the other hand, it follows immediately from \eqref{C-RHO.111} that if $\rho$
is a genuine \textit{metric} then $C_\rho\leq2$ and so,  $(\log_2C_\rho)^{-1}\geq 1$.
Therefore, by combining this observation with the fact
that $M^1_{p,\infty}=M^{1,p}$ (see \cite{AYY21}),
we have the following consequence of Theorem~\ref{measdens-ext-sob-INT} (or, more specifically, Theorem~\ref{measdens-ext-sob}), which
is a brand new result for Haj\l asz--Sobolev spaces in the setting of
metric spaces and which generalizes and extends the work in \cite{hajlaszkt2} and
\cite{Karak1}.

\begin{corollary}
\label{measdens-ext-sob-Cor}
Let $(X,\rho,\mu)$ be a doubling metric measure space,
where $\mu$ is a Borel regular  measure satisfying \eqref{Doub-2}
for some $Q\in(0,\infty)$, and suppose that
$\Omega\subset X$ is a $\mu$-measurable
locally uniformly perfect set in the sense of \eqref{U-perf}. Then  the following statements are equivalent.
\begin{enumerate}[label={\rm (\alph*)}]
\item $\Omega$ satisfies the  measure density condition \eqref{measdens-INT-a}.
		
\item $\Omega$ is an $M^{1,p}$-extension domain for some (or all)
$p\in(0,\infty)$ in the sense that there exists a positive constant $C $ such that,
\begin{eqnarray*}
\begin{array}{c}
\mbox{for any }\,u\in M^{1,p}(\Omega,\rho,\mu),
\,\,\mbox{ there exists a\ }\ \widetilde{u}\in M^{1,p}(X,\rho,\mu)
\\[6pt]
\mbox{satisfying}\,\,\, u=\widetilde{u}|_{\Omega}\,\,
\mbox{ and }\,\,\,\,\|\widetilde{u}\|_{M^{1,p}(X,\rho,\mu)}
\leq C\|u\|_{M^{1,p}(\Omega,\rho,\mu)}.
\end{array}
\end{eqnarray*}

\item For some (or all) $p\in\big(Q/(Q+1),\infty\big)$, there exists a linear and bounded extension operator $\mathscr{E}\colon M^{1,p}(\Omega,\rho,\mu)\to M^{1,p}(X,\rho,\mu)$
    such that $(\mathscr{E}u)|_{\Omega}=u$ for any $u\in M^{1,p}(\Omega,\rho,\mu)$.

\item $\Omega$ is a local $\dot{M}^{1,p}$-embedding domain for some (or all)  $p\in(0,\infty)$.
\end{enumerate}
If the measure $\mu$ is actually $Q$-Ahlfors regular on $X$,  then the following statement is also equivalent to
each of {\rm(a)}-{\rm(d)}.
\begin{enumerate}[label={\rm(\alph*)}]\addtocounter{enumi}{4}
\item  $\Omega$ is a global ${M}^{1,p}$-embedding domain for some (or all) $p\in(0,\infty)$.
\end{enumerate}
\end{corollary}

The remainder of this article is organized as follows.
In Section~\ref{section:preliminaries}, we review some basic
terminology and results pertaining to quasi-metric
spaces and the main classes of function spaces considered in this work,
including the fractional Haj\l asz--Sobolev spaces, the Haj\l asz--Triebel--Lizorkin spaces,
and the Haj\l asz--Besov  spaces.
In particular, we present some optimal embedding results for these spaces
recently established in \cite{AYY21}, which extends the work of \cite{agh20} and \cite{Karak2}.

The main aim of Section~\ref{section:extensions} is to
 prove Theorem~\ref{measext-INT}. This is done in
 Subsection~\ref{sssec:extensions2}, after we collect a number
 of necessary and key tools in Subsection~\ref{sssec:extensions1},
including a Whitney type decomposition of the underlying space
(see Theorem \ref{L-WHIT} below) and the related
partition of unity  (see Theorem \ref{MSz7b} below), as well as some
Poincar\'e-type inequalities in terms of both
the ball integral averages and the  median values of functions (see Lemmas \ref{embeddfracgrad}
and \ref{intavgest} below).
The  regularity parameter of the aforementioned partition of unity
is closely linked to the geometry of the underlying  quasi-metric space,
and is important  in constructing the desired extension of functions from
the spaces $M^s_{p,q}(\Omega,\rho,\mu)$ and $N^s_{p,q}(\Omega,\rho,\mu)$
for an optimal range of $s$. On the other hand,
as we consider the extension theorems of
$M^s_{p,q}(\Omega,\rho,\mu)$ and $N^s_{p,q}(\Omega,\rho,\mu)$ for full ranges of parameters,
whose elements might not  be locally integrable when $p$ or $q$ is small,
we need to use the  median values instead of the usual ball
integral averages. Via these tools, in  Subsection~\ref{sssec:extensions2},
we first successfully construct a local extension for
functions in $M^s_{p,q}(\Omega,\rho,\mu)$
[and also $N^s_{p,q}(\Omega,\rho,\mu)$] from $\Omega$ to a neighborhood $V$ of $\Omega$,
which, together with
the boundedness of  bounded H\"older continuous functions with support $V$,
operating as pointwise multipliers,
on these spaces (see Lemma~\ref{GVa2-prev} below),
further leads to the desired global extension to
the entire space $X$.

Finally, in Section~\ref{section:measuredensity}, we formulate
 and prove Theorem~\ref{measdens-ext-sob-INT} (see Theorems~\ref{measdens-ext-sob}
 and  \ref{measdens-ext-sob-besov} below), that is, to show that,
on spaces of homogeneous type and, in particular, on Ahlfors regular spaces,
the existence of an extension for functions in $M^s_{p,q}(\Omega,\rho,\mu)$
and $N^s_{p,q}(\Omega,\rho,\mu)$, from
a $\mu$-measurable locally uniformly perfect set $\Omega\subset X$
to the entire space $X$, is equivalent to the measure density condition.
The key tools are the embedding
results of these spaces recently obtained in \cite{AYY21} (see also Theorems \ref{DOUBembedding}
and \ref{mainembedding-epsilon} below) and
a family of maximally smooth H\"older continuous
`bump' functions belonging to $M^{s}_{p,q}$ and $N^{s}_{p,q}$
constructed in \cite[Lemma 4.6]{AYY21} (see also Lemma \ref{HolderBump} below).
Indeed, as we can see in Theorems~\ref{measdens-ext-sob-INT}, \ref{measdens-ext-sob},
and  \ref{measdens-ext-sob-besov},
we not only obtain the equivalence between the extension properties
for the spaces $M^s_{p,q}(\Omega,\rho,\mu)$
and $N^s_{p,q}(\Omega,\rho,\mu)$, and the measure density condition, but also
their equivalence to various Sobolev-type embeddings of these spaces.

In closing, we emphasize again that
the main results in this article
are not just generalizations of \cite[Theorems~5 and 6]{hajlaszkt2} and
\cite[Theorems~1.2 and 1.3]{HIT16} from metric to
quasi-metric spaces, but they also sharpen these
known results, even in the metric setting,
via seeking an optimal range of the smoothness parameter which is essentially
determined (quantitatively) by the geometry of the underlying space.

\section{Preliminaries}
\label{section:preliminaries}

This section is devoted to presenting some basic assumptions on the underlying quasi-metric measure space,
as well as the definitions and some facts of function spaces considered in this work.

\subsection{The Setting}

Given a nonempty set $X$, a function
$\rho\colon X\times X\to[0,\infty)$ is called a
{\it quasi}-{\it metric} on $X$, provided there exist
two positive constants $C_0$ and $C_1$ such  that, for any $x$, $y$, $z\in X$, one has
\begin{eqnarray}\label{gabn-T.2}
\begin{array}{c}
\rho(x,y)=0\Longleftrightarrow x=y,\quad
\rho(y,x)\leq C_0\rho(x,y),
\\[6pt]
\mbox{and }\quad
\rho(x,y)\leq C_1\max\{\rho(x,z),\rho(z,y)\}.
\end{array}
\end{eqnarray}
A pair $(X,\rho)$ is called a {\it quasi-metric} {\it space} if
$X$ is a nonempty set and $\rho$  a quasi-metric on $X$. Throughout
the whole article,
we tacitly assume that $X$ is of cardinality $\geq2$.
It follows from this assumption  that the
constants $C_0$ and $C_1$ appearing in \eqref{gabn-T.2} are $\geq1$.

Define $C_\rho$ to be the least constant which can play the
role of $C_1$ in \eqref{gabn-T.2}, namely,
\begin{eqnarray}\label{C-RHO.111}
C_\rho:=\sup_{\substack{x,\,y,\,z\in X\\\mbox{\scriptsize{{not all equal}}}}}
\frac{\rho(x,y)}{\max\{\rho(x,z),\rho(z,y)\}}\in[1,\fz).
\end{eqnarray}
Also, let $\widetilde{C}_\rho$ to be the least constant
which can play the role of $C_0$ in \eqref{gabn-T.2}, that is,
\begin{eqnarray}\label{C-RHO.111XXX}
\widetilde{C}_\rho:=\sup_{x,\,y\in X,\ x\not=y}
\frac{\rho(y,x)}{\rho(x,y)}\in[1,\fz).
\end{eqnarray}
When $C_\rho=\widetilde{C}_\rho=1$, $\rho$ is a genuine metric that is commonly
called  an ultrametric. Note that, if the underlying metric space
is $\mathbb{R}^n$, $n\in\mathbb{N}$, equipped with the Euclidean distance $d$, then $C_d=2$.

Two quasi-metrics $\rho$ and $\varrho$ on $X$ are said to be {\it equivalent},
denoted by $\rho\approx\varrho$, if there exists a positive constant $c$ such that
$$
c^{-1}\varrho(x,y)\leq\rho(x,y)\leq c\varrho(x,y),\quad\forall\ x,\, y\in X.
$$
It follows from \cite[(4.289)]{MMMM13} that $\rho\approx\varrho$ if and only if the
 identity $(X,\rho)\to (X,\varrho)$  is bi-Lipschitz (see \cite[Definition 4.32)]{MMMM13}).
As such, also refer to  $\rho\approx\varrho$ as $\rho$
 is \emph{bi-Lipschitz equivalent} to $\varrho$.

We now recall the concept of an `index' which was originally
introduced in \cite[Definition 4.26]{MMMM13}.
The \textit{lower smoothness index} of a quasi-metric space $(X,\rho)$ is defined as
\begin{equation}
\label{index}
{\rm ind}\,(X,\rho):=\sup_{\varrho\approx\rho}\left(\log_2C_\varrho\right)^{-1}\in(0,\infty],
\end{equation}
where the supremum is taken over all the quasi-metrics $\varrho$ on $X$ which are
bi-Lipschitz equivalent to $\rho$.
The index ${\rm ind}\,(X,\rho)$
encodes  information about the geometry of
the underlying space, as evidenced by the properties
listed in the introduction, as well as   the following additional ones:
\begin{itemize}[itemsep=1pt]

\item \mbox{${\rm ind}\,(Y,\rho)\geq{\rm ind}\,(X,\rho)$ whenever $Y \subset X$;}

\item {${\rm ind}\,(X,\|\cdot-\cdot\|)=1$ if $(X,\|\cdot\|)$ is a nontrivial
normed vector space; hence ${\rm ind}\,(\mathbb{R}^n,|\cdot-\cdot|)=1$;}

\item {${\rm ind}\,(Y,\|\cdot-\cdot\|)=1$ if $Y$ is a subset
of a normed vector space $(X,\|\cdot\|)$ containing an open line segment; hence ${\rm ind}\,([0,1]^n,|\cdot-\cdot|)=1$;}

\item {${\rm ind}\,(X,\rho)\leq 1$ whenever the interval $[0,1]$ can be
bi-Lipschitzly embedded into $(X,\rho)$;}

\item {${\rm ind}\,(X,\rho)\leq Q$ if $(X,\tau_\rho)$ is pathwise connected
and $(X,\rho)$ is equipped with a $Q$-Ahlfors-regular measure, where $\tau_\rho$ denotes the topology generated by $\rho$;}

\item{there are compact, totally disconnected, Ahlfors regular spaces
with lower smoothness index $\infty$; for instance, the four-corner planar Cantor set equipped with $|\cdot-\cdot|$;}

\item {${\rm ind}\,(X,\rho)=\infty$ whenever the underlying set
$X$ has finite cardinality;}

\end{itemize}
see also \cite[Remark~4.6]{AYY21} and \cite[Section 4.7]{MMMM13} or \cite[Section 2.5]{AM15} for more details.

We conclude this subsection with some notational conventions.
Balls in $(X,\rho)$ will be denoted by $B_\rho(x,r):=\{y\in X:\, \rho(x,y)<r\}$  with $x\in X$ and
$r\in (0,\fz)$, and let $$\overline{B}(x,r):=\lf\{y\in X:\, \rho(x,y)\leq r\r\}.$$
As a sign of warning, note that in general $\overline{B}_\rho(x,r)$ is not necessarily equal to the closure of $B_\rho(x,r)$. If $r=0$, then $B_\rho(x,r)=\emptyset$, but $\overline{B}_\rho(x,r)=\{ x\}$. Moreover, since $\rho$ is not necessarily symmetric, one needs to pay particular attention to the order of $x$ and $y$ in the definition of $B_\rho(x,r)$.  The triplet $(X,\rho,\mu)$ is  called  a \textit{quasi-metric measure space} if $X$ is a set of cardinality $\geq2$, $\rho$ is a quasi-metric on $X$, and $\mu$ is a Borel measure such that all $\rho$-balls are $\mu$-measurable and $\mu(B_\rho(x,r))\in(0,\infty)$ for any $x\in X$ and any $r\in(0,\infty)$.
Here and thereafter, the measure $\mu$ is said to be Borel regular if every $\mu$-measurable
set is contained in a Borel set of equal measure.
See also  \cite{MMMM13,AMM13,AM15} for more information on this setting.

Let $\zz$ denote all integers and $\nn$ all (strictly) positive integers.
 We always denote by $C$ a \emph{positive constant}
which is independent of the main parameters, but it
may vary from line to line. We also use
$C_{(\alpha,\beta,\ldots)}$ to denote a positive
constant depending on the indicated parameters $\alpha,
\beta,\ldots.$ The \emph{symbol} $f\lesssim g$ means
that $f\le Cg$. If $f\lesssim g$ and $g\lesssim f$,
we then write $f\approx g$. If $f\le Cg$ and $g=h$ or
$g\le h$, we then write $f\ls g\approx h$ or $f\ls g\ls h$,
\emph{rather than} $f\ls g=h$ or $f\ls g\le h$. The integral average of a locally $\mu$-measurable function $u$ on a $\mu$-measurable set
$E\subset X$ with $\mu(E)\in(0,\fz)$ is denoted by
$$
u_E:=\mvint_Eu\, d\mu :=\frac{1}{\mu(E)}\int_E u\, d\mu,
$$
whenever the integral is well defined.	
For sets $E\subset (X,\rho)$, let ${\rm diam}_\rho(E):=\sup\{\rho(x,y):\, x,\,y\in E\}$ and ${\bf 1}_E$ be the characteristic
function of $E$.
For any $p\in(0,\fz]$, let $L^p(X):=L^p(X,\mu)$ denote the \emph{Lebesgue space} on $(X,\mu)$,
that is, the set of all the $\mu$-measurable functions $f$ on $(X,\mu)$ such that
$$
\|f\|_{L^p(X)}:=\|f\|_{L^p(X,\mu)}:=
\begin{cases}
\displaystyle\left[\int_X |f(x)|^p\,d\mu(x)\right]^{1/p},\  \  & p\in(0,\fz),\\
\displaystyle{\rm ess\,sup}\left\{|f(x)|:\ x\in X\right\},\ \ & p=\fz,
\end{cases}$$
is finite.

\subsection{Triebel--Lizorkin, Besov, and Sobolev Spaces}
\label{s-func}
Suppose $(X,\rho,\mu)$ is a quasi-metric space equipped
with a nonnegative Borel measure, and let $s\in(0,\infty)$.
Following \cite{KYZ11}, a sequence $\{g_k\}_{k\in\mathbb{Z}}$ of nonnegative $\mu$-measurable
functions on $X$ is called a \textit{fractional $s$-gradient} of
a $\mu$-measurable function $u\colon X\rightarrow\mathbb{R}$ if there exists
a set $E\subset X$ with $\mu(E)=0$ such that
\begin{equation}\label{Hajlasz}
\vert u(x)-u(y)\vert\leq [\rho(x,y)]^s \left[g_k(x)+g_k(y)\right]
\end{equation}
for any $k\in\mathbb{Z}$ and  $x,\,y\in X\setminus E$ satisfying $2^{-k-1}\leq \rho(x,y)<2^{-k}.$
Let $\mathbb{D}_\rho^s(u)$ denote the set of all the fractional $s$-gradients of $u$.

Given $p\in(0,\infty)$, $q\in(0,\infty]$, and a sequence  $\vec{g}:=\{g_k\}_{k\in\mathbb{Z}}$
of $\mu$-measurable functions on $X$, define
\begin{equation*}
\Vert \vec{g}\Vert_{L^p(X,\ell^q)}:=
\lf\Vert\, \Vert \{g_k\}_{k\in\mathbb{Z}}\Vert_{\ell^q} \r\Vert_{L^p(X,\mu)}
\end{equation*}
and
\begin{equation*}
\Vert \vec{g}\Vert_{\ell^q(L^p(X))}:=\lf\Vert \lf\{\Vert \{g_k\}\Vert_{L^p(X,\mu)}\r\}_{k\in\mathbb{Z}} \r\Vert_{\ell^q},
\end{equation*}
where
\begin{equation*}
\Vert \{g_k\}_{k\in\mathbb{Z}}\Vert_{\ell^q}:=
\begin{cases}
 \displaystyle\left(\sum_{k\in\mathbb{Z}}\vert g_k\vert^q\right)^{1/q}& ~\text{if}~q\in(0,\infty),\\
 \displaystyle \sup_{k\in\mathbb{Z}}\vert g_k\vert& ~\text{if}~q=\infty.
\end{cases}
\end{equation*}
The
\textit{homogeneous Haj\l asz--Triebel--Lizorkin space}
$\dot{M}^s_{p,q}(X,\rho,\mu)$ is defined as the
collection of all the $\mu$-measurable functions $u\colon  X\rightarrow\mathbb{R}$ such that
\begin{equation*}
\Vert u\Vert_{\dot{M}^s_{p,q}(X)}:=\Vert u\Vert_{\dot{M}^s_{p,q}(X,\rho,\mu)}:=\inf_{\vec{g}\in
\mathbb{D}_\rho^s(u)}\Vert\vec{g}\Vert_{L^p(X,\ell^q)}<\infty.
\end{equation*}
Here and thereafter, we make the agreement that $\inf\emptyset:=\infty$.
The corresponding \textit{inhomogeneous Haj\l asz--Triebel--Lizorkin space}
 is defined as
$M^s_{p,q}(X,\rho,\mu):=\dot{M}^s_{p,q}(X,\rho,\mu)\cap L^p(X),$ and it is
equipped with the (quasi-)norm
\begin{equation*}
\Vert u\Vert_{M^s_{p,q}(X)}:=\Vert u\Vert_{M^s_{p,q}(X,\rho,\mu)}:=\Vert u\Vert_{\dot{M}^s_{p,q}(X,\rho,\mu)}+\Vert u\Vert_{L^p(\Omega)},
\quad\forall\ u\in M^s_{p,q}(X,\rho,\mu).
\end{equation*}

The \textit{homogeneous Haj\l asz--Besov space} $\dot{N}^s_{p,q}(X,\rho,\mu)$
is defined as the collection of all the $\mu$-measurable functions
$u\colon  X\rightarrow\mathbb{R}$ such that
\begin{equation*}
\Vert u\Vert_{\dot{N}^s_{p,q}(X,\rho,\mu)}:=
\inf_{\vec{g}\in\mathbb{D}_\rho^s(u)}\Vert\vec{g}\Vert_{\ell^q(L^p(X))}<\infty,
\end{equation*}
and the \textit{inhomogeneous Haj\l asz--Besov space}
is defined as $N^s_{p,q}(X,\rho,\mu):=\dot{N}^s_{p,q}(X,\rho,\mu)\cap L^p(X),$
and it is equipped with the   (quasi-)norm
\begin{equation*}
\Vert u\Vert_{N^s_{p,q}(X)}:=\Vert u\Vert_{N^s_{p,q}(X,\rho,\mu)}:=\Vert u\Vert_{\dot{N}^s_{p,q}(X,\rho,\mu)}+\Vert u\Vert_{L^p(\Omega)},\quad\forall\ u\in N^s_{p,q}(X,\rho,\mu).
\end{equation*}

It is known that, when  $p\in[1,\infty)$ and $q\in[1,\infty]$,
$\Vert\cdot\Vert_{M^s_{p,q}(X,\rho,\mu)}$ and
$\Vert\cdot\Vert_{N^s_{p,q}(X,\rho,\mu)}$ are genuine norms and
the corresponding spaces $M^s_{p,q}(X,\rho,\mu)$ and $N^s_{p,q}(X,\rho,\mu)$ are Banach spaces.
Otherwise, they are quasi-Banach spaces.  We will simply use
$\dot{M}^s_{p,q}(X)$, $\dot{N}^s_{p,q}(X)$, $M^s_{p,q}(X)$,
and $N^s_{p,q}(X)$ in place of $\dot{M}^s_{p,q}(X,\rho,\mu)$,
$\dot{N}^s_{p,q}(X,\rho,\mu)$, $M^s_{p,q}(X,\rho,\mu)$, and
$N^s_{p,q}(X,\rho,\mu)$, respectively, whenever the
quasi-metric and the measure are well-understood from the context.
It was shown in \cite{KYZ11} that $M^s_{p,q}(\mathbb{R}^n)$ coincides with the classical Triebel--Lizorkin space $F^s_{p,q}(\mathbb{R}^n)$ for any $s\in(0,1)$, $p\in(\frac{n}{n+s},\infty)$, and $q\in(\frac{n}{n+s},\infty]$,  and $N^s_{p,q}(\mathbb{R}^n)$ coincides with the classical Besov space $B^s_{p,q}(\mathbb{R}^n)$ for any $s\in(0,1)$, $p\in(\frac{n}{n+s},\infty)$, and $q\in(0,\infty]$.
We also refer the reader to
\cite{GKZ13,HIT16,HKT17,Karak1,Karak2,agh20,hhhpl21,AWYY21} for more information on Triebel--Lizorkin and
Besov spaces on quasi-metric measure spaces.

The following proposition highlights the fact that
equivalent quasi-metrics generate equivalent Triebel--Lizorkin and Besov spaces.

\begin{proposition}
\label{equivspaces}
Let $(X,\rho,\mu)$ be a quasi-metric space equipped with a nonnegative Borel measure, and let
$s,\,p\in(0,\infty)$ and $q\in(0,\infty]$.
Suppose that $\varrho$ is a quasi-metric
on $X$ such that $\varrho\approx\rho$ and all $\varrho$-balls are $\mu$-measurable.
Then
$\dot{M}^s_{p,q}(X,\rho,\mu)=\dot{M}^s_{p,q}(X,\varrho,\mu)$ and $\dot{N}^s_{p,q}(X,\rho,\mu)=\dot{N}^s_{p,q}(X,\varrho,\mu)$, as sets, with equivalent (quasi-)norms, and,
consequently,  $$M^s_{p,q}(X,\rho,\mu)=M^s_{p,q}(X,\varrho,\mu)\quad \mbox{and}\quad N^s_{p,q}(X,\rho,\mu)=N^s_{p,q}(X,\varrho,\mu),$$ as sets, with equivalent (quasi-)norms.
\end{proposition}	

\begin{proof}
We only provide the details
for the proof of $\dot{M}^s_{p,q}(X,\rho,\mu)\subset\dot{M}^s_{p,q}(X,\varrho,\mu)$
as all other inclusions can be handled similarly. To this end, suppose that
$u\in\dot{M}^s_{p,q}(X,\rho,\mu)$ and
take any $\vec{g}:=\{g_k\}_{k\in\mathbb{Z}}\in\mathbb{D}_\rho^s(u)$.
Since $\varrho\approx\rho$, it follows that there exists a
constant $N\in\mathbb{N}$ such that
\begin{equation}\label{av-ui}
2^{-N}\varrho(x,y)\leq\rho(x,y)\leq 2^{N}\varrho(x,y),\quad\forall\ x,\,y\in X.
\end{equation}
Define
$$
h_k:=2^{sN}\sum_{j=-N}^Ng_{k+j},\quad\forall \ k\in\mathbb{Z}.
$$
We claim that $\vec{h}:=\{h_k\}_{k\in\mathbb{Z}}\in\mathbb{D}_\varrho^s(u)$.
To see this, let $E\subset X$ be a $\mu$-measurable set such that $\mu(E)=0$ and
\begin{equation}\label{av-ui2}
\vert u(x)-u(y)\vert\leq [\rho(x,y)]^s\left[g_k(x)+g_k(y)\right]
\end{equation}
for any $k\in\mathbb{Z}$
and  $x,\,y\in X\setminus E$
satisfying $2^{-k-1}\leq \rho(x,y)<2^{-k}$.
Fix a $k\in\mathbb{Z}$ and take $x,\,y\in X\setminus E$
satisfying  $2^{-k-1}\leq \varrho(x,y)<2^{-k}$.
By \eqref{av-ui}, we have
$2^{-k-1-N}\leq \rho(x,y)<2^{-k+N}$ and hence there exists a unique integer $j_0\in[k-N,k+N]$
such that $2^{-j_0-1}\leq \rho(x,y)<2^{-j_0}$. From this,
 \eqref{av-ui2}, and \eqref{av-ui}, we deduce that
\begin{align*}
\vert u(x)-u(y)\vert&\leq [\rho(x,y)]^s\lf[g_{j_0}(x)+g_{j_0}(y)\r]\noz\\
&\leq [\rho(x,y)]^s\left[\sum_{j=-N}^Ng_{k+j}(x)+\sum_{j=-N}^Ng_{k+j}(y)\right]\noz\\
&\leq [\varrho(x,y)]^s\left[h_k(x)+h_k(y)\right].
\end{align*}
Hence, $\vec{h}:=\{h_k\}_{k\in\mathbb{Z}}\in\mathbb{D}_\varrho^s(u)$, as wanted.
It is straightforward to prove that
$\Vert\vec{h}\Vert_{L^p(X,\ell^q)}\lesssim\Vert\vec{g}\Vert_{L^p(X,\ell^q)}$,
where the implicit positive constant  only depends on $s,\,p,\,q$, and $N$
(which ultimately depends on the proportionality constants in $\varrho\approx\rho$).
This finishes the proof of Proposition \ref{equivspaces}.
\end{proof}	

\begin{proposition}
\label{constant}
Let $(X,\rho,\mu)$ be a quasi-metric space equipped
with a nonnegative Borel measure, and suppose that
$s,\,p\in(0,\infty)$ and $q\in(0,\infty]$.
For any  $u\in\dot{M}^s_{p,q}(X,\rho,\mu)$,
one has that $\Vert u\Vert_{\dot{M}^s_{p,q}(X,\rho,\mu)}=0$
if and only if $u$ is constant $\mu$-almost everywhere in $X$.
The above statement also holds true with $\dot{M}^s_{p,q}$
replaced by $\dot{N}^s_{p,q}$.
\end{proposition}	

\begin{proof}
By similarity, we only consider $\dot{M}^s_{p,q}(X,\rho,\mu)$.
Fix a  $u\in\dot{M}^s_{p,q}(X,\rho,\mu)$.
If $u$ is constant $\mu$-almost everywhere in $X$,
then it is easy to check that $\{g_k\}_{k\in\mathbb{Z}}\in\mathbb{D}^s_\rho(u)$,
where $g_k\equiv0$ on $X$ for any $k\in\mathbb{N}$. Clearly,
$\Vert u\Vert_{\dot{M}^s_{p,q}(X,\rho,\mu)}
\leq\Vert\{g_k\}_{k\in\mathbb{Z}}
\Vert_{L^p(X,\ell^q)}=0$ and so, $\Vert u\Vert_{\dot{M}^s_{p,q}(X,\rho,\mu)}=0$, as desired.

Suppose next that $\Vert u\Vert_{\dot{M}^s_{p,q}(X,\rho,\mu)}=0$.
By the definition of
$\Vert \cdot\Vert_{\dot{M}^s_{p,q}(X,\rho,\mu)}$,
we can find a collection of sequences, $\{\vec{g}_{j}\}_{j\in\mathbb{N}}$,
with the property that, for any $j\in\mathbb{N}$,
$\vec{g}_j:=\{g_{j,k}\}_{k\in\mathbb{Z}}\in\mathbb{D}^s_\rho(u)$
and $\Vert g_{j,k}\Vert_{L^p(X,\mu)}\leq \Vert \vec{g}_j\Vert_{L^p(X,\ell^q)}<2^{-j}$
for any $k\in\mathbb{Z}$. Then it is straightforward to
check that, for any
fixed $k\in\mathbb{Z}$, $\sum_{j\in\mathbb{N}}g_{j,k}\in L^p(X,\mu)$ and
there exists a $\mu$-measurable set $N_k\subset X$   such that $\mu(N_k)=0$
where, for any $x\in X\setminus N_k$, $g_{j,k}(x)\to0$ as $j\to\infty$.
Going further, by \eqref{Hajlasz}, for any $j\in\mathbb{N}$,
there exists a $\mu$-measurable set $E_j\subset X$ such that $\mu(E_j)=0$ and
$$|u(x)-u(y)|\leq[\rho(x,y)]^s\lf[g_{j,k}(x)+g_{j,k}(y)\r]
$$
for any $k\in\mathbb{Z}$ and   $x,\,y\in X\setminus E_j$
satisfying $2^{-k-1}\leq\rho(x,y)<2^{-k}$. Consider the set
$$E:=\lf(\bigcup_{k\in\mathbb{Z}}N_k\r)\bigcup\lf(\bigcup_{j\in\mathbb{N}}E_j\r).$$
Then $E$ is  $\mu$-measurable and $\mu(E)=0$. Moreover, if $x,\,y\in X\setminus E$,
then there exists a unique $k\in\mathbb{Z}$ such that $2^{-k-1}\leq\rho(x,y)<2^{-k}$ and
\begin{equation}
\label{npw-656}
|u(x)-u(y)|\leq[\rho(x,y)]^s\lf[g_{j,k}(x)+g_{j,k}(y)\r],\quad\forall\ j\in\mathbb{N}.
\end{equation}
Passing to the limit in \eqref{npw-656} as $j\to\infty$, we find that $u(x)=u(y)$.
By the arbitrariness of $x,\,y\in X\setminus E$, we conclude that $u$ is constant
in $X\setminus E$, as wanted. This finishes the proof of Proposition \ref{constant}.
\end{proof}

Now, we recall the definition of  the Haj\l asz--Sobolev space.
Let $(X,\rho)$ be a quasi-metric space equipped with a
nonnegative Borel measure $\mu$, and let $s\in(0,\infty)$. Following \cite{hajlasz2,hajlasz,Y03},
a nonnegative $\mu$-measurable function  $g$  on $X$  is called
an \textit{$s$-gradient} of a $\mu$-measurable function $u\colon  X\rightarrow\mathbb{R}$ if
there exists a set $E\subset X$ with $\mu(E)=0$ such that
\begin{equation*}
\vert u(x)-u(y)\vert\leq [\rho(x,y)]^s\left[g(x)+g(y)\right],\quad\forall\ x,\,y\in X\setminus E.
\end{equation*}
The collection of all the $s$-gradients of $u$ is denoted by $\mathcal{D}_\rho^s(u)$.

Given $p\in(0,\infty)$, the \textit{homogeneous Haj\l asz--Sobolev space}
$\dot{M}^{s,p}(X,\rho,\mu)$ is defined as the collection of
all the $\mu$-measurable functions $u\colon  X\rightarrow\mathbb{R}$ such that
\begin{equation*}
\Vert u\Vert_{\dot{M}^{s,p}(X)}:=
\Vert u\Vert_{\dot{M}^{s,p}(X,\rho,\mu)}:=\inf_{g\in\mathcal{D}_\rho^s(u)}\Vert g\Vert_{L^p(X,\mu)}<\infty,
\end{equation*}
again, with the understanding that $\inf\emptyset:=\infty$.
The corresponding \textit{inhomogeneous Haj\l asz--Sobolev space}
$M^{s,p}(X,\rho,\mu)$ is defined by setting
$M^{s,p}(X,\rho,\mu):=\dot{M}^{s,p}(X,\rho,\mu)\cap L^p(X),$
equipped with the  (quasi-)norm
\begin{equation*}
\Vert u\Vert_{M^{s,p}(X)}:=\Vert u\Vert_{M^{s,p}(X,\rho,\mu)}
:=\Vert u\Vert_{\dot{M}^{s,p}(X,\rho,\mu)}+\Vert u\Vert_{L^p(X)},
\quad\forall\ u\in M^{s,p}(X,\rho,\mu).
\end{equation*}
It is known that,  for any $s\in(0,\fz)$ and $p\in(0,\fz)$,
$$\dot{M}^{s,p}(X,\rho,\mu)=\dot{M}^s_{p,\infty}(X,\rho,\mu)\quad \mbox{and}\quad
{M}^{s,p}(X,\rho,\mu)={M}^s_{p,\infty}(X,\rho,\mu);$$
see, for instance, \cite[Proposition~2.1]{KYZ11} and \cite[Proposition 2.4]{AYY21}.



\subsection{Embedding Theorems for Triebel--Lizorkin and Besov Spaces}

In this subsection, we recall a number of embedding
results that were recently obtained in
\cite{AYY21}, beginning with some Sobolev--Poincar\'e-type inequalities
for fractional Haj\l asz--Sobolev spaces.
To facilitate the formulation of the result,
we introduce the following piece of notation:
Given $Q,\,b\in(0,\infty)$, $\sigma\in[1,\infty)$,
and a ball $B_0\subset X$ of radius $R_0\in(0,\infty)$,
the measure $\mu$ is said to satisfy the \emph{$V(\sigma B_0,Q,b)$
condition} provided
\begin{equation}
\label{measbound}
\mu(B_\rho(x,r))\geq br^Q
\quad
\text{for any}\
\text{$x\in X$ and $r\in(0,\sigma R_0]$ satisfying $B_\rho(x,r)\subset\sigma B_0$.}
\end{equation}
We remind the reader
that the constants $C_\rho$, $\widetilde{C}_\rho\in[1,\infty)$
were defined in \eqref{C-RHO.111} and \eqref{C-RHO.111XXX}, respectively.
The following conclusion was obtained in \cite[Theorem 3.1]{AYY21}.

\begin{theorem}\label{embedding}
Let $(X,\rho,\mu)$ be a quasi-metric measure space.	
Let $s,\,p\in(0,\infty)$, $\sigma\in[C_\rho,\infty)$, and $B_0$ be  a
$\rho$-ball of radius $R_0\in(0,\infty)$.
 Assume that the measure $\mu$
satisfies the  $V(\sigma B_0,Q,b)$ condition for some $Q,\,b\in(0,\infty)$.
Let $u\in\dot{M}^{s,p}(\sigma B_0,\rho,\mu)$ and $g\in \mathcal{D}_\rho^s(u)$.
Then  $u\in L^{p^*}(B_0,\mu)$ and there exists a positive constant  $C$,
depending only on $\rho$, $s$, $p$, $Q$, and $\sigma$, such that
\begin{equation}
\label{eq19}
\inf_{\gamma\in\mathbb{R}}\left(\, \mvint_{B_0} |u-\gamma|^{p^*}\, d\mu\right)^{1/p^*}\leq
C\left[\frac{\mu(\sigma B_0)}{bR_0^Q}\right]^{1/p}R_0^s\left(\,\mvint_{\sigma B_0}g^{p}\, d\mu\right)^{1/p},
\end{equation}	
where $p^*:=Qp/(Q-sp)$.
\end{theorem}

We will also employ the following
embeddings of Haj\l asz--Triebel--Lizorkin and Haj\l asz--Besov spaces
on spaces of homogeneous type obtained in \cite{AYY21}. To state these results,
let us introduce the following piece of terminology:
Given a quasi-metric measure space $(X,\rho,\mu)$,
the measure $\mu$ is said to be \emph{$Q$-doubling up to
scale} $r_\ast\in(0,\infty]$, provided that
there exist positive
constants $\kappa$ and $Q$ satisfying
\begin{equation*}
\kappa\left(\frac{r}{R}\r)^{Q}\leq\frac{\mu(B_\rho(x,r))}{\mu(B_\rho(y,R))},
\end{equation*}
whenever $x,\,y\in X$ with $B_\rho(x,r)\subset B_\rho(y,R)$ and $0<r\leq R\le r_\ast$.
When $r_\ast=\infty$,    $\mu$ is simply $Q$-doubling  as before [see \eqref{Doub-2}].

The following inequalities were obtained in \cite[Theorem 3.7 and Remark 3.8]{AYY21}.

\begin{theorem}
\label{DOUBembedding}
Let $(X,\rho,\mu)$ be a quasi-metric measure space,
where $\mu$ is  $Q$-doubling up to scale $r_\ast\in(0,\infty]$
for some $Q\in(0,\infty)$. Let $s,\,p\in(0,\infty)$ and
$q\in(0,\infty]$, and set $\sigma:=C_\rho$.
Then there exist positive
constants $C$, $C_1$, and $C_2$, depending only on $\rho$, $\kappa$,
$Q$, $s$, and $p$, such that the following statements hold true for any
$\rho$-ball $B_0:=B_\rho(x_0,R_0)$, where $x_0\in X$ and
$R_0\in(0,r_\ast/\sigma]$ is finite:
\begin{enumerate}
\item[{\rm(a)}] If $p\in(0,Q/s)$, then any
$u\in \dot{M}^s_{p,q}(\sigma B_0,\rho,\mu)$ belongs
to $L^{p^*}(B_0,\mu)$ with $p^*:=Qp/(Q-sp)$, and satisfies
\begin{equation}
\label{eq18-DOUB}
\Vert u\Vert_{L^{p^\ast}(B_0)}\leq
\frac{C}{[\mu(\sigma B_0)]^{s/Q}}\left[
R_0^s\Vert u\Vert_{\dot{M}^s_{p,q}(\sigma B_0)}
+\Vert u\Vert_{L^{p}(\sigma B_0)}\right]
\end{equation}
and
\begin{equation}
\label{eq19-DOUB}
\inf_{\gamma\in\mathbb{R}}\Vert u-\gamma\Vert_{L^{p^\ast}(B_0)}\leq
\frac{C}{[\mu(\sigma B_0)]^{s/Q}}\,R_0^{s}\Vert u\Vert_{\dot{M}^s_{p,q}(\sigma B_0)}.
\end{equation}

\item[{\rm(b)}] If $p=Q/s$, then, for any $u\in\dot{M}^s_{p,q}(\sigma B_0,\rho,\mu)$
with $\Vert u\Vert_{\dot{M}^s_{p,q}(\sigma B_0)}>0$, one has
\begin{equation}
\label{eq20-DOUB}
\int_{B_0} {\rm exp}\left(C_1\frac{[\mu(\sigma B_0)]^{s/Q}|u-u_{B_0}|}
{R_0^s\Vert u\Vert_{\dot{M}^s_{p,q}(\sigma B_0)}}\right)\,d\mu\leq C_2\mu(B_0).
\end{equation}

\item[{\rm(c)}] If $p\in(Q/s,\fz)$,  then
 each function $u\in\dot{M}^s_{p,q}(\sigma B_0,\rho,\mu)$
 has a H\"older continuous representative of order $s-Q/p$ on $B_0$,
 denoted by $u$ again, satisfying
\begin{equation}
\label{eq30-DOUB}
|u(x)-u(y)|\leq C[\rho(x,y)]^{s-Q/p}\frac{R_0^{Q/p}}
{[\mu(\sigma B_0)]^{1/p}}\,\Vert u\Vert_{\dot{M}^s_{p,q}(\sigma B_0)},
\quad \forall \ x,\, y\in B_0.
\end{equation}
\end{enumerate}
In addition, if $q\leq p$, then the above statements are
valid with  $\dot{M}^s_{p,q}$ replaced by $\dot{N}^s_{p,q}$.
\end{theorem}

The embeddings for the spaces $\dot{N}^s_{p,q}$
in Theorem~\ref{DOUBembedding} are restricted to
the case when $q\leq p$; however, an upper bound
on the exponent $q$ is to be expected; see  \cite[Remark~4.17]{AYY21}.
As the following theorem (see \cite[Theorem 3.9]{AYY21}) indicates, one can relax the restriction on $q$ and still obtain
Sobolev-type embeddings with the critical exponent
$Q/s$ replaced by $Q/\varepsilon$, where
$\varepsilon\in(0,s)$ is any fixed number.
These embeddings will be very useful in Section~\ref{section:measuredensity} when
characterizing ${M}^s_{p,q}$ and ${N}^s_{p,q}$-extension domains.

\begin{theorem}
\label{mainembedding-epsilon}
Let $(X,\rho,\mu)$ be a quasi-metric measure space, where $\mu$ is
$Q$-doubling up to scale $r_\ast\in(0,\infty]$ for some $Q\in(0,\infty)$.
Let $s,\,p\in(0,\infty)$ and $q\in(0,\infty]$, and set $\sigma:=C_\rho$.
Then, for any fixed $\varepsilon\in(0,s)$,
there exist positive constants $C$, $C_1$, and $C_2$, depending only on
$\rho$, $\kappa$, $Q$, $s$, $\varepsilon$, and $p$, such that
 the following statements hold true
for any $\rho$-ball $B_0:=B_\rho(x_0,R_0)$, where $x_0\in X$ and $R_0\in(0,r_\ast/\sigma]$ is finite:
\begin{enumerate}
\item[{\rm(a)}] If $p\in(0,Q/\varepsilon)$, then
any $u\in\dot{N}^s_{p,q}(\sigma B_0,\rho,\mu)$
belongs to $L^{p^*}(B_0,\mu)$
with $p^*:=Qp/(Q-\varepsilon p)$, and  satisfies
\begin{equation*}
\Vert u\Vert_{L^{p^\ast}(B_0)}\leq
\frac{C}{[\mu(\sigma B_0)]^{\varepsilon/Q}}\left[
R_0^s\Vert u\Vert_{\dot{N}^s_{p,q}(\sigma B_0)}
+\Vert u\Vert_{L^{p}(\sigma B_0)}\right]
\end{equation*}
and
\begin{equation*}
\inf_{\gamma\in\mathbb{R}}\Vert u-\gamma\Vert_{L^{p^\ast}(B_0)}\leq
\frac{C}{[\mu(\sigma B_0)]^{\varepsilon/Q}}\,R_0^{s}\Vert u\Vert_{\dot{N}^s_{p,q}(\sigma B_0)}.
\end{equation*}

\item[{\rm(b)}] If $p=Q/\varepsilon$, then, for any
$u\in\dot{N}^s_{p,q}(\sigma B_0,\rho,\mu)$  with
$\Vert u\Vert_{\dot{N}^s_{p,q}(\sigma B_0)}>0$, one has
\begin{equation*}
\int_{B_0} {\rm exp}\left(C_1\frac{[\mu(\sigma B_0)]^{\varepsilon/Q}
|u-u_{B_0}|}{R_0^{s}\Vert u\Vert_{\dot{N}^s_{p,q}(\sigma B_0)}}\right)\,d\mu\leq C_2\mu(B_0).
\end{equation*}

\item[{\rm(c)}] If $p\in(Q/\varepsilon,\fz)$, then each  function
$u\in\dot{N}^s_{p,q}(\sigma B_0,\rho,\mu)$
has a H\"older continuous representative of order $s-Q/p$ on $B_0$, denoted by $u$ again,  satisfying
\begin{equation*}
|u(x)-u(y)|\leq C[\rho(x,y)]^{s-Q/p}\frac{R_0^{Q/p}}{[\mu(\sigma B_0)]^{1/p}}
\Vert u\Vert_{\dot{N}^s_{p,q}(\sigma B_0)},
\quad
\forall\ x,\,y\in B_0.
\end{equation*}
\end{enumerate}
\end{theorem}

In the Ahlfors regular setting, we have the following estimates, which follows from
\cite[Theorem~1.4]{AYY21}. (Note that the uniformly perfect property was not used in proving that {\rm (a)} implies {\rm (b)-(e)} in  \cite[Theorem~1.4]{AYY21}.)

\begin{theorem}
\label{GlobalEmbeddCor}
Let $(X,\rho,\mu)$ be a quasi-metric measure space, where $\mu$ is a $Q$-Ahlfors regular
measure on $X$ for some $Q\in(0,\infty)$.
Let $s,\,p\in(0,\infty)$ and $q\in(0,\infty]$.
Then there exists a positive constant $C$
such that the following statements hold true
for any $u\in \dot{M}^s_{p,q}(X,\rho,\mu)$:
\begin{enumerate}[label={\rm(\alph*)}]
\item If $p\in(0,Q/s)$, then
\begin{equation*}
\|u\|_{L^{p^*}(X)}\leq C\|u\|_{\dot{M}^s_{p,q}(X)}+\frac{C}{[{\rm diam}_\rho(X)]^s}\,\|u\|_{L^p(X)}
\end{equation*}		
and
\begin{equation*}
\inf_{\gamma\in\mathbb{R}}\|u-\gamma\|_{L^{p^\ast}(X)}\leq
C\|u\|_{\dot{M}^s_{p,q}(X)},
\end{equation*}
where $p^*:=Qp/(Q-sp)$.

\item  If $p=Q/s$ and  $\Vert u\Vert_{\dot{M}^s_{p,q}(X)}>0$, then
\begin{equation*}
\int_{B} {\rm exp}
\left(C_1\frac{|u-u_B|}{\Vert u\Vert_{\dot{M}^s_{p,q}(X)}}\right)\,d\mu\leq C_2r^Q
\end{equation*}
for any $\rho$-ball $B\subset X$ having finite radius $r\in(0,{\rm diam}_\rho(X)]$.

\item If $p\in(Q/s,\fz)$,
then $u$ has a H\"older continuous representative of order $s-Q/p$ on $X$, denoted by $u$ again,  satisfying
\begin{equation*}
|u(x)-u(y)|\leq C[\rho(x,y)]^{s-Q/p}\Vert u\Vert_{\dot{M}^s_{p,q}(X)},
\quad \forall\ x,\,y\in X.
\end{equation*}
\end{enumerate}
In addition, if $q\leq p$, then all of the statements above are valid
with $\dot{M}^s_{p,q}$ replaced by $\dot{N}^s_{p,q}$.
\end{theorem}

\section{Optimal Extensions for Triebel--Lizorkin Spaces and Besov Spaces on Spaces of Homogeneous Type}
\label{section:extensions}

This section is devoted to the proof of Theorem~\ref{measext-INT}.
To this end, we first collect   a number of
necessary tools and results in Subsection~\ref{sssec:extensions1}.
The proof of Theorem~\ref{measext-INT} will be presented in Subsection
\ref{sssec:extensions2}.

\subsection{Main Tools}
\label{sssec:extensions1}

The existence of   $M^s_{p,q}$ and
$N^s_{p,q}$ extension operators on domains
satisfying the measure density condition
for an \textit{optimal} range of $s$ relies on two basic ingredients:
\begin{enumerate}
\item the existence of a Whitney decomposition of an open subset
of the quasi-metric space under consideration  (into the Whitney balls of bounded overlap);

\item the existence of a partition of unity subordinate
(in an appropriate, quantitative manner) to such a
decomposition that exhibits an optimal amount of
smoothness (measured on the H\"older scale).
\end{enumerate}
It is the existence of a partition of unity
of maximal smoothness order that  permits us to construct an
extension operator for an optimal range of $s$.
The availability of these ingredients was established in
\cite[Theorem 6.3]{AMM13} (see also \cite[Theorems~2.4 and 2.5]{AM15}).
We recall their statements here for the convenience of the reader.

\begin{theorem}[{\bf Whitney-type} {\bf decomposition}]\label{L-WHIT}
Suppose that $(X,\rho,\mu)$ is a quasi-metric measure space,
where $\mu$ is a doubling measure. Then, for any
$\theta\in (1,\infty)$, there exist constants
$\Lambda\in(\theta,\infty)$ and $M\in{\mathbb{N}}$,
which depend only on  both $\theta$ and the space $(X,\rho,\mu)$,
and which have the following significance. 	
For any proper, nonempty, open subset $\mathcal{O}$ of
 the topological space $(X,\tau_\rho)$, where $\tau_\rho$
 denotes the topology canonically induced by the quasi-metric
 $\rho$, there exist a sequence $\{x_j\}_{j\in{\mathbb{N}}}$ of points
 in $\mathcal{O}$ along with a family  $\{r_j\}_{j\in{\mathbb{N}}}$ of (strictly) positive numbers,
for which the following properties are valid:
\begin{enumerate}
\item[{\rm(i)}] $\mathcal{O}=\bigcup\limits_{j\in{\mathbb{N}}}B_\rho(x_j,r_j)$;
\item[{\rm(ii)}] $\sum\limits_{j\in{\mathbb{N}}}{\mathbf 1}_{B_\rho(x_j,\theta r_j)}\leq M$
pointwise in $\mathcal{O}$. In fact,
there exists an $\varepsilon\in(0,1)$,   depending only on both $\theta$ and the space $(X,\rho,\mu)$,
with the property that, for any $x_0\in\mathcal{O}$,
\begin{eqnarray*}
\hskip -0.20in
\#\,\Bigl\{j\in{\mathbb{N}}:\,
B_\rho\bigl(x_0,\varepsilon\,{\rm dist}_\rho(x_0,X\setminus\mathcal{O})\bigr)
\cap B_{\rho}(x_j,\theta r_j)\not=\emptyset\Bigr\}\leq M,
\end{eqnarray*}
where, in general,   $\#E$ denotes the cardinality of a set $E$;
\item[{\rm(iii)}] For any $j\in{\mathbb{N}}$ ,$B_\rho(x_j,\theta r_j)\subset
\mathcal{O}$ and $B_\rho(x_j,\Lambda r_j)
\cap [X\setminus\mathcal{O}]\not=\emptyset$;
\item[{\rm(iv)}] $r_i\approx r_j$ uniformly for $i,\,j\in{\mathbb{N}}$
such that $B_\rho(x_i,\theta r_i)\cap B_\rho(x_j,\theta r_j)\not=\emptyset$.
\end{enumerate}
\end{theorem}

Given a quasi-metric space $(X,\rho)$ and an exponent $\alpha\in(0,\infty)$,
recall that the {\it homogeneous} {\it H\"older} {\it space} (of order $\alpha$)
$\dot{\mathscr{C}}^\alpha(X,\rho)$ is the collection of all the functions $f\colon X\to\mathbb{R}$ such that
\begin{eqnarray*}
\|f\|_{\dot{\mathscr{C}}^\alpha(X,\rho)}:=
\sup_{x,\,y\in X,\ x\not=y}\frac{|f(x)-f(y)|}{[\rho(x,y)]^\alpha}<\infty.
\end{eqnarray*}

\begin{theorem}[{\bf Partition} {\bf of} {\bf Unity}]\label{MSz7b}
Let $(X,\rho,\mu)$ be a quasi-metric measure space, where
$\mu$ is a doubling measure. Suppose that
$\mathcal{O}$ is a proper nonempty
subset of $X$. Fix a
number $\theta\in (C^2_\rho,\fz)$,
where $C_\rho$ is as in \eqref{C-RHO.111},
and consider the decomposition of $\mathcal{O}$
into the family $\{B_\rho(x_j,r_j)\}_{j\in\mathbb{N}}$
as given by Theorem~\ref{L-WHIT} for this choice of $\theta$.
Let $\theta'\in(C_\rho,\theta/C_\rho)$.
Then, for any $\alpha\in\mathbb{R}$ satisfying
\begin{equation}\label{FFg-11.1}
0<\alpha\leq\left[{\rm log}_2 C_{\rho}\r]^{-1},
\end{equation}
there exists a  constant $C_\ast\in [1,\fz)$,
depending only on $\rho$, $\alpha$, $M$, and
the proportionality constants in
Theorem~\ref{L-WHIT}(iv), along with a family $\{\psi_j\}_{j\in\mathbb{N}}$
of real-valued functions  on $X$ such that the following statements are valid:
\begin{enumerate}
\item[{\rm(i)}] for any $j\in\mathbb{N}$, one has
\begin{eqnarray*}
\psi_j\in\dot{\mathscr{C}}^\beta(X,\rho)\quad\mbox{and}
\quad \|\psi_j\|_{\dot{\mathscr{C}}^\beta(X,\rho)}\leq C_\ast r_j^{-\beta},\quad\forall\ \beta\in(0,\alpha];
\end{eqnarray*}

\item[{\rm(ii)}] for any $j\in\mathbb{N}$, one has
\begin{eqnarray*}
\hskip -0.40in
0\leq\psi_j\leq 1\,\,\mbox{ on }\,\,X,\quad\psi_j\equiv 0
\,\,\mbox{ on }\,\,X\setminus B_\rho(x_j,\theta' r_j),
\quad\mbox{and}\quad\psi_j\geq 1/C_\ast\,\,\mbox{ on }\,\,B_\rho(x_j,r_j);
\end{eqnarray*}
\item[{\rm(iii)}] it holds true that $$\sum_{j\in\mathbb{N}}\psi_j
={\mathbf 1}_{\bigcup_{j\in\mathbb{N}}B_\rho(x_j,r_j)}
={\mathbf 1}_{\bigcup_{j\in\mathbb{N}}B_\rho(x_j,\theta' r_j)}
={\mathbf 1}_{\bigcup_{j\in\mathbb{N}}B_\rho(x_j,\theta r_j)}$$ pointwise on $X$.
\end{enumerate}
\end{theorem}

The range for the parameter $\alpha$ in \eqref{FFg-11.1} is directly linked to
the geometry of
the underlying  quasi-metric spaces
and it is a result of the following
sharp metrization theorem established in \cite[Theorem 3.46]{MMMM13}.
\begin{theorem}
\label{DST1}
Let $(X,\rho)$ be a quasi-metric space
and assume that  $C_\rho$, $\widetilde{C}_\rho\in[1,\infty)$ are as in
\eqref{C-RHO.111} and \eqref{C-RHO.111XXX}, respectively.
Then there exists a symmetric quasi-metric
$\rho_{\#}$ on $X$ satisfying  $\rho_{\#}\approx\rho$ and  having the property that
for any finite number $\alpha\in(0,(\log_2C_\rho)^{-1}]$,
the function $(\rho_{\#})^\alpha\colon  X\times X\to\mathbb{R}$
is a genuine metric on $X$,  that is,
\begin{equation*}
[\rho_{\#}(x,y)]^\alpha\leq [\rho_{\#}(x,z)]^\alpha+[\rho_{\#}(z,y)]^\alpha,
\quad\forall\ x,\,y,\,z\in X.
\end{equation*}
Moreover, the function $\rho_{\#}:X\times X\longrightarrow [0,\infty)$
is continuous when $X\times X$ is equipped with the natural
product topology $\tau_\rho\times\tau_\rho$.
In particular, all $\rho_{\#}$-balls are open in the topology $\tau_\rho$.
\end{theorem}

The following Poincar\'e-type inequality
will play an important role in proving our main extension result.
\begin{lemma}
\label{embeddfracgrad}
Let $(X,\rho,\mu)$ be a quasi-metric measure space, where
$\mu$ is a $Q$-doubling measure for some $Q\in(0,\infty)$,
and suppose that $\Omega\subset X$ is a nonempty
$\mu$-measurable set that satisfies the measure density condition \eqref{measdens-INT-a}, namely,
there exists a positive constant $C_\mu$ such that
\begin{equation}
\label{measdens-lemma-X}
\mu(B_\rho(x,r))\leq C_\mu\,
\mu(B_\rho(x,r)\cap\Omega),\quad\forall\ x\in\Omega, \ \ \forall\  r\in(0,1].
\end{equation}
Fix $t,\, s,\,\varepsilon,\,\varepsilon'\in(0,\infty)$ with
$t<Q/\varepsilon'$ and $\varepsilon'<\varepsilon<s$, and
let $k_0\in\mathbb{N}$ be any number such
that $2^{k_0}\geq C_\rho^2\widetilde{C}_\rho$, where $C_\rho$,
$\widetilde{C}_\rho\in[1,\infty)$ are as in \eqref{C-RHO.111}
and \eqref{C-RHO.111XXX}, respectively.
Then, for any $r_\ast\in[1,\infty)$,
there exists a positive constant $C$,
depending on $r_\ast$, $\mu$, $\rho$, $Q$, $t$, $\varepsilon'$, and $\varepsilon$,
such that, for any $\mu$-measurable function
$u\colon  \Omega\to\mathbb{R}$, $\{g_k\}_{k\in\mathbb{Z}}\in\mathbb{D}^s_\rho(u)$,
$x\in \Omega$, and $L\in\mathbb{Z}$ satisfying $2^{-L}\leq r_\ast$, one has
\begin{equation}
\label{cie3-2}
\inf_{\gamma\in\mathbb{R}}\left[\, \mvint_{B_\rho(x,2^{-L})\cap\Omega} |u-\gamma|^{t^\ast}\,
d\mu\right]^{1/t^\ast}
\leq C 2^{-L\varepsilon}\left[\,\mvint_{B_\rho(x,C_\rho2^{-L})}
\sup_{j\geq L-k_0}\lf\{2^{-j(s-\varepsilon)t}g_j^t\r\}\, d\mu\right]^{1/t},
\end{equation}
where $t^\ast:=Qt/(Q-\varepsilon' t)$. Here, for any $j\in\mathbb{Z}$, the function $g_j$ is identified with function defined on all of $X$ by setting $g_j\equiv0$ on $X\setminus\Omega$.
\end{lemma}

\begin{remark}
\label{linearlocal}
In the context of Lemma~\ref{embeddfracgrad}, if the function
$u$ is  integrable on domain balls and we take $t:=Q/(Q+\varepsilon')$,
then $t^\ast=1$ and we can replace the
infimum in \eqref{cie3-2} by the ball average $u_{B_\rho(x,2^{-L})\cap\Omega}$,  and obtain
\begin{equation*}
\begin{split}
\mvint_{B_\rho(x,2^{-L})\cap\Omega} \lf|u-u_{B_\rho(x,2^{-L})\cap\Omega}\r|\, d\mu
\le C 2^{-L\varepsilon}
\left[\,\mvint_{B_\rho(x,C_\rho2^{-L})}\sup_{j\geq L-k_0}
\lf\{2^{-j(s-\varepsilon)t}g_j^t\r\}\, d\mu\right]^{1/t},
\end{split}
\end{equation*}
where the positive constant $C$ is twice the constant appearing in Lemma \ref{embeddfracgrad}.
This inequality will be important in the proof of Theorem~\ref{measext}
when constructing an extension operator that is linear.
\end{remark}

\begin{proof}[Proof of Lemma \ref{embeddfracgrad}]
Fix a $\mu$-measurable function $u\colon  \Omega\to\mathbb{R}$ and
suppose that $r_\ast\in[1,\infty)$, $\{g_k\}_{k\in\mathbb{Z}}\in\mathbb{D}^s_\rho(u)$,
$x\in \Omega$, and $L\in\mathbb{Z}$ satisfies $2^{-L}\leq r_\ast$.
Without loss of generality, we may assume that the right-hand side of
\eqref{cie3-2} is finite.

If $\Omega=\{x\}$ then \eqref{cie3-2} trivially holds. Therefore, we can assume that $\Omega$ as cardinality at least two, in which case, we have that  $(\Omega,\rho,\mu)$ is a well-defined quasi-metric measure space where $\rho$ and $\mu$ are restricted to $\Omega$.
Our plan is to use \eqref{eq19} in Theorem~\ref{embedding} for the
quasi-metric measure space $(\Omega,\rho,\mu)$ with $\sigma:=C_\rho$,
$R_0:=2^{-L}$, and $B_0:=B_\rho(x,2^{-L})$.
To this end, we first claim that the measure $\mu$
satisfies the $V(\sigma B_0\cap\Omega, Q, b)$ condition  [that is, \eqref{measbound} with $\sigma B_0$,
$B_\rho(x,r)$, and $x\in X$ replaced, respectively, by $\sigma B_0 \cap \Omega$,
$B_\rho(z,r)\cap \Omega$, and $z\in \Omega$]
for some choice of $b$. To show this,
suppose that $z\in\Omega$
and $r\in(0,\sigma R_0]$ satisfy
$B_\rho(z,r)\cap\Omega\subset \sigma B_0\cap\Omega$.
Then $(\sigma r_\ast)^{-1}r\leq r\leq \sigma^2 R_0$
and $(\sigma r_\ast)^{-1}r\leq r_\ast^{-1}R_0=r_\ast^{-1}2^{-L}\leq 1$.
Moreover, by $z\in\sigma B_0$, we obtain $B_\rho(z,(\sigma r_\ast)^{-1}r)\subset \sigma^2 B_0$.
Indeed, if $y\in B_\rho(z,(\sigma r_\ast)^{-1}r)$, then
$$
\rho(x,y)\leq C_\rho\max\{\rho(x,z),\rho(z,y)\}<C_\rho\max\{\sigma R_0,
\sigma^{-1}r\}=C_\rho\sigma R_0=\sigma^2 R_0.
$$
Using the $Q$-doubling property \eqref{Doub-2} and  the measure density condition \eqref{measdens-lemma-X},  we conclude that
\begin{align*}
\frac{\mu(B_\rho(z,r)\cap\Omega)}{\mu(\sigma B_0\cap\Omega)}
&\geq\frac{\mu(B_\rho(z,(\sigma r_\ast)^{-1}r)\cap\Omega)}{\mu(\sigma B_0)}\\
&\geq\frac{C_\mu^{-1}\mu(B_\rho(z,(\sigma r_\ast)^{-1}r))}{\mu(\sigma^2 B_0)}\geq
\kappa C_\mu^{-1}\lf[\frac{(\sigma r_\ast)^{-1}r}{\sigma^2 R_0}\r]^{Q}.
\end{align*}
Hence, $\mu$ satisfies the $V(\sigma B_0\cap\Omega, Q, b)$ condition
with $b:=\kappa C_\mu^{-1}\mu(\sigma B_0\cap\Omega)(\sigma^3 r_\ast R_0)^{-Q}$.
This proves the above claim.

Moving on, we claim next that $u\in\dot{M}^{\varepsilon',t}(\sigma B_0\cap\Omega)$.  To see this,
let
$$g:=\sup_{j\geq L-k_0}\lf\{2^{-j(s-\varepsilon')}g_j\right\},$$
where $k_0\in\mathbb{N}$ satisfies
$2^{k_0}\geq C_\rho^2\widetilde{C}_\rho$.
Then $g\in\mathcal{D}_\rho^{\varepsilon'}(u)$. Indeed, observe that,
for any points $y,\, z\in \sigma B_0\cap\Omega=B_\rho(x,C_\rho2^{-L})\cap\Omega$, we have
$$
\rho(y,z)\leq C_\rho\max\{\rho(y,x),\rho(x,z)\}<C_\rho^2\widetilde{C}_\rho 2^{-L}\leq 2^{-L+k_0},
$$
which further implies that there exists a unique integer $j_0 \geq L-k_0$
satisfying $2^{-j_0-1}\leq \rho(y,z)<2^{-j_0}.$
Then, by $\{g_k\}_{k\in\zz}\in\mathbb{D}_\rho^s(u)$, we conclude
that there exists an $E\subset X$ with $\mu(E)=0$ such that,
for any $y,\, z\in (\sigma B_0\cap\Omega)\setminus E$,
\begin{align*}
|u(y)-u(z)|&\leq[\rho(y,z)]^s\lf[g_{j_0}(y)+g_{j_0}(z)\r]\\
&<
[\rho(y,z)]^{\varepsilon'} 2^{-j_0(s-\varepsilon')}\lf[g_{j_0}(y)+g_{j_0}(z)\r]\leq [\rho(y,z)]^{\varepsilon'}[g(y)+g(z)],
\end{align*}
which implies that $g\in\mathcal{D}_\rho^{\varepsilon'}(u)$, as wanted. Observe that
\begin{align}
\label{ntq-23-2}
\|g\|_{L^t(\sigma B_0\cap\Omega,\mu)}
&=\left[\int_{\sigma B_0}
\sup_{j\geq L-k_0}\lf\{2^{-j(s-\varepsilon')t}g_j^t\r\}\,d\mu\right]^{\frac{1}{t}}\noz\\
&\leq
\sup_{j\geq L-k_0}2^{-j(\varepsilon-\varepsilon')}\left[\int_{\sigma B_0}\sup_{j\geq L-k_0}\lf\{2^{-j(s-\varepsilon)t}g_j^t\r\}\,d\mu\right]^{\frac{1}{t}}\noz\\
&\leq 2^{-(L-k_0)(\varepsilon-\varepsilon')}\left[\int_{\sigma B_0}\sup_{j\geq L-k_0}\lf\{2^{-j(s-\varepsilon)t}g_j^t\r\}\,d\mu\right]^{\frac{1}{t}}<\infty,
\end{align}
where we have used the fact that for any $j\in\mathbb{Z}$, $g_j\equiv0$ on $X\setminus\Omega$.
It follows that $u\in\dot{M}^{\varepsilon',t}(\sigma B_0\cap\Omega)$.

Since $t<Q/\varepsilon'$, from \eqref{eq19} in
Theorem~\ref{embedding} (keeping in mind
the value of $b$  above) and \eqref{ntq-23-2}, we deduce that
\begin{align}
\label{xwi-845}
&\inf_{\gamma\in\mathbb{R}}\left[\, \mvint_{B_\rho(x,2^{-L})\cap\Omega} |u-\gamma|^{t^\ast}\, d\mu\right]^{1/t^\ast}\noz\\
&\quad\lesssim 2^{-L\varepsilon'}\left[\,\mvint_{B_\rho(x,C_\rho2^{-L})\cap\Omega}g^t\, d\mu\right]^{1/t}
\lesssim 2^{-L\varepsilon}\left[\,\mvint_{B_\rho(x,C_\rho2^{-L})\cap\Omega}\sup_{j\geq L-k_0}\lf\{2^{-j(s-\varepsilon)t}g_j^t\r\}\,d\mu\right]^{\frac{1}{t}}.
\end{align}
Note that the measure density condition \eqref{measdens-lemma-X} and the doubling
property for $\mu$ imply (keeping in mind that $C_\rho,r_\ast\geq1$ and $r^{-1}_\ast2^{-L}\leq1$)
$$
\mu(B_\rho(x,C_\rho2^{-L})\cap\Omega)\geq\mu(B_\rho(x,r^{-1}_\ast2^{-L})\cap\Omega)
\gtrsim\mu(B_\rho(x,r^{-1}_\ast2^{-L}))\gtrsim\mu(B_\rho(x,C_\rho2^{-L})),
$$
which, together with \eqref{xwi-845}, implies
the desired inequality \eqref{cie3-2}. The proof of Lemma~\ref{embeddfracgrad} is now complete.
\end{proof}

The $M^s_{p,q}$ and the $N^s_{p,q}$ extension operators that we construct in Theorem~\ref{measext} will be a Whitney-type operator which, generally speaking in this context, is based on gluing various `averages' of a function together with using a partition of unity that is sufficiently smooth (related to the parameter $s$). Since functions in $M^s_{p,q}$ and $N^s_{p,q}$ may not be locally integrable for small values of $p$, we cannot always use integral averages of these functions when constructing the extension operator. As a substitute, we use the so-called median value of a function in place of its integral average when $p$ or $q$ is  small (see, for instance, \cite{F91}, \cite{Z15}, and \cite{HIT16}).

\begin{definition}\label{median}
Let $(X,\mu)$ be a measure space, where $\mu$ is a nonnegative measure.
Given a $\mu$-measurable set $E\subset X$ with $\mu(E)\in(0,\infty)$,
the \textit{median value} $m_u(E)$ of  a $\mu$-measurable
function $u\colon X\to\mathbb{R}$ on $E$ is defined by setting
$$
m_u(E):=\max_{\theta\in\mathbb{R}}\left\{\mu\left(\left\{x\in E:\,u(x)<\theta\right\}\right)\leq\frac{\mu(E)}{2}\right\}.
$$
\end{definition}	

It is straightforward to check that $m_u(E)\in\mathbb{R}$ is a well-defined number using basic properties of the measure $\mu$. We now take a moment to collect a few key properties of $m_u$ that illustrate how these quantities mimic the usual integral average of a locally integrable function. One drawback of the median value is that there is no guarantee that $m_u(E)$ is linear in $u$. Consequently, the resulting extension operator constructed using median values may not be linear.

\begin{lemma}
\label{difflemma}
Let $(X,\rho,\mu)$ be a quasi-metric measure space and fix an $\eta\in(0,\infty)$. Then,
for any $\mu$-measurable function $u\colon  X\to\mathbb{R}$,
any $\rho$-ball $B\subset X$, and any $\gamma\in\mathbb{R}$, one has
\begin{equation}
\label{hcx-23}
|m_u(B)-\gamma|\leq\left(2\,\mvint_{B}|u-\gamma|^\eta\, d\mu\right)^{1/\eta}.	
\end{equation}
\end{lemma}

Estimate \eqref{hcx-23} was established in \cite[(2.4)]{GKZ13} under the assumption
that $\rho$ is a genuine metric (see also \cite[Lemma 2.2]{F91} for a proof in the
Euclidean setting). The proof in the setting of quasi-metric spaces is the same and,
therefore, we omit the details.

The following lemma is a refinement of \cite[Lemmas~3.6 and 3.8]{HIT16}, which, among other things,
sharpens the estimate \cite[(3.9)]{HIT16} in a fashion
that will permit us to simultaneously generalize \cite[Theorem 1.2]{HIT16} and \cite[Theorem~6]{hajlaszkt2} in a unified manner; see Theorem~\ref{measext} below.

\begin{lemma}
\label{intavgest}
Let $(X,\rho,\mu)$ be a quasi-metric measure space
with $\mu$ being a doubling measure, and suppose that $\Omega\subset X$
is a nonempty $\mu$-measurable set that satisfies the measure density condition
\eqref{measdens-INT-a}, namely,
there exists a constant $C_\mu\in(0,\infty)$  satisfying
\begin{equation}
\label{measdens-lemma}
\mu(B_\rho(x,r))\leq C_\mu\,\mu(B_\rho(x,r)\cap\Omega),
\quad\forall x\in\Omega,\ \ \forall\ r\in(0,1].
\end{equation}
Then the following statements are valid.
\begin{enumerate}
\item[{\rm(i)}] Suppose that $u\colon \Omega\to\mathbb{R}$ belongs locally to $\dot{M}^{\varepsilon,p}(\Omega,\rho,\mu)$ for some $\varepsilon,\,p\in(0,\infty)$,
in the sense that $u\in\dot{M}^{\varepsilon,p}(B\cap\Omega,\rho,\mu)$
for any fixed $\rho$-ball $B\subset X$ that is
centered in $\Omega$. Then
\begin{equation}
\label{SobDiff}
\lim\limits_{r\to 0^+}m_u(B_{\rho}(x,r)
\cap\Omega)=u(x)\quad\mbox{for $\mu$-almost every point $x\in \Omega$,}	
\end{equation}
here and thereafter, $r\to 0^+$ means $r\in(0,\fz)$ and $r\to0$.
In particular, \eqref{SobDiff} holds true if $u\in \dot{M}^s_{p,q}(\Omega,\rho,\mu)$ or $u\in\dot{N}^s_{p,q}(\Omega,\rho,\mu)$ for some $s,\,p\in(0,\infty)$  and $q\in(0,\infty]$.
	
\item[{\rm(ii)}] Fix $t,\, s,\,\varepsilon\in(0,\infty)$
with $\varepsilon<s$, and let $k_0\in\mathbb{N}$ be
any number such that
$2^{k_0}\geq C_\rho^2\widetilde{C}_\rho$, where $C_\rho$,
$\widetilde{C}_\rho\in[1,\infty)$ are as in \eqref{C-RHO.111} and \eqref{C-RHO.111XXX},
respectively.
Then, for any given $r_\ast\in[1,\infty)$, there exists a positive constant $C$ such that
\begin{align}
\label{BXPZ-2}
&\lf|m_u(B\cap\Omega)-m_u\lf(B_\rho(x,2^{-L})\cap\Omega\r)\r|\noz\\
&\qquad\leq C2^{-L\varepsilon}\left[\,\mvint_{B_\rho(x,C_\rho2^{-L})}\sup_{j\geq L-k_0}\lf\{2^{-j(s-\varepsilon)t}g_j^t\r\}\, d\mu\right]^{1/t}
\end{align}
for any  $\mu$-measurable function $u\colon  \Omega\to\mathbb{R}$, $\{g_k\}_{k\in\mathbb{Z}}\in\mathbb{D}^s_\rho(u)$, $x\in \Omega$,  $L\in\mathbb{Z}$
satisfying $2^{-L}\leq r_\ast$, and $B\subset B_\rho(x,2^{-L})$ which is a $\rho$-ball
centered in $\Omega$ such that $\mu(B\cap\Omega)\approx\mu(B_\rho(x,2^{-L})\cap\Omega)$.
Here, for any $j\in\mathbb{Z}$, the function $g_j$ is identified with function defined on all points of $X$ by setting $g_j\equiv0$ on $X\setminus\Omega$.
\end{enumerate}
\end{lemma}

\begin{remark}
We stress here that \eqref{SobDiff} holds true
without any additional regularity assumptions on the measure $\mu$.
If one assumes that the measure is, say, Borel regular, then one can rely on Lebesgue's differentiation theorem (\cite[Theorem~3.14]{AM15}) and Lemma~\ref{difflemma}  (with $\eta=1$) to show that \eqref{SobDiff} holds true whenever $u\in L^1_{\rm loc}(\Omega,\mu)$ (the set of all locally integrable functions on $\Omega$).	 	
\end{remark}

\begin{proof}[Proof of Lemma \ref{intavgest}]
Disposing of the claim in   (i) first, fix an exponent $\eta\in(0,p)$, a
point $x_0\in\Omega$, and let $B_k:=B_\rho(x_0,k)$ for any $k\in\mathbb{N}$.
Since $u$ belongs locally to $\dot{M}^{\varepsilon,p}(\Omega,\rho,\mu)$,
for any fixed $k\in\mathbb{N}$, we can find an $\varepsilon$-gradient
$g_k\in L^p(B_k\cap\Omega,\mu)$ of $u$ on the set $B_k\cap\Omega$.
Although $g_k$ is defined only on $B_k\cap\Omega$, we will
identify it with the function defined on all points of $X$ by setting $g_k\equiv0$
outside of $B_k\cap\Omega$. With this in mind, define the set
$$
E_k:=\lf\{x\in B_k\cap\Omega:\,g_k(x)=\infty\,\mbox{ or }
\,\mathcal{M}_{\rho_\#}(g_k^\eta)(x)=\infty\r\},
$$
where $\rho_\#$ is the regularized quasi-metric given by Theorem~\ref{DST1} and $\mathcal{M}_{\rho_\#}$ is
the Hardy-Littlewood maximal operator with
respect to   $\rho_\#$, which is defined by setting, for any  $f\in L^1_{\rm loc}(X,\mu)$,
$$\mathcal{M}_{\rho_\#}(f)(x):=\sup_{r\in(0,\fz)} \mvint_{B_{\rho_\#}(x,r)}
|f(y)|\,d\mu(y),\quad \forall\ x\in X.$$
By \cite[Theorem~3.7]{AM15}, we know that  the function
$ \mathcal{M}_{\rho_\#}(g_k^\eta) $ is $\mu$-measurable.
Also, $g_k^\eta\in L^{p/\eta}(X,\mu)$,
where $p/\eta>1$.  As such, it follows from the  boundedness of $\mathcal{M}_{\rho_\#}$
on $L^{p/\eta}(X,\mu)$
(see \cite[Theorem~3.7]{AM15}) that
$\mathcal{M}_{\rho_\#}(g_k^\eta)\in L^{p/\eta}(X,\mu)$.
This, together with the fact that $g_k\in L^p(B_k\cap\Omega,\mu)$,
implies that $\mu(E_k)=0$.

Suppose now that $x\in(C_\rho^{-1}B_k\cap\Omega)\setminus E_k$ for some $k\in\mathbb{N}$
and that $r\in(0,1]$ satisfies $r\leq C_\rho^{-1}$. Then $B_\rho(x,r)\subset B_k$ and, by combining \eqref{hcx-23}
with $\gamma:=u(x)$, the measure density condition \eqref{measdens-lemma},
and the fact that $g_k$ is an $\varepsilon$-gradient
for $u$ on $B_k\cap\Omega$, we conclude that
\begin{align}
\label{ynqq-34}
\lf|m_u(B_{\rho}(x,r)\cap\Omega)-u(x)\r|
&\leq\left[2\,\mvint_{B_{\rho}(x,r)\cap\Omega}|u(y)-u(x)|^\eta\, d\mu(y)\right]^{1/\eta}\noz\\
&\lesssim r^\varepsilon\left[\,\mvint_{B_{\rho}(x,r)}|g_k(y)|^\eta\, d\mu(y)\right]^{1/\eta} +r^\varepsilon g_k(x)\noz\\
&\lesssim r^\varepsilon\lf[\mathcal{M}_{\rho_\#}(g_k^\eta)(x)\r]^{1/\eta} +r^\varepsilon g_k(x).
\end{align}
Note that, in obtaining the third inequality in \eqref{ynqq-34}, we have relied on the fact that $\rho\approx\rho_\#$ implies $\mathcal{M}_{\rho}(g_k^\eta)\approx\mathcal{M}_{\rho_\#}(g_k^\eta)$, granted that $\mu$ is doubling. Given how the sets $E_k$ are defined, looking at the extreme most sides of \eqref{ynqq-34} and passing to the limit as $r\to0^+$ gives \eqref{SobDiff} for any point $x\in(C_\rho^{-1}B_k\cap\Omega)\setminus E_k$. Since we have $\Omega=\bigcup_{k=1}^\infty (C_\rho^{-1}B_k\cap\Omega)$ and $\mu\big(\bigcup_{k=1}^\infty E_k\big)=0$, the claim in \eqref{SobDiff}, as stated, now follows.

The fact that \eqref{SobDiff} holds true whenever $u\in \dot{M}^s_{p,q}(\Omega,\rho,\mu)$
or $u\in\dot{N}^s_{p,q}(\Omega,\rho,\mu)$ for some exponents $s,\,p\in(0,\infty)$,
and $q\in(0,\infty]$, is a consequence
of \cite[Proposition~2.4]{AYY21}, which gives the inclusions
$$
\dot{M}^s_{p,q}(\Omega,\rho,\mu)\hookrightarrow \dot{M}^s_{p,\infty}(\Omega,\rho,\mu)=\dot{M}^{s,p}(\Omega,\rho,\mu)
$$
and
$$
\dot{N}^s_{p,q}(B\cap\Omega,\rho,\mu)\hookrightarrow\dot{M}^{\varepsilon,p}(B\cap\Omega,\rho,\mu),
$$
whenever $B$ is a $\rho$-ball centered in $\Omega$ and $\varepsilon\in(0,s)$. This finishes the proof of (i).

Moving on to proving (ii), suppose that $u\colon \Omega\to\mathbb{R}$
is a $\mu$-measurable function, $\{g_k\}_{k\in\mathbb{Z}}\in\mathbb{D}^s_\rho(u)$,
$x\in \Omega$, $r_\ast\in[1,\infty)$, and $L\in\mathbb{Z}$
satisfies $2^{-L}\leq r_\ast$.
Recall that the doubling property of $\mu$
implies that $\mu$ is $Q$-doubling
with $Q:=\log_2C_{D}\in(0,\infty)$,
where $C_{D}\in(1,\infty)$ is the doubling constant for $\mu$
[see \eqref{doub}].
Let $\varepsilon'\in(0,\min\{\varepsilon,Q/t\})$ and
suppose that $B\subset B_\rho(x,2^{-L})$ is a
$\rho$-ball centered in $\Omega$ such that
$\mu(B\cap\Omega)\approx\mu(B_\rho(x,2^{-L})\cap\Omega)$.
Also, let $\gamma_0\in\mathbb{R}$ be such that
\begin{equation}
\label{xxz-482}
\left[\, \mvint_{B_\rho(x,2^{-L})\cap\Omega} |u-\gamma_0|^{t^\ast}\, d\mu\right]^{1/t^\ast}\leq2\inf_{\gamma\in\mathbb{R}}\left[\, \mvint_{B_\rho(x,2^{-L})\cap\Omega} |u-\gamma|^{t^\ast}\, d\mu\right]^{1/t^\ast},
\end{equation}
where $t^\ast:=Qt/(Q-\varepsilon' t)$. In concert,  \eqref{hcx-23} in Lemma~\ref{difflemma}
[used here with the induced quasi-metric measure space $(\Omega,\rho,\mu)$],
the measure density condition \eqref{measdens-lemma}, \eqref{xxz-482},
as well as \eqref{cie3-2} in Lemma~\ref{embeddfracgrad}, imply
\begin{align*}
&\lf|m_u\lf(B\cap\Omega\r)-m_u\lf(B_\rho(x,2^{-L})\cap\Omega\r)\r|\noz\\
&\quad\leq\lf|m_u\lf(B\cap\Omega\r)-\gamma_0\r|+\lf|\gamma_0-m_u\lf(B_\rho(x,2^{-L})\cap\Omega\r)\r|\noz\\
&\quad\leq\left[2\, \mvint_{B\cap\Omega} |u-\gamma_0|^{t^\ast}\,d\mu\right]^{1/t^\ast}+\left[2\, \mvint_{B_\rho(x,2^{-L})\cap\Omega} |u-\gamma_0|^{t^\ast}\,d\mu\right]^{1/t^\ast}\noz\\
&\quad\lesssim \left[\,\mvint_{B_\rho(x,2^{-L})\cap\Omega} |u-\gamma_0|^{t^\ast}\,d\mu\right]^{1/t^\ast}\lesssim 2^{-L\varepsilon}\left[\,\mvint_{B_\rho(x,C_\rho2^{-L})}\sup_{j\geq L-k_0}\lf\{2^{-j(s-\varepsilon)t}g_j^t\r\}\, d\mu\right]^{1/t}.
\end{align*}
This finishes the proof of (ii) and, in turn, the proof of  Lemma \ref{intavgest}.
\end{proof}

\begin{lemma}
\label{GVa2-prev}
Suppose that $(X,\rho,\mu)$ is a quasi-metric measure space and
fix an $\alpha\in(0,\infty)$, a $p\in(0,\infty)$, a $q\in(0,\infty]$, and an $s\in(0,\alpha]$,
where the value  $s=\alpha$ is only permissible when $q=\infty$.
Let $f\colon  X\to\mathbb{R}$ be a $\mu$-measurable function with
$\vec{h}:=\{h_k\}_{k\in\mathbb{Z}}\in\mathbb{D}^{s}_\rho(f)$, and
assume that $\Psi\in\dot{\mathscr{C}}^\alpha(X,\rho)$ is a
bounded function that vanishes pointwise outside of a $\mu$-measurable set $V\subset X$.
Then there exists a sequence $\vec{g}\in\mathbb{D}^{s}_\rho(\Psi f)$
satisfying
\begin{equation}
\label{gradest-prev}
\Vert\vec{g}\Vert_{L^p(X,\ell^q)}
\leq C\left[\|\Psi\|_{L^\infty(X)}\|\vec{h}\Vert_{L^p(V, \ell^q)}+\|\Psi\|^{s/\alpha}_{\dot{\mathscr{C}}^{\alpha}(X,\rho)}
\lf\{\|\Psi\|_{L^\infty(X)}+1\r\}\|f\Vert_{L^p(V, \mu)}\right]
\end{equation}
and
\begin{equation}
\label{gradest2-prev}
\Vert\vec{g}\Vert_{\ell^q(L^p(X))}\leq C\left[\|\Psi\|_{L^\infty(X)}\|\vec{h}\Vert_{\ell^q(L^p(V))}
+\|\Psi\|^{s/\alpha}_{\dot{\mathscr{C}}^{\alpha}(X,\rho)}
\lf\{\|\Psi\|_{L^\infty(X)}+1\r\}\|f\Vert_{L^p(V, \mu)}\right]
\end{equation}
for some positive constant  $C$ depending only on $s$, $p$, $q$, and $\alpha$.

Consequently, if $f\in M^s_{p,q}(V,\rho,\mu)$, then $\Psi f\in M^s_{p,q}(X,\rho,\mu)$ and
$$\Vert \Psi f\Vert_{M^s_{p,q}(X,\rho,\mu)}\lesssim
\Vert f\Vert_{M^s_{p,q}(V,\rho,\mu)},$$
where the implicit positive constant depends on $\Psi$, but is
independent of $f$. This last statement is also valid with $M^s_{p,q}$ replaced by $N^s_{p,q}$.
\end{lemma}

\begin{proof}
If $\|\Psi\|_{\dot{\mathscr{C}}^\alpha(X,\rho)}=0$, then $\Psi$ is constant. In this case, we take $g_k:=\|\Psi\|_{L^\infty(X)} h_k$ for any $k\in\mathbb{Z}$. Then it is easy to check that $\vec{g}:=\{g_k\}_{k\in\mathbb{Z}}\in\mathbb{D}^{s}_\rho(\Psi f)$ and satisfies both
\eqref{gradest-prev} and \eqref{gradest2-prev}.

Suppose next that $\|\Psi\|_{\dot{\mathscr{C}}^\alpha(X,\rho)}>0$ and let $k_\Psi\in\mathbb{Z}$ be the unique integer such that
\begin{equation}
\label{VN-2-prev}
2^{k_\Psi-1}\leq\|\Psi\|^{1/\alpha}_{\dot{\mathscr{C}}^\alpha(X,\rho)}<2^{k_\Psi}.
\end{equation}
For any  $k\in\mathbb{Z}$ and $x\in X$,  define
\begin{eqnarray*}
g_k(x):=\left\{
\begin{array}{ll}
\lf[|f(x)|\,2^{-k(\alpha-s)}\|\Psi\|_{\dot{\mathscr{C}}^\alpha(X,\rho)}
+h_k(x)\|\Psi\|_{L^\infty(X)}\r]\,{\mathbf 1}_{V}(x)\quad&\mbox{if}\,\,\,k\geq k_\Psi,
\\[6pt]
\lf[2^{ks+s+1}\,|f(x)|+h_k(x)\r]\|\Psi\|_{L^\infty(X)}\,{\mathbf 1}_{V}(x)&\mbox{if}\,\,\,k<k_\Psi.
\end{array}
\right.
\end{eqnarray*}
Note that each $g_k$ is $\mu$-measurable because $f$ and each $h_k$ are $\mu$-measurable,
by assumption. In order to show that this sequence is a fractional $s$-gradient
of $\Psi f$ (with respect to $\rho$), we fix a $k\in\mathbb{Z}$ and let
$y,\,z\in X$ be such that $2^{-k-1}\leq \rho(y,z)< 2^{-k}$.
To proceed, we first consider the case when $k\geq k_\Psi$.
If $y,\,z\in X\setminus V$, then $\Psi(y)=\Psi(z)=0$ and there is nothing to show.
If, on the other hand, $y\in V$ and $z\in X$, then we write
\begin{equation}
\label{rnp-28}
|\Psi(y) f(y)-\Psi(z) f(z)|\leq|f(y)| |\Psi(y)-\Psi(z)|+|\Psi(z)| |f(y)-f(z)|,
\end{equation}
and use the H\"older continuity of $\Psi$ and the fact
that $s\leq\alpha$ and ${\rm supp}\,\Psi\subset V$ in order to estimate
\begin{align*}
&|\Psi(y) f(y)-\Psi(z) f(z)|\\
&\quad\leq |f(y)|  [\rho(y,z)]^{\alpha}
\|\Psi\|_{\dot{\mathscr{C}}^\alpha(X,\rho)}
+\|\Psi\|_{L^\infty(X)} [\rho(y,z)]^{s}\lf[h_k(y)+h_k(z)\r]\,{\mathbf 1}_{V}(z)	
\\
&\quad= |f(y)| [\rho(y,z)]^{s}[\rho(y,z)]^{\alpha-s}\|\Psi\|_{\dot{\mathscr{C}}^\alpha(X,\rho)}
+\|\Psi\|_{L^\infty(X)}[\rho(y,z)]^{s}\lf[h_k(y)+h_k(z)\r]\,{\mathbf 1}_{V}(z)	
\\
&\quad\leq [\rho(y,z)]^{s}\lf\{|f(y)| 2^{-k(\alpha-s)}\,\|\Psi\|_{\dot{\mathscr{C}}^\alpha(X,\rho)}
+\|\Psi\|_{L^\infty(X)}\lf[h_k(y)+h_k(z)\r]\,{\mathbf 1}_{V}(z)\r\}
\\
&\quad\leq [\rho(y,z)]^{s}\lf[g_k(y)+g_k(z)\r].
\end{align*}	
The estimate when $y\in X$ and $z\in V$ is similar where,
in place of \eqref{rnp-28}, we use
\begin{equation*}
|\Psi(y) f(y)-\Psi(z) f(z)|\leq|f(z)| |\Psi(y)-\Psi(z)|+|\Psi(y)| |f(y)-f(z)|.
\end{equation*}
This finishes the proof of the case when $k\geq k_\Psi$.
Assume next that $k<k_\Psi$. As in the case $k\geq k_\Psi$,
it suffices to consider the scenario when $y\in V$ and $z\in X$.
Appealing to \eqref{rnp-28}, we conclude that
\begin{align*}
&|\Psi(y) f(y)-\Psi(z) f(z)|\\
&\quad\leq 2\,|f(y)| \|\Psi\|_{L^\infty(X)}+\|\Psi\|_{L^\infty(X)}[\rho(y,z)]^{s}\lf[h_k(y)+h_k(z)\r]\,{\mathbf 1}_{V}(z)	\\
&\quad= 2\,|f(y)| \|\Psi\|_{L^\infty(X)} [\rho(y,z)]^{s}[\rho(y,z)]^{-s}\\
&\qquad+\|\Psi\|_{L^\infty(X)}[\rho(y,z)]^{s}\lf[h_k(y)+h_k(z)\r]\,{\mathbf 1}_{V}(z)\\
&\quad\leq [\rho(y,z)]^{s}\|\Psi\|_{L^\infty(X)}
\lf\{2^{s(k+1)+1}|f(y)|+\lf[h_k(y)+h_k(z)\r]\,{\mathbf 1}_{V}(z)\r\}\\
&\quad\leq [\rho(y,z)]^{s}\lf[g_k(y)+g_k(z)\r].
\end{align*}	
This finishes the proof of the claim that
$\vec{g}:=\{g_k\}_{k\in\mathbb{Z}}\in\mathbb{D}^s_\rho(\Psi f)$.

We next show that $\vec{g}:=\{g_k\}_{k\in\mathbb{Z}}$
satisfies the desired estimates in \eqref{gradest-prev} and \eqref{gradest2-prev}.
If $q<\infty$, then [keeping in mind \eqref{VN-2-prev} and the
fact that $0<s<\alpha$ in this case]
\begin{align}
\label{xc1-prev}
\left(\sum_{k\in\mathbb{Z}}g_k^q\right)^{1/q}
&=\left\{\sum_{k=-\infty}^{k_\Psi-1}\|\Psi\|_{L^\infty(X)}^q\lf[2^{(ks+s+1)}|f|+h_k\r]^q\r.\noz\\
&\quad\lf.+\sum_{k=k_\Psi}^\infty\lf[|f|\,2^{-k(\alpha-s)}
\|\Psi\|_{\dot{\mathscr{C}}^\alpha(X,\rho)}+h_k\|\Psi\|_{L^\infty(X)}\r]^q\right\}^{1/q}\noz
\\
&\lesssim\|\Psi\|_{L^\infty(X)}\left(\sum_{k\in\mathbb{Z}}h_k^q
\right)^{1/q}
+\|\Psi\|_{L^\infty(X)}\,|f|\,\left[\sum_{k=-\infty}^{k_\Psi-1}2^{(ks+s+1)q}
\right]^{1/q}\noz\\
&\quad
+\|\Psi\|_{\dot{\mathscr{C}}^\alpha(X,\rho)}\,|f|\,
\left[\sum_{k=k_\Psi}^\infty2^{-k(\alpha-s)q}\right]^{1/q}\noz
\\
&\lesssim\|\Psi\|_{L^\infty(X)}\left(\sum_{k\in\mathbb{Z}}h_k^q
\right)^{1/q}+\|\Psi\|_{L^\infty(X)}\,|f|\,2^{k_\Psi s}
+\|\Psi\|_{\dot{\mathscr{C}}^\alpha(X,\rho)}\,|f|\,2^{-k_\Psi(\alpha-s)}\noz
\\
&\lesssim\|\Psi\|_{L^\infty(X)}\left(\sum_{k\in\mathbb{Z}}h_k^q
\right)^{1/q}+\|\Psi\|_{L^\infty(X)}\|\Psi\|_{\dot{\mathscr{C}}^\alpha(X,\rho)}^{s/\alpha}\,|f|
+\|\Psi\|_{\dot{\mathscr{C}}^\alpha(X,\rho)}^{s/\alpha}\,|f|.
\end{align}
The estimate in \eqref{gradest-prev} (for $q<\infty$) now follows from the estimate in \eqref{xc1-prev} and the fact that each $g_k$ is supported in $V$.
 To verify \eqref{gradest2-prev} when $q<\infty$, observe that
\begin{align*}
\Vert\vec{g}\Vert_{\ell^q(L^p(X))}
&=\left[\sum_{k\in\mathbb{Z}}\Vert g_k\Vert_{L^p(V, \mu)}^q\right]^{1/q}\noz
\\
&\lesssim\left[\sum_{k=-\infty}^{k_\Psi-1}\|\Psi\|_{L^\infty(X)}^q\left\Vert 2^{(ks+s+1)}|f|+h_k\right\Vert_{L^p(V, \mu)}^{q}\right.\noz
\\
&\left.\quad+\sum_{k=k_\Psi}^\infty\left\Vert|f|\,2^{-k(\alpha-s)}
\|\Psi\|_{\dot{\mathscr{C}}^\alpha(X,\rho)}+h_k\|\Psi\|_{L^\infty(X)}
\right\Vert_{L^p(V, \mu)}^{q}\right]^{1/q}\noz
\\
&\lesssim\|\Psi\|_{L^\infty(X)}\|\vec{h}\Vert_{L^p(V, \ell^q)}
+\|\Psi\|_{L^\infty(X)}\Vert f\Vert_{L^p(V, \mu)}
\left[\sum_{k=-\infty}^{k_\Psi-1}2^{(ks+s+1)q}\right]^{1/q}\noz
\\
&\quad+\|\Psi\|_{\dot{\mathscr{C}}^\alpha(X,\rho)}\Vert f\Vert_{L^p(V, \mu)}\left[\sum_{k=k_\Psi}^\infty2^{-k(\alpha-s)q}\right]^{1/q}\noz
\\
&\lesssim\|\Psi\|_{L^\infty(X)}\|\vec{h}\Vert_{L^p(V, \ell^q)}
+\|\Psi\|_{L^\infty(X)}\Vert f\Vert_{L^p(V, \mu)}\,2^{k_\Psi s}
+\|\Psi\|_{\dot{\mathscr{C}}^\alpha(X,\rho)}\Vert f\Vert_{L^p(V, \mu)}
\,2^{-k_\Psi(\alpha-s)}\noz
\\
&\lesssim\|\Psi\|_{L^\infty(X)}\|\vec{h}\Vert_{L^p(V, \ell^q)}
+\|\Psi\|^{s/\alpha}_{\dot{\mathscr{C}}^{\alpha}(X,\rho)}\lf\{\|\Psi\|_{L^\infty(X)}+1\r\}\|f\Vert_{L^p(V, \mu)},
\end{align*}
as wanted.
The proof of  \eqref{gradest-prev} and  \eqref{gradest2-prev} when $q=\infty$ follow along a similar line of reasoning.
Note that the choice of $s=\alpha$ is permissible in this case because, in this scenario, the summation $[\sum_{k=k_\Psi}^\infty2^{-k(\alpha-s)q}]^{1/q}$ is
replaced by $\sup_{k\geq k_\Psi}2^{-k(\alpha-s)}=1$. This finishes the proof of  Lemma \ref{GVa2-prev}.
\end{proof}

We end this section by recalling the following inequality established in \cite[Lemma~3.1]{HIT16}.

\begin{lemma}
\label{heli-est}
Let $a\in(1,\infty)$ and $b\in(0,\infty)$.
Then there exists a positive constant $C=C(a,b)$
such that, for any sequence $\{c_k\}_{k\in\mathbb{Z}}$
of nonnegative real numbers,
$$
\sum\limits_{k\in\mathbb{Z}}\left(\sum\limits_{j\in\mathbb{Z}}a^{-|j-k|}c_j\right)^b\leq
C\sum\limits_{j\in\mathbb{Z}}c_j^b.
$$ 	
\end{lemma}	

\subsection{Principal Results}
\label{sssec:extensions2}

At this stage, we are ready to construct a Whitney-type extension operator  for the spaces $M^s_{p,q}$ and $N^s_{p,q}$ on domains satisfying the so-called measure density condition for an optimal range of $s$. Before stating this result, we recall the following piece of notational convention.

\begin{convention}
\label{sindex}
Given a quasi-metric space $(X,\rho)$ and fixed numbers $s\in(0,\infty)$ and $q\in(0,\infty]$, we will understand by $s\preceq_q{\rm ind}\,(X,\rho)$ that $s\leq{\rm ind}\,(X,\rho)$ and that the value $s={\rm ind}\,(X,\rho)$ is only permissible when $q=\infty$ and the supremum defining ${\rm ind}\,(X,\rho)$ in \eqref{index} is attained.
\end{convention}

Here is the extension result alluded to above.

\begin{theorem}
\label{measext}
Let $(X,\rho,\mu)$ be a quasi-metric measure space,
where $\mu$ is a Borel regular $Q$-doubling measure on $X$
for some $Q\in(0,\infty)$, and fix exponents $s,\,p\in(0,\infty)$ and a $q\in(0,\infty]$, where $s\preceq_q{\rm ind}\,(X,\rho)$.
Also, suppose that $\Omega\subset X$
is a nonempty $\mu$-measurable set that satisfies the   measure density condition \eqref{measdens-INT-a}, that
is, there exists a positive  constant $C_\mu$ such that
\begin{equation}
\label{measdens}
\mu(B_\rho(x,r))\leq C_\mu\,\mu(B_\rho(x,r)\cap\Omega)
\quad\mbox{for any $x\in\Omega$ and   $r\in(0,1]$.}
\end{equation}
Then there exists a positive constant $C$ such that,
\begin{eqnarray}
\label{II+kan}
\begin{array}{c}
\mbox{for any }\,u\in M^s_{p,q}(\Omega,\rho,\mu),
\,\,\mbox{ there exists a }\ \widetilde{u}\in M^s_{p,q}(X,\rho,\mu)
\\[6pt]
\mbox{ for which }\,\,\, u=\widetilde{u}|_{\Omega}\,\,
\mbox{ and }\,\,\,\,\|\widetilde{u}\|_{M^s_{p,q}(X,\rho,\mu)}
\leq C\|u\|_{M^s_{p,q}(\Omega,\rho,\mu)}.
\end{array}
\end{eqnarray}
Furthermore, if $p,\,q>Q/(Q+s)$, then there is a  linear and bounded operator
$$\mathscr{E}\colon  M^s_{p,q}(\Omega,\rho,\mu)\to M^s_{p,q}(X,\rho,\mu)$$
such that
$(\mathscr{E}u)|_{\Omega}=u$ for any $u\in M^s_{p,q}(\Omega,\rho,\mu)$.
In addition, if $s<{\rm ind}\,(X,\rho)$, then all of the statements above are also valid with $M^s_{p,q}$ replaced by $N^s_{p,q}$.
\end{theorem}

\begin{remark}
In Theorem~\ref{measext}, we only assume that $\mu$ is Borel regular to ensure that
the Lebesgue differentiation theorem holds true.
It is instructive to note that one could assume a
weaker regularity condition on $\mu$, namely, the Borel semi-regularity,
which turns out to be equivalent to the availability of the Lebesgue differentiation theorem; see \cite[Theorem~3.14]{AM15}.
\end{remark}

\begin{remark}
The linear extension operator constructed in Theorem~\ref{measext} (when $p$ and $q$ are large enough) is a Whitney-type operator and it is universal
(in the sense that it simultaneously preserves all orders of smoothness).
\end{remark}

\begin{proof}[Proof of Theorem~\ref{measext}]
We begin by making a few important observations. First, note that we
can assume  $\Omega$ is a closed set. Indeed,
the measure density condition \eqref{measdens}, when put in concert with
the Lebesgue differentiation theorem, implies that $\mu(\partial\Omega)=0$ (see, for instance,
\cite[Lemma~2.1]{S07}). Hence, $M^s_{p,q}\big(\overline{\Omega}\big)=M^s_{p,q}(\Omega)$ and $N^s_{p,q}(\overline{\Omega})=N^s_{p,q}(\Omega)$
because the membership to the spaces $M^s_{p,q}$ and $N^s_{p,q}$
is defined up to a set of measure zero.
Moreover, we can assume that $\Omega\neq X$ as the claim is trivial when $\Omega=X$.

Moving on, by  $s\preceq_q{\rm ind}\,(X,\rho)$ and its meaning,  there exists a quasi-metric $\varrho$
on $X$ such that $\varrho\approx\rho$  and $s\leq(\log_{2}C_\varrho)^{-1}$, where  $C_\varrho\in[1,\infty)$ is as in \eqref{C-RHO.111}, and the value  $s=(\log_{2}C_\varrho)^{-1}$ can only occur when $q=\infty$ and $C_\varrho>1$.
Next, let $\varrho_{\#}$ be the regularized quasi-metric given by Theorem~\ref{DST1}. Then $\varrho_\#\approx\varrho$ and $C_{\varrho_\#}\leq C_\varrho$. Thus, $s\leq(\log_{2}C_{\varrho_\#})^{-1}$, where the value  $s=(\log_{2}C_{\varrho_\#})^{-1}$ can only occur when $q=\infty$ and $C_{\varrho_\#}>1$.

Recall that, by Theorem~\ref{DST1}, all $\varrho_\#$-balls are $\mu$-measurable. Combining this with the fact that $\varrho_\#\approx\rho$, we conclude that
$(X,\varrho_\#,\mu)$ is a quasi-metric measure space where
$\mu$ is doubling with respect to $\varrho_\#$-balls.
As such, we can consider the decomposition of the open set
$\mathcal{O}:=X\setminus\Omega$ into the family $\{B_j\}_{j\in\mathbb{N}}$ of ${\varrho_\#}$-balls, where,
for any given $j\in\nn$, $B_j:=B_{\varrho_\#}(x_j,r_j)$ with $x_j\in \mathcal{O}$ and $r_j\in(0,\fz)$,
as given by Theorem~\ref{L-WHIT} with $\theta:=2C^2_{\varrho_\#}$. Also, fix  $\theta'\in(C_{\varrho_\#},\theta/C_{\varrho_\#})=(C_{\varrho_\#},2C_{\varrho_\#})$
and choose any finite number $\alpha\in[s,(\log_{2}C_{\varrho_\#})^{-1}]$,
where $\alpha\neq s$ unless $s=(\log_{2}C_{\varrho_\#})^{-1}$.
In this context, let $\{\psi_j\}_{j\in\mathbb{N}}$ be the associated partition of unity of
order $\alpha$ given by Theorem~\ref{MSz7b} [applied here with the space $(X,\varrho_\#,\mu)$].

By Theorem~\ref{L-WHIT}(iii), for any
$j\in\mathbb{N}$, there exists an  $x_j^\ast\in\Omega$   satisfying
${\varrho_\#}(x_j,x_j^\ast)<\Lambda r_j$, where $\Lambda\in(\theta,\fz)$
is as in Theorem~\ref{L-WHIT}.
Let
\begin{equation}\label{ball-is}
B_j^\ast:=B_{\varrho_\#}(x_j^\ast,r_j),\quad \forall \ j\in\mathbb{N},
\end{equation}
and, for any
$x\in X\setminus\Omega$,
define
$$r(x):={\rm dist}_{\varrho_\#}(x,\Omega)/(4C_{\varrho_\#})\quad \mbox{and}\quad
B_x:=B_{\varrho_\#}(x,\Lambda^2 r(x)),$$
where ${\rm dist}_{\varrho_\#}(x,\Omega):=\inf_{y\in \Omega} \varrho_\#(x,y)$.
We claim that
\begin{equation}
\label{qwpk-1}
B_j^\ast\subset B_x\subset C_{\varrho_\#}^2\Lambda^3 B_j^\ast\quad\mbox{whenever $j\in\mathbb{N}$ and $x\in \theta' B_j$.}
\end{equation}
To show \eqref{qwpk-1}, we   first prove that
\begin{equation}
\label{qwpk-1-radii}
\frac{r_j}{2}\leq r(x)\leq\frac{\Lambda r_j}{4}\quad\mbox{whenever $j\in\mathbb{N}$ and $x\in \theta' B_j$.}
\end{equation}
Indeed, for any $j\in\mathbb{N}$ and $x\in \theta' B_j$, by $x_j^\ast\in\Omega$ and
$\Lambda>\theta>{C}_{\varrho_\#}\theta'$, we find that (keeping in mind that $\varrho_\#$
is symmetric)
\begin{align*}
{\rm dist}_{\varrho_\#}(x,\Omega)
&\leq\varrho_\#(x,x_j^\ast)
\leq C_{\varrho_\#}\max\lf\{\varrho_\#(x,x_j),\varrho_\#(x_j,x_j^\ast)\r\}\\
&<C_{\varrho_\#}\max\lf\{\theta'r_j,\Lambda r_j\r\}
=C_{\varrho_\#}\Lambda r_j,
\end{align*}
which implies the second inequality in
\eqref{qwpk-1-radii}. To prove the
 first inequality in \eqref{qwpk-1-radii},
 fix an arbitrary $z\in\Omega$. By Theorem~\ref{L-WHIT}(iii),
 we have $\theta B_j\subset X\setminus\Omega$.
Therefore, $z\not\in \theta B_j$, which, together with $x\in \theta'B_j$ and $\theta>{C}_{\varrho_\#}\theta'$,
implies that
\begin{align}
\label{472-4ni}
\theta r_j
&\leq\varrho_\#(x_j,z)\leq C_{\varrho_\#}\max\{\varrho_\#(x_j,x),\varrho_\#(x,z)\}\nonumber\\
&<\max\{{C}_{\varrho_\#}\theta'r_j,{C}_{\varrho_\#}\varrho_\#(x,z)\}\nonumber\\
&<\max\{\theta r_j,{C}_{\varrho_\#}\varrho_\#(x,z)\}=C_{\varrho_\#}\varrho_\#(x,z).
\end{align}
Given that $z\in\Omega$ was arbitrary, by looking at the extreme
most sides of \eqref{472-4ni} and taking the infimum over all $z\in \Omega$, we  conclude that
\begin{equation}\label{est-O}
2C^2_{\varrho_\#}r_j=\theta r_j\leq C_{\varrho_\#}{\rm dist}_{\varrho_\#}(x,\Omega),\quad \forall\ x\in\theta' B_j,
\end{equation}
and the first inequality in \eqref{qwpk-1-radii} follows. This finishes the proof of \eqref{qwpk-1-radii}.

Returning to the proof of \eqref{qwpk-1}, we
fix again a $j\in\mathbb{N}$ and an $x\in \theta' B_j$.
By \eqref{qwpk-1-radii}, the choice of $\theta:=2C_{\varrho_\#}^2$, and the fact that $\Lambda>\theta>\theta'$,
we find that, for any $z\in B_j^\ast$,
\begin{align*}
\varrho_\#(x,z)
&\leq C_{\varrho_\#}^2
\max\lf\{\varrho_\#(x,x_j),\varrho_\#(x_j,x_j^\ast),\varrho_\#(x_j^\ast,z)\r\}
\\
&<C_{\varrho_\#}^2\max\lf\{\theta'r_j,\Lambda r_j,r_j\r\}
=C_{\varrho_\#}^2\Lambda r_j\leq 2C_{\varrho_\#}^2\Lambda r(x)<\Lambda^2 r(x).
\end{align*}
This implies $B_j^\ast\subset B_x$.

Next, we show  $B_x\subset C_{\varrho_\#}^2\Lambda^3 B_j^\ast$. Indeed,
by \eqref{qwpk-1-radii}, one has, for any $z\in B_x$,
\begin{align*}
\varrho_\#(x_j^\ast,z)
&\leq C_{\varrho_\#}^2\max\lf\{\varrho_\#(x_j^\ast,x_j),\varrho_\#(x_j,x),\varrho_\#(x,z)\r\}
\\
&<C_{\varrho_\#}^2\max\lf\{\Lambda r_j,\theta'r_j,\Lambda^2 r(x)\r\}
\\
&\leq C_{\varrho_\#}^2\max\lf\{\Lambda r_j,\theta'r_j,\Lambda^3 r_j\r\}
=C_{\varrho_\#}^2\Lambda^3 r_j.
\end{align*}
Thus, $B_x\subset C_{\varrho_\#}^2\Lambda^3 B_j^\ast$.
 This finishes the proof of \eqref{qwpk-1}.

Moving on, define $J:=\{j\in\mathbb{N}:\ r_j\leq1\}$. It follows from \eqref{qwpk-1}, the definition of $J$,  the measure density condition \eqref{measdens},  and the doubling property of $\mu$ that
\begin{equation}
\label{qwpk-2}
\mu(B_j^\ast\cap\Omega)\gtrsim \mu(B_j^\ast)\approx\mu(B_x),
\quad\mbox{whenever $j\in J$ and $x\in \theta' B_j$.}
\end{equation}

Our plan below is to show that each function belonging to
$M^s_{p,q}(\Omega,\varrho_\#,\mu)$ can be extended to the
entire space $X$ with preservation of smoothness and while
retaining control of the associated `norm'.
Then the conclusion of this theorem will follow from
Proposition~\ref{equivspaces}, which
implies $M^s_{p,q}(\Omega,\varrho_\#,\mu)=M^s_{p,q}(\Omega,\rho,\mu)$
with equivalent `norms'.
To this end, fix a $u\in M^s_{p,q}(\Omega,\varrho_\#,\mu)$ and choose a $\vec{g}:=\{g_k\}_{k\in\mathbb{Z}}\in\mathbb{D}^s_{\varrho_\#}(u)$ satisfying
$\Vert\vec{g}\Vert_{L^p(\Omega,\ell^q)}\lesssim \Vert u\Vert_{\dot{M}^s_{p,q}(\Omega)}$.
Note that such a choice is possible, due to Proposition~\ref{constant}.	
Although, for any $k\in\zz$,  $g_k$ and $u$ are defined only in $\Omega$,
we will identify them as a function defined on all
of $X$ by setting $u\equiv g_k\equiv0$ on $X\setminus\Omega$.

We  first extend the function $u$ to the
following neighborhood of $\Omega$:
\begin{equation}\label{neib}
V:=\lf\{x\in X:\ {\rm dist}_{\varrho_\#}(x,\Omega)<2C_{\varrho_\#}\r\}.
\end{equation}
To this end, for each $x\in V\setminus\Omega$, we let
$I_x$ denote the collection of all $j\in\mathbb{N}$ such
that $x\in \theta'B_j$.  Since $\theta'<\theta$,
by   Theorem~\ref{L-WHIT}(ii),
we have  $\#I_x\leq M$, where $M$ is a positive integer depending on $\theta$ and the space $(X,\varrho_{\#},\mu)$.
Moreover, if $j\in\mathbb{N}\setminus J$,
then it follows from \eqref{est-O} and the definition of $J$
that $${\rm dist}_{\varrho_\#}(\theta'B_j,\Omega)\geq 2C_{\varrho_\#}r_j\geq 2C_{\varrho_\#},$$
and hence  $\theta'B_j\cap V=\emptyset$ and $j\not\in I_x$.
As such, we have $I_x\subset J$. Observe that, if $j\not\in I_x$, then, by
Theorem~\ref{MSz7b}(ii), one has  $\psi_j(x)=0$ which, together with
Theorem~\ref{MSz7b}(iii), implies that
\begin{equation}
\label{qwpk-3}
\sum_{j\in I_x}\psi_j(x)=\sum_{j\in J}\psi_j(x)=\sum_{j\in\mathbb{N}}\psi_j(x)=1,\quad\forall\ x\in V\setminus\Omega.
\end{equation}

Define the local extension, $\mathcal{F}u\colon  X\to\mathbb{R}$, of $u$ to $V$ by setting
\begin{eqnarray}
\label{localextdef}
\mathcal{F}u(x):=\left\{
\begin{array}{ll}
u(x)\quad &\mbox{ if $x\in\Omega$,}\\[6pt]
\displaystyle\sum_{i\in\mathbb{N}}\psi_i(x)\,m_u(B^\ast_i\cap\Omega)
&\mbox{ if $x\in X\setminus\Omega$,}
\end{array}
\right.\quad
\end{eqnarray}
where $m_u$ is as in Definition \ref{median}.
Given the presence of the median value factor $m_u(B^\ast_i\cap\Omega)$ in \eqref{localextdef}, it is not to be expected that $\mathcal{F}$ is linear. However, if $p,\,q>Q/(Q+s)$,
then we can construct a linear extension as follows:
First, observe that $u\in L^1(B^\ast_i\cap\Omega)$ for any $j\in J$.
Indeed, if $\Omega$ contains only one point then this claim is obvious. Otherwise,
the measure density condition \eqref{measdens} and
the doubling property for $\mu$, ensure that $(\Omega,\varrho_\#,\mu)$
is a quasi-metric measure space, where $\mu$ is $Q$-doubling
on domain $\varrho_\#$-balls up to scale $C_{\varrho_\#}$ .
Here, $\varrho_\#$ is naturally restricted to the set $\Omega$.
Let $r:=Q/(Q+s)<Q/s$. Then $r<p$ and by $u\in M^s_{p,q}(\Omega,\varrho_\#,\mu)$ and
the H\"older inequality, we have $u\in M^s_{r,q}(\Omega,\varrho_\#,\mu)$. Hence,
$u\in L^{r^\ast}(B^\ast_i\cap\Omega)=L^1(B^\ast_i\cap\Omega)$, where $r^\ast:=Qr/(Q-sr)$, by Theorem~\ref{DOUBembedding}(a)
[applied here for the space $(\Omega,\varrho_\#,\mu)$ and
with $r_\ast:=\sigma:=C_{\varrho_\#}$]. \footnote{Similarly,
one can show that restrictions of functions in $N^s_{p,q}(\Omega)$
belong to $L^1(B^\ast_i\cap\Omega)$ by choosing  an $\varepsilon\in(0,s)$
so that $\tilde{r}:=Q/(Q+\varepsilon)<p$ and then using
Theorem~\ref{mainembedding-epsilon}(a) in place of Theorem~\ref{DOUBembedding}(a).}
As such, we can replace the median value  $m_u(B^\ast_i\cap\Omega)$ in \eqref{localextdef}
with the integral average $u_{B^\ast_i\cap\Omega}$.
The resulting operator $\mathcal{F}$ is linear and we can
show that it is also a local extension of $u$ by using the estimate
in Remark~\ref{linearlocal} [with $0<\varepsilon'<\varepsilon<s$ chosen
so that $p,\,q>Q/(Q+\varepsilon')=:t$] in place of \eqref{BXPZ-2}
in the proof that follows here. We omit  the details.

We now prove that $\mathcal{F}u\in M^s_{p,q}(V)$ with $\|\mathcal{F}u\|_{M^s_{p,q}(V,\mu)}
\lesssim\|u\|_{M^s_{p,q}(\Omega,\mu)}$ by establishing the following three lemmas.

\begin{lemma}
\label{claim1}
Let all the notation be as in the above proof of   Theorem~\ref{measext}.
Then $\mathcal{F}u\in L^p(V,\mu)$ and $\|\mathcal{F}u\|_{L^p(V,\mu)}
\leq C\|u\|_{L^p(\Omega,\mu)}$, where $C$ is a
positive constant independent of $u$, $V$, and $\Omega$.
\end{lemma}
\begin{proof}
First observe that $\mathcal{F}u$ is $\mu$-measurable, because
$u$ is $\mu$-measurable and, by Theorem~\ref{L-WHIT}(ii), and both (i) and
(ii) of Theorem~\ref{MSz7b}, the sum in \eqref{localextdef}
is a \textit{finite} linear combination of continuous functions nearby
each point in $X\setminus\Omega$.

Let $x\in X\setminus\Omega$ and fix an
$\eta\in(0,\min\{1,p\})$. Combining \eqref{qwpk-3}, \eqref{hcx-23} in Lemma~\ref{difflemma} with $\gamma=0$,
\eqref{qwpk-1}, \eqref{qwpk-2},
the fact that $\#I_x\leq M$,
and this choice of $\eta$, we conclude that
\begin{align*}
\left\vert\mathcal{F}u(x)\right\vert
&\leq\sum_{j\in I_x}\psi_j(x)\,\left\vert m_u(B^\ast_j\cap\Omega)\right\vert
\lesssim\sum_{j\in I_x}\left(\,\mvint_{B^\ast_j\cap\Omega}|u|^\eta\, d\mu\right)^{1/\eta}
\\
&\lesssim\left(\,\mvint_{B_x}|u|^\eta\, d\mu\right)^{1/\eta}
\lesssim\lf[\mathcal{M}_{\varrho_\#}\lf(|u|^\eta\r)(x)\r]^{1/\eta},
\end{align*}
which, together with
the definition of $\mathcal{F}u$ and the boundedness of the Hardy-Littlewood maximal operator on $L^{p/\eta}(X)$,
implies the desired estimate $\|\mathcal{F}u\|_{L^p(V,\mu)}\lesssim\|u\|_{L^p(\Omega,\mu)}$.
This finishes the proof of Lemma \ref{claim1}.
\end{proof}

\begin{lemma}
\label{claim2}
Let all the notation be as in the above proof of   Theorem~\ref{measext}.
Fix an $\varepsilon\in(0,s)$ and choose any number $\delta\in(0,\min\{\alpha-s,s-\varepsilon\})$
if $\alpha\neq s$, and set $\delta=0$ if $\alpha=s$. Suppose that  $t\in
(0,\min\{p,q\})$ and let $k_0\in\mathbb{Z}$
be such that $2^{k_0-1}\leq 16C_{\varrho_\#}^4\Lambda^2<2^{k_0}$.
In this context, define a sequence
$\vec{h}:=\{h_k\}_{k\in\mathbb{Z}}$ by
setting, for any $k\in\mathbb{Z}$ and  $x\in X$,
\begin{eqnarray}
\label{localgraddef}
h_k(x):=
\left\{
\begin{array}{ll}
\displaystyle \lf[\mathcal{M}_{\varrho_\#}\lf(\sup_{j\in\mathbb{Z}}
\lf\{2^{-|k-j|\delta t}g_j^t\r\}\r)(x)\r]^{1/t}
\quad&\mbox{if\ $k\geq k_0$,}\\[20pt]
2^{(k+1)s}|\mathcal{F}u(x)| &\mbox{if\ $k<k_0$.}
\end{array}
\right.
\end{eqnarray}
Then there exists a positive constant $C$, independent of $u$,  $\vec{g}$, $\vec{h}$, and $\Omega$,  such that $\{Ch_k\}_{k\in\mathbb{Z}}$
is an $s$-fractional gradient of $\mathcal{F}u$ on $V$ with respect to ${\varrho_\#}$.
\end{lemma}

\begin{remark}
The sequence in \eqref{localgraddef} is different from than one considered in \cite[(5.3)]{HIT16}, and it is this new sequence defined in \eqref{localgraddef} that allows us to generalize \cite[Theorem 1.2]{HIT16} and \cite[Theorem~6]{hajlaszkt2} simultaneously and in a unified manner.
\end{remark}

\begin{proof}[Proof of Lemma~\ref{claim2}]
We first prove that, for any $k\in\zz$, $h_k$ is a well-defined function.
Indeed, by the choice  of $\vec g$, one has  $\Vert\vec{g}(\,\cdot\,)\Vert_{\ell^q}\in L^p(X,\mu)$.
Then, when $q<\infty$, we have  $\delta>0$ and, by the H\"older inequality (used with the
exponent $r:=q/t>1$),
\begin{equation*}
\sup_{j\in\mathbb{Z}}\lf\{2^{-|k-j|\delta t}g_j^t\r\}\leq\sum_{j\in\mathbb{Z}}2^{-|k-j|\delta t}g_j^t
\leq\left(\sum_{j\in\mathbb{Z}}2^{-|k-j|\delta tr'}\right)^{1/r'}
\left(\sum_{j\in\mathbb{Z}}g_j^q\right)^{t/q}\in L^{p/t}(X,\mu);
\end{equation*}
while, when $q=\infty$, one also has
\begin{equation*}
\sup_{j\in\mathbb{Z}}\lf\{2^{-|k-j|\delta t}g_j^t\r\}
\leq\lf(\sup_{j\in\mathbb{Z}}g_j\r)^t\in L^{p/t}(X,\mu).
\end{equation*}
Altogether, we have proved that $\sup_{j\in\mathbb{Z}}\{2^{-|k-j|\delta t}g_j^t\}\in L^{p/t}(X,\mu)$.
Consequently, by $p/t>1$, we conclude
that $\sup_{j\in\mathbb{Z}}\{2^{-|k-j|\delta t}g_j^t\}$ is locally integrable on $X$.
From this and Lemma \ref{claim1}, we further deduce that $h_k$ is a well-defined
$\mu$-measurable function on $X$ \footnote{Note that the consideration of the Hardy-Littlewood
maximal operator with respect to the regularized quasi-metric, $\varrho_\#$, ensures that the function $\mathcal{M}_{\varrho_\#}(\sup_{j\in\mathbb{Z}}\{2^{-|k-j|\delta t}g_j^t\})$ is $\mu$-measurable
(see \cite[Theorem~3.7]{AM15}).}.
	
Moving on, to prove Lemma~\ref{claim2},
we need to show that there exists a positive constant $C$
such that,  for any fixed $k\in\mathbb{Z}$, the following inequality holds true,
for $\mu$-almost every $x,\,y\in V$ satisfying $2^{-k-1}\leq {\varrho_\#}(x,y)<2^{-k}$,
\begin{equation}
\label{fracgradext}
\lf|\mathcal{F}u(x)-\mathcal{F}u(y)\r|\leq [{\varrho_\#}(x,y)]^s\lf[Ch_k(x)+Ch_k(y)\r].
\end{equation}
To this end, fix $x,\,y\in V$ and suppose that $2^{-k-1}\leq{\varrho_\#}(x,y)<2^{-k}$ for some $k\in\mathbb{Z}$.
If $k< k_0$, then
\begin{equation*}
\begin{split}
|\mathcal{F}u(x)-\mathcal{F}u(y)|
&\leq|\mathcal{F}u(x)|+|\mathcal{F}u(y)|\\
&\leq2^{(k+1)s}[{\varrho_\#}(x,y)]^s\lf[|\mathcal{F}u(x)|+|\mathcal{F}u(y)|\r]\\
&=[{\varrho_\#}(x,y)]^s\lf[h_k(x)+h_k(y)\r],
\end{split}
\end{equation*}
which is a desired estimate in this case.

Suppose next that $k\geq k_0$. By the choice of $k_0$, we actually have  $k\geq k_0\geq5$.
Moreover, since $\varrho_\#$ is symmetric (that is, $\widetilde{C}_{\varrho_\#}=1$),
our choice of $k_0$ ensures that $2^{k_0}>2C_{\varrho_\#}=C_{\varrho_\#}^2\widetilde{C}_{\varrho_\#}$.
We proceed by considering four cases based on the locations of both  $x$ and $y$.

\vspace{0.2cm}
\noindent{\bf CASE 1}: $x,\,y\in\Omega$.
\vspace{0.2cm}

In this case, without loss of generality,  we may assume that $x$ and
$y$ satisfy the inequality \eqref{Hajlasz} with $\rho$ therein replaced by
$\varrho_\#$,  as \eqref{Hajlasz} holds true
$\mu$-almost everywhere in $\Omega$. Recall that we have already
shown that $\sup_{j\in\mathbb{Z}}\{2^{-|k-j|\delta t}g_j^t\}$ is locally integrable.
As such, in light of the Lebesgue differentiation theorem (\cite[Theorem~3.14]{AM15}),
we can further assume that both
$x$ and $y$ are Lebesgue points of $\sup_{j\in\mathbb{Z}}\{2^{-|k-j|\delta t}g_j^t\}$.
For any $k\in\zz$, since $g_k^t\leq\sup_{j\in\mathbb{Z}}\{2^{-|k-j|\delta t}g_j^t\}$,
by appealing once
again to the Lebesgue differentiation theorem, we have $g_k\leq h_k$  pointwise $\mu$-almost
everywhere on $X$ which, in turn, allows us to conclude that
\begin{equation*}
|\mathcal{F}u(x)-\mathcal{F}u(y)|=|u(x)-u(y)|\leq [{\varrho_\#}(x,y)]^s\lf[h_k(x)+h_k(y)\r].
\end{equation*}
This is a desired estimate in Case 1.

\vspace{0.2cm}
\noindent{\bf CASE 2}: $x\in V\setminus\Omega$ and $y\in\Omega$.
\vspace{0.2cm}

In this case, we have
\begin{equation}\label{est-rx}
r(x):={\rm dist}_{\varrho_\#}(x,\Omega)/(4C_{\varrho_\#})
\leq{\varrho_\#}(x,y)/4<2^{-k-2}.
\end{equation}
Recall that $\Lambda^2>\theta^2\geq4C_{\varrho_\#}$,
where $\theta:=2C^2_{\varrho_\#}$ and $\Lambda\in(\theta,\fz)$
is as in Theorem~\ref{L-WHIT}.
Then  it follows from the definition of $B_x:=B_{\varrho_\#}(x,\Lambda^2 r(x))$
that we can find a point $x^\ast\in B_x\cap\Omega$. Let $m\in\mathbb{Z}$
be such that $2^{-m-1}\leq C_{\varrho_\#}\Lambda^2 r(x)<2^{-m}$ and set
$B_{x^\ast}:=B_{\varrho_\#}(x^\ast,2^{-m})$. The existence of such an $m\in\mathbb{Z}$
is  guaranteed by the fact that  $x$ belongs to the open set $X\setminus\Omega$, which implies that $r(x)>0$.
In order to obtain \eqref{fracgradext}, we write
\begin{equation}
\label{xzp-12}
\lf|\mathcal{F}u(x)-\mathcal{F}u(y)\r|\leq
\lf|\mathcal{F}u(x)-m_u(B_{x^\ast}\cap\Omega)\r|+\lf|m_u(B_{x^\ast}\cap\Omega)-u(y)\r|.
\end{equation}

We  now separately estimate the two terms appearing
in the right-hand side of \eqref{xzp-12}. For the first term,
using the fact that $\sum_{i\in I_x}\psi_i(x)=1$ in \eqref{qwpk-3} and the definition of
$\mathcal{F}u$ in \eqref{localextdef}, we have
\begin{equation}
\label{ytrw-234}
\lf|\mathcal{F}u(x)-m_u(B_{x^\ast}\cap\Omega)\r|
\leq\sum_{i\in I_x}\psi_i(x)\,\lf|m_u(B^\ast_i\cap\Omega)-m_u(B_{x^\ast}\cap\Omega)\r|,
\end{equation}
where $B^\ast_i$ is as in \eqref{ball-is}.
From \eqref{est-rx}, the choices of both $m$ and $k_0$, and $k\geq k_0$, it follows that
\begin{equation}
\label{zcpe-284}
2^{-m}\leq 2C_{\varrho_\#}\Lambda^2 r(x)<2^{-k-1}C_{\varrho_\#}\Lambda^2<1
\end{equation}
and, moreover,
\begin{equation}
\label{zcpe-285}
B_x\subset B_{x^\ast}\quad\mbox{ and }\quad C_{\varrho_\#}B_{x^\ast}\subset 2C_{\varrho_\#}^3B_x.
\end{equation}
Combining this with \eqref{qwpk-1}, the definition of $I_x$, and the doubling and
the measure density conditions of
$\mu$, we have, for any $i\in I_x$,
$$\mu(B^\ast_i\cap\Omega)\approx\mu(B_{x^\ast}\cap\Omega)\quad
\mbox{and}\quad \mu(C_{\varrho_\#}B_{x^\ast})\approx\mu(2C_{\varrho_\#}^3B_x).$$
As such, using both \eqref{qwpk-1} and \eqref{zcpe-285}, and  employing \eqref{BXPZ-2} in Lemma~\ref{intavgest}
(with $2^{-L}=2^{-m}<1=r_\ast$) allow  us to further bound the
right-hand side of \eqref{ytrw-234} as follows
\begin{align}
\label{pprb-392}
&\sum_{i\in I_x}\psi_j(x)\,\lf|m_u(B^\ast_i\cap\Omega)-m_u(B_{x^\ast}\cap\Omega)\r|\noz\\
&\quad\lesssim 2^{-m\varepsilon}\left[\,\mvint_{C_{\varrho_\#}B_{x^\ast}}
\sup\limits_{j\geq m-k_0}\lf\{2^{-j(s-\varepsilon)t}g_j^t\r\}\, d\mu\right]^{1/t}
\lesssim 2^{-m\varepsilon}\left[\,\mvint_{2C_{\varrho_\#}^3B_x}
\sup\limits_{j\geq m-k_0}\lf\{2^{-j(s-\varepsilon)t}g_j^t\r\}\, d\mu\right]^{1/t}\noz\\
&\quad\lesssim 2^{-m\varepsilon}\lf[\mathcal{M}_{\varrho_\#}
\lf(\sup\limits_{j\geq m-k_0}\lf\{2^{-j(s-\varepsilon)t}g_j^t\r\}\r)(x)\r]^{1/t},
\end{align}
where we have also used the fact that $\#I_x\leq M$ and $\psi_i(x)\leq1$ for any $i\in I_x$. Moreover, by \eqref{zcpe-284}
and the choice of $k_0$, we have
$2^{-m}<2^{-k-1}C_{\varrho_\#}\Lambda^2<2^{-k+k_0}$, so $m> k-k_0$ and by \eqref{ytrw-234} and \eqref{pprb-392}, we can estimate
\begin{align}
\label{pprb-392-2}
&\lf|\mathcal{F}u(x)-m_u(B_{x^\ast}\cap\Omega)\r|\noz\\
&\quad\lesssim 2^{-m\varepsilon}\lf[\mathcal{M}_{\varrho_\#}
\lf(\sup\limits_{j\geq m-k_0}\lf\{2^{-j(s-\varepsilon)t}g_j^t\r\}\r)(x)\r]^{1/t}\nonumber\\
&\quad\lesssim 2^{-ks}\lf[\mathcal{M}_{\varrho_\#}\lf(\sup\limits_{j\geq k-2k_0}
\lf\{2^{(k-j)(s-\varepsilon)t}g_j^t\r\}\r)(x)\r]^{1/t}\noz\\
&\quad\lesssim [\varrho_\#(x,y)]^s\lf[\mathcal{M}_{\varrho_\#}\lf(\sup\limits_{j\geq k-2k_0}\lf\{2^{(k-j)(s-\varepsilon)t}g_j^t\r\}\r)(x)\r]^{1/t}.
\end{align}
Observe next that, by $\delta\in[0,s-\varepsilon)$, if $k-2k_0-j\leq0$, then
\begin{equation}
\label{qkb-573}
2^{(k-j)(s-\varepsilon)}<2^{2k_0(s-\varepsilon)}2^{(k-2k_0-j)\delta}
=2^{2k_0(s-\varepsilon)}2^{-|k-2k_0-j|\delta}
\leq 2^{2k_0(s-\varepsilon+\delta)}2^{-|k-j|\delta},
\end{equation}
where, in obtaining the second inequality in \eqref{qkb-573}, we have simply
used the fact that $-|k-2k_0-j|\leq-|k-j|+2k_0$. This, in conjunction with \eqref{pprb-392-2}
and the definition of $h_k$, now gives
\begin{align}
\label{pprb-392-2-X}
\lf|\mathcal{F}u(x)-m_u(B_{x^\ast}\cap\Omega)\r|
&\lesssim  [\varrho_\#(x,y)]^s\lf[\mathcal{M}_{\varrho_\#}\lf(\sup\limits_{j\geq k-2k_0}
\lf\{2^{-|k-j|\delta t}g_j^t\r\}\r)(x)\r]^{1/t}\noz\\
&\lesssim  [\varrho_\#(x,y)]^s h_k(x),
\end{align}
which is a desired estimate for the first term in the right-hand side of \eqref{xzp-12}.

We now turn our attention to estimating  the second term in the right-hand side of
\eqref{xzp-12}. Let $n\in\mathbb{Z}$ be the least
integer satisfying $B_{\varrho_\#}(y,2^{-k})\subset 2^n B_{x^\ast}$. We claim that
\begin{equation}
\label{gnw-687}
2^{-k-1}C_{\varrho_\#}^{-1}\leq 2^{n-m}<C_{\varrho_\#}^2\Lambda^2 2^{-k+1}<2^{k_0-k}\leq 1.
\end{equation}
Given that we are assuming $k\geq k_0$, our choice of $k_0$ ensures that
$C_{\varrho_\#}^2\Lambda^2 2^{-k+1}<2^{k_0-k}\leq 1$. To prove the second inequality
in \eqref{gnw-687}, note that, by the choice of $n$, we necessarily have
$B_{\varrho_\#}(y,2^{-k})\not\subset 2^{n-1}B_{x^\ast}$.
Therefore, we can find a point $z_0\in X$ satisfying
${\varrho_\#}(y,z_0)<2^{-k}$ and ${\varrho_\#}(x^\ast,z_0)\geq 2^{n-1-m}$,
which, together with \eqref{est-rx}, implies that
\begin{align*}
2^{n-1-m}
&\leq{\varrho_\#}(x^\ast,z_0)
\leq C_{\varrho_\#}^2\max\{{\varrho_\#}(x^\ast,x),{\varrho_\#}(x,y),{\varrho_\#}(y,z_0)\}\\
&<C_{\varrho_\#}^2\max\{\Lambda^2r(x),2^{-k},2^{-k}\}<C_{\varrho_\#}^2\Lambda^2\,2^{-k}.
\end{align*}
The second inequality in \eqref{gnw-687} now follows. Furthermore,
by $x,\,y\in B_{\varrho_\#}(y,2^{-k})\subset2^n B_{x^\ast}$,
we have
\begin{align*}
2^{-k-1}\leq{\varrho_\#}(x,y)
&\leq C_{\varrho_\#}\max\{{\varrho_\#}(x,x^\ast),{\varrho_\#}(x^\ast,y)\}\\
&< C_{\varrho_\#}\max\{2^{n-m},2^{n-m}\}=C_{\varrho_\#}2^{n-m}.
\end{align*}
Hence, $2^{-k-1}C_{\varrho_\#}^{-1}\leq2^{n-m}$, that is,
the first inequality in \eqref{gnw-687} holds true. This proves the above claim \eqref{gnw-687}.

Moving on, we write
\begin{align}
\label{triplea}
&\lf|m_u(B_{x^\ast}\cap\Omega)-u(y)\r|\noz\\
&\quad=\lf|u(y)-m_u(B_{x^\ast}\cap\Omega)\r|\noz\\
&\quad\leq \lf|u(y)-m_u(B_{\varrho_\#}(y,2^{-k})\cap\Omega)\r|+
\lf|m_u(B_{\varrho_\#}(y,2^{-k})\cap\Omega)-m_u(2^nB_{x^\ast}\cap\Omega)\r|\noz\\
&\qquad+\lf|m_u(2^nB_{x^\ast}\cap\Omega)-m_u(B_{x^\ast}\cap\Omega)\r|\noz\\
&\quad=:{\rm I}+{\rm II}+{\rm III}.
\end{align}
Since \eqref{SobDiff} holds true $\mu$-almost everywhere in $\Omega$, without loss of generality,
we may assume  that \eqref{SobDiff}
holds true for $y$. Since $k\in\mathbb{N}$ implies that
$2^{-i-k}<1$ for any $i\in\mathbb{N}_0$,  appealing to
\eqref{BXPZ-2} in Lemma~\ref{intavgest} (with $2^{-L}=2^{-i-k}<1=r_\ast$)
and \eqref{qkb-573}, we conclude that
\begin{align*}
{\rm I}
&\leq\sum_{i=0}^\infty\big|m_u(B_{\varrho_\#}(y,2^{-(i+1)-k})
\cap\Omega)-m_u(B_{\varrho_\#}(y,2^{-i-k})\cap\Omega)\big|\\
&\lesssim\sum_{i=0}^\infty 2^{-(i+k)\varepsilon}
\left[\,\mvint_{B_{\varrho_\#}(y,C_{\varrho_\#}2^{-(i+k)})}
\sup\limits_{j\geq i+k-k_0}\lf\{2^{-j(s-\varepsilon)t}g_j^t\r\}\, d\mu\right]^{1/t}\\
&\lesssim 2^{-k\varepsilon}\lf[\mathcal{M}_{\varrho_\#}\lf(\sup\limits_{j\geq k-k_0}
\lf\{2^{-j(s-\varepsilon)t}g_j^t\r\}\r)(y)\r]^{1/t}\\
&\lesssim 2^{-ks}\lf[\mathcal{M}_{\varrho_\#}\lf(\sup\limits_{j\geq k-k_0}
\lf\{2^{(k-j)(s-\varepsilon)t}g_j^t\r\}\r)(y)\r]^{1/t}\\
&\lesssim [\varrho_\#(x,y)]^s \lf[\mathcal{M}_{\varrho_\#}
\lf(\sup\limits_{j\geq k-2k_0}\lf\{2^{-|k-j|\delta t}g_j^t\r\}\r)(y)\r]^{1/t}
\lesssim [\varrho_\#(x,y)]^s h_k(y).
\end{align*}

To estimate ${\rm II}$ in \eqref{triplea}, observe that
\eqref{gnw-687} implies both $m-n>k-k_0$ and $2^{n-m}<2^{k_0-k}\leq 1$.
On the other hand, appealing once again to \eqref{gnw-687}, we
have
\begin{equation}\label{emd}
B_{\varrho_\#}(y,2^{-k})\subset  2^{n}
B_{x^\ast}\subset B_{\varrho_\#}(y,\Lambda^2C_{\varrho_\#}^42^{-k+1}),
\end{equation}
which,  in conjunction with the doubling and the measure density properties of $\mu$,
gives
$$\mu(B_{\varrho_\#}(y,2^{-k})\cap\Omega)\approx\mu(2^nB_{x^\ast}\cap\Omega). $$
Putting these all together, we can then use  \eqref{BXPZ-2} in
Lemma~\ref{intavgest} (with $2^{-L}=2^{n-m}<1=r_\ast$) and \eqref{qkb-573} to estimate
\begin{align*}
{\rm II}&=
\lf|m_u(B_{\varrho_\#}(y,2^{-k})\cap\Omega)-m_u(2^nB_{x^\ast}\cap\Omega)\r|\noz\\
&\lesssim2^{-(m-n)\varepsilon}\left[\,\mvint_{C_{\varrho_\#}2^{n}B_{x^\ast}}
\sup_{j\geq m-n-k_0}\lf\{2^{-j(s-\varepsilon)t}g_j^t\r\}\, d\mu\right]^{1/t}\noz\\
&\lesssim2^{-k\varepsilon}\left[\,\mvint_{B_{\varrho_\#}(y,\Lambda^2C_{\varrho_\#}^42^{-k+1})}\sup_{j\geq k-2k_0}\lf\{2^{-j(s-\varepsilon)t}g_j^t\r\}\, d\mu\right]^{1/t}\noz\\
&\lesssim2^{-ks}\lf[\mathcal{M}_{\varrho_\#}\lf(\sup\limits_{j\geq k-2k_0}
\lf\{2^{(k-j)(s-\varepsilon)t}g_j^t\r\}\r)(y)\r]^{1/t}
\lesssim[\varrho_\#(x,y)]^sh_k(y).
\end{align*}
As concerns ${\rm III}$, note that, when $n\geq1$, we can write
\begin{equation}
\label{rht-3834-1}
{\rm III}=
\big|m_u(B_{x^\ast}\cap\Omega)-m_u(2^nB_{x^\ast}\cap\Omega)\big|
\leq\sum_{i=0}^{n-1}\lf|m_u(2^iB_{x^\ast}\cap\Omega)-m_u(2^{i+1}B_{x^\ast}\cap\Omega)\r|.
\end{equation}
For any $i\in\{0,\dots,n-1\}$, by \eqref{gnw-687}, we have
$2^{i+1-m}\leq2^{n-m}<1$, which,
in conjunction with the doubling and the measure
density properties of $\mu$,
gives $\mu(2^iB_{x^\ast}\cap\Omega)\approx\mu(2^{i+1}B_{x^\ast}\cap\Omega)$.
As such, we use \eqref{BXPZ-2} in
Lemma~\ref{intavgest} (with $2^{-L}=2^{i+1-m}<1=r_\ast$) again  to obtain
\begin{align}
\label{rht-3834}
&\sum_{i=0}^{n-1}\lf|m_u(2^iB_{x^\ast}\cap\Omega)-m_u(2^{i+1}B_{x^\ast}\cap\Omega)\r|\noz\\
&\quad\lesssim\sum_{i=0}^{n-1}2^{-(m-i-1)\varepsilon}
\left[\,\mvint_{C_{\varrho_\#}2^{i+1}B_{x^\ast}}
\sup\limits_{j\geq m-i-1-k_0}\lf\{2^{-j(s-\varepsilon)t}g_j^t\r\}\, d\mu\right]^{1/t}.
\end{align}
Moreover, since $B_x\subset B_{x^\ast}\subset 2C_{\varrho_\#}^2B_x$ [see \eqref{zcpe-285}], we deduce that
$$C_{\varrho_\#}2^{i+1}B_x\subset C_{\varrho_\#}2^{i+1}B_{x^\ast}\subset 2^{i+2}C_{\varrho_\#}^3B_x.$$
Hence, by \eqref{qkb-573}, the doubling property for $\mu$,
and the fact that $m-n>k-k_0$ and $2^{n-m}< 2^{k_0-k}$,
we  further estimate
the last term in \eqref{rht-3834} as follows
\begin{align}
\label{rht-3834-3}
&\sum_{i=0}^{n-1}2^{-(m-i-1)\varepsilon}\left[\,\mvint_{C_{\varrho_\#}2^{i+1}B_{x^\ast}}\sup\limits_{j\geq m-i-1-k_0}\lf\{2^{-j(s-\varepsilon)t}g_j^t\r\}\, d\mu\right]^{1/t}\noz\\
&\quad\lesssim\sum_{i=-\infty}^{n-1}2^{-(m-i-1)\varepsilon}\lf[\mathcal{M}_{\varrho_\#}\lf(\sup\limits_{j\geq m-n-k_0}\lf\{2^{-j(s-\varepsilon)t}g_j^t\r\}\r)(x)\r]^{1/t}\noz\\
&\quad\lesssim2^{(n-m)\varepsilon}\lf[\mathcal{M}_{\varrho_\#}\lf(\sup\limits_{j\geq m-n-k_0}\lf\{2^{-j(s-\varepsilon)t}g_j^t\r\}\r)(x)\r]^{1/t}\noz\\
&\quad\lesssim2^{-ks}\lf[\mathcal{M}_{\varrho_\#}\lf(\sup\limits_{j\geq k-2k_0}\lf\{2^{(k-j)(s-\varepsilon)t}g_j^t\r\}\r)(x)\r]^{1/t}
\lesssim[\varrho_\#(x,y)]^sh_k(x).
\end{align}
Combining  \eqref{rht-3834-1}, \eqref{rht-3834}, and \eqref{rht-3834-3}, we conclude that
$$
{\rm III}\lesssim[\varrho_\#(x,y)]^sh_k(x)\quad\mbox{whenever $n\geq1$.}
$$
If $n\leq0$, then $2^nB_{x^\ast}\subset B_{x^\ast}
\subset B_{\varrho_\#}(y,C_{\varrho_\#}^2\Lambda^22^{-k+1})$, where the second inclusion
follows from the fact that, for any
$z\in B_{x^\ast}$, one has (keeping in mind that $y \in  2^{n}
B_{x^\ast}$ and $2^{-m}<C_{\varrho_\#}\Lambda^22^{-k+1}$)
\begin{align*}
{\varrho_\#}(y,z)&\leq C_{\varrho_\#}\max\{{\varrho_\#}(y,x^\ast),{\varrho_\#}(x^\ast,z)\}
\\
&\leq C_{\varrho_\#}\max\{2^{n-m},2^{-m}\}=C_{\varrho_\#}2^{-m}<C_{\varrho_\#}^2\Lambda^22^{-k+1}.
\end{align*}
From this, \eqref{emd}, and the doubling and
the  measure density properties of $\mu$, it follows that
\begin{align*}
\mu(B_{x^\ast}\cap\Omega)
&\leq\mu\lf(B_{\varrho_\#}(y,C_{\varrho_\#}^2\Lambda^22^{-k+1})\r)
\lesssim\mu(B_{\varrho_\#}(y,2^{-k}))\\
&\lesssim\mu(2^nB_{x^\ast})\lesssim\mu(2^nB_{x^\ast}\cap\Omega),
\end{align*}
where we have used the fact that $2^{n-m}<1$ [see \eqref{gnw-687}]
when employing the measure density condition. Thus, we actually have $\mu(2^nB_{x^\ast}\cap\Omega)\approx\mu(B_{x^\ast}\cap\Omega)$.
Recall that $2^{-m}<1$ by \eqref{zcpe-284}. Thus, appealing again to
\eqref{BXPZ-2} in Lemma~\ref{intavgest} (with $2^{-L}=2^{-m}<1=r_\ast$) and
\eqref{qkb-573}, as well as recycling some of the estimates in
\eqref{pprb-392}-\eqref{pprb-392-2-X}, we find that
\begin{align}
\label{bvu-38}
{\rm III}&=
\lf|m_u(2^nB_{x^\ast}\cap\Omega)-m_u(B_{x^\ast}\cap\Omega)\r|
\lesssim2^{-m\varepsilon}\left[\,\mvint_{C_{\varrho_\#}B_{x^\ast}}\sup\limits_{j\geq m-k_0}\lf\{2^{-j(s-\varepsilon)t}g_j^t\r\}\, d\mu\right]^{1/t}\noz\\
&\lesssim[\varrho_\#(x,y)]^s\lf[\mathcal{M}_{\varrho_\#}\lf(\sup\limits_{j\geq k-2k_0}\lf\{2^{(k-j)(s-\varepsilon)t}g_j^t\r\}\r)(x)\r]^{1/t}
\lesssim[\varrho_\#(x,y)]^sh_k(x).
\end{align}
Combining the estimates for ${\rm I}$,
 ${\rm II}$, and ${\rm III}$,  we have proved that the
second term in the right-hand side of
\eqref{xzp-12} can be bounded by $[\varrho_\#(x,y)]^s [h_k(x)+h_k(y)]$ multiplied a positive constant.
This gives the desired estimate in  Case 2.

\vspace{0.2cm}
\noindent{\bf CASE 3}: $x,\,y\in V\setminus\Omega$ with
${\varrho_\#}(x,y)\geq\min\{{\rm dist}_{\varrho_\#}(x,\Omega),{\rm dist}_{\varrho_\#}(y,\Omega)\}$.
\vspace{0.2cm}

In this case, as in the beginning of the proof for Case 2,
we start by choosing points $x^\ast\in B_x\cap\Omega$ and  $y^\ast\in B_y\cap\Omega$, where $B_x:=B_{\varrho_\#}(x,\Lambda^2 r(x))$ and $B_y:=B_{\varrho_\#}(y,\Lambda^2 r(y))$, with $r(x)$ and $r(y)$  defined as in \eqref{est-rx}.  Let $m_x,\,m_y\in\mathbb{Z}$ satisfy
$2^{-m_x-1}\leq C_{\varrho_\#}\Lambda^2 r(x)<2^{-m_x}$ and $2^{-m_y-1}\leq C_{\varrho_\#}\Lambda^2 r(y)<2^{-m_y}$, respectively,
and set $B_{x^\ast}:=B_{\varrho_\#}(x^\ast,2^{-m_x})$ and
$B_{y^\ast}:=B_{\varrho_\#}(y^\ast,2^{-m_y})$.

In order to establish \eqref{fracgradext} in this case, we write
\begin{align}
\label{qwk-384}
\lf|\mathcal{F}u(x)-\mathcal{F}u(y)\r|&\leq \lf|\mathcal{F}u(x)-m_u(B_{x^\ast}\cap\Omega)\r|
+\lf|m_u(B_{x^\ast}\cap\Omega)-m_u(B_{y^\ast}\cap\Omega)\r|\noz\\
&\quad+\lf|m_u(B_{y^\ast}\cap\Omega)-\mathcal{F}u(y)\r|,
\end{align}
and estimate each term on the right-hand side of \eqref{qwk-384} separately. To this end,
without loss of generality, we may assume that
${\rm dist}_{\varrho_\#}(x,\Omega)\leq{\rm dist}_{\varrho_\#}(y,\Omega)$.
Then, by the assumption of Case 3, one has
\begin{equation}\label{esrx}
r(x):=\frac{{\rm dist}_{\varrho_\#}(x,\Omega)}{4C_{\varrho_\#}}\leq\frac{1}{4}{\varrho_\#}(x,y)<2^{-k-2},
\end{equation}
where we have also used the fact that $C_{\varrho_\#}\geq1$. Moreover,

\begin{equation}\label{esry}
r(y):=\frac{{\rm dist}_{\varrho_\#}(y,\Omega)}{4C_{\varrho_\#}}\leq\frac{C_{\varrho_\#}\max\{{\varrho_\#}(x,y),{\rm dist}_{\varrho_\#}(x,\Omega)\}}{4C_{\varrho_\#}}\leq\frac{1}{4}{\varrho_\#}(x,y)<2^{-k-2}.
\end{equation}
In view of these,
the first and the third terms in the right-hand side of
\eqref{qwk-384}
can be estimated, in a fashion similar to the arguments for
\eqref{ytrw-234}-\eqref{pprb-392-2-X},  in order to obtain
\begin{align}
\label{pprb-392-xxn}
&\lf|\mathcal{F}u(x)-m_u(B_{x^\ast}\cap\Omega)\r|+\lf|m_u(B_{y^\ast}\cap\Omega)-\mathcal{F}u(y)\r|\noz\\
&\quad\lesssim [\varrho_\#(x,y)]^s\left\{\lf[\mathcal{M}_{\varrho_\#}\lf(\sup\limits_{j\geq k-2k_0}
\lf\{2^{-|k-j|\delta t}g_j^t\r\}\r)(x)\r]^{1/t}\r.\noz\\
&\qquad\qquad\qquad\qquad\lf.+\lf[\mathcal{M}_{\varrho_\#}\lf(\sup\limits_{j\geq k-2k_0}\lf\{2^{-|k-j|\delta t}g_j^t\r\}\r)(y)\r]^{1/t}\right\}\noz\\
&\quad\lesssim [\varrho_\#(x,y)]^s\lf[h_k(x)+h_k(y)\r],
\end{align}
as wanted.

We still need to estimate the second term in the right-hand side of  \eqref{qwk-384}. To this end,
Let $n_x,\,n_y\in\mathbb{Z}$ be the least integers satisfying
\begin{equation}
\label{xcd-475-trnb}
B_{\varrho_\#}(y,2^{-k})\subset 2^{n_y} B_{y^\ast}\quad\mbox{ and }\quad2^{n_y} B_{y^\ast}\subset 2^{n_x} B_{x^\ast},
\end{equation}
and write
\begin{align}
\label{qwk-384-58}
&\lf|m_u(B_{x^\ast}\cap\Omega)-m_u(B_{y^\ast}\cap\Omega)\r|\noz\\
&\quad\leq \lf|m_u(B_{x^\ast}\cap\Omega)-m_u(2^{n_x}B_{x^\ast}\cap\Omega)\r|
+\lf|m_u(2^{n_x}B_{x^\ast}\cap\Omega)-m_u(2^{n_y}B_{y^\ast}\cap\Omega)\r|\noz\\
&\qquad+\lf|m_u(2^{n_y}B_{y^\ast}\cap\Omega)-m_u(B_{y^\ast}\cap\Omega)\r|\noz\\
&\quad=:\widetilde{\rm I}+\widetilde{\rm II}+\widetilde{\rm III}.
\end{align}

In order to estimate both $\widetilde{\rm I}$ and $\widetilde{\rm III}$,
we  first show that $2^{n_y-m_y}\approx 2^{-k}\approx2^{n_x-m_x}$.
Indeed, by the choice of $n_y$, there exists a
point $z_0\in X$ satisfying ${\varrho_\#}(y,z_0)<2^{-k}$
and ${\varrho_\#}(y^\ast,z_0)\geq 2^{n_y-m_y-1}$,
which, together with \eqref{esry},  allows us to estimate
\begin{align*}
2^{n_y-m_y-1}
&\leq{\varrho_\#}(y^\ast,z_0)
\leq C_{\varrho_\#}\max\{{\varrho_\#}(y^\ast,y),{\varrho_\#}(y,z_0)\}\\
&<C_{\varrho_\#}\max\{\Lambda^2r(y),2^{-k}\}\\
&<C_{\varrho_\#}\max\{\Lambda^22^{-k-2},2^{-k}\}<C_{\varrho_\#}\Lambda^2\,2^{-k}.
\end{align*}
Hence, $2^{n_y-m_y}<C_{\varrho_\#}\Lambda^2\,2^{-k+1}$, which further implies $m_y-n_y>k-k_0$, because
$C_{\varrho_\#}\Lambda^2\,2^{-k+1}<2^{k_0-k}$ [see \eqref{gnw-687}]. On the other hand,
since $x,\,y\in B_{\varrho_\#}(y,2^{-k})\subset2^{n_y}B_{y^\ast}$, one has
\begin{equation*}
2^{-k-1}\leq{\varrho_\#}(x,y)
\leq C_{\varrho_\#}\max\{{\varrho_\#}(x,y^\ast),{\varrho_\#}(y^\ast,y)\}<C_{\varrho_\#}2^{n_y-m_y},
\end{equation*}
which implies $2^{-k-1}C_{\varrho_\#}^{-1}<2^{n_y-m_y}$. Thus, by \eqref{gnw-687}, we obtain
\begin{equation}
\label{gnx-345}
2^{-k-1}C_{\varrho_\#}^{-1}<2^{n_y-m_y}<C_{\varrho_\#}\Lambda^2\,2^{-k+1}<1.
\end{equation}
Similarly, by the choice of $n_x$, we know that there exists
a point $z_0\in X$ satisfying
${\varrho_\#}(y^\ast,z_0)<2^{n_y-m_y}$ and
${\varrho_\#}(x^\ast,z_0)\geq 2^{n_x-m_x-1}$.
Making use of the second inequality in \eqref{gnx-345}, as well as
\eqref{esrx} and \eqref{esry}, we find that
\begin{align*}
2^{n_x-m_x-1}
&\leq{\varrho_\#}(x^\ast,z_0)
\leq C_{\varrho_\#}^3\max\{{\varrho_\#}(x^\ast,x),{\varrho_\#}(x,y),{\varrho_\#}(y,y^\ast),{\varrho_\#}(y^\ast,z_0)\}\\
&<C_{\varrho_\#}^3\max\{\Lambda^2r(x),2^{-k},\Lambda^2r(y),2^{n_y-m_y}\}\\
&<C_{\varrho_\#}^3\max\{\Lambda^22^{-k-2},2^{-k},
\Lambda^22^{-k-2},C_{\varrho_\#}\Lambda^2\,2^{-k+1}\}
=C_{\varrho_\#}^4\Lambda^2\,2^{-k+1}.
\end{align*}
From this and the definition of $k_0$, we deduce that
$2^{n_x-m_x}<C_{\varrho_\#}^4\Lambda^2\,2^{-k+2}\leq2^{k_0-k}$,
which further implies that
\begin{equation}
\label{xcd-475-tn}
m_x-n_x>k-k_0.
\end{equation}
On the other hand, since $x,\,y\in2^{n_y}B_{y^\ast}\subset2^{n_x}B_{x^\ast}$, we have
\begin{equation*}
2^{-k-1}\leq{\varrho_\#}(x,y)
\leq C_{\varrho_\#}\max\{{\varrho_\#}(x,x^\ast),{\varrho_\#}(x^\ast,y)\}<C_{\varrho_\#}2^{n_x-m_x}.
\end{equation*}
Thus,
\begin{equation}
\label{xcd-475}
2^{-k-1}C_{\varrho_\#}^{-1}<2^{n_x-m_x}<C_{\varrho_\#}^4\Lambda^2\,2^{-k+2}\leq2^{k_0-k}\le1,
\end{equation}
and so $2^{n_y-m_y}\approx 2^{-k}\approx2^{n_x-m_x}$, as wanted.

At this stage, we can then argue as in \eqref{rht-3834-1}-\eqref{bvu-38} to obtain
\begin{equation}\label{est-I+III}
\widetilde{\rm I}+\widetilde{\rm III}\lesssim [\varrho_\#(x,y)]^s\lf[h_k(x)+h_k(y)\r].
\end{equation}
So it remains to estimate
$\widetilde{\rm II}$. With the goal of using \eqref{BXPZ-2} in Lemma~\ref{intavgest}, we make a few observations. First, $2^{n_y}B_{y^\ast}\subset2^{n_x}B_{x^\ast}$ by design  and, from \eqref{xcd-475}, we deduce that
$2^{n_x-m_x}<1$.  Moreover, if $z\in 2^{n_x} B_{x^\ast}$,
then, by \eqref{gnx-345} and the fact that $2^{n_x-m_x}
<C_{\varrho_\#}^4\Lambda^2\,2^{-k+2}$ and
${\varrho_\#}(y^\ast,x^\ast)<C_{\varrho_\#}^2\Lambda^2\,2^{-k}$, one has
\begin{align*}
{\varrho_\#}(y^\ast,z)
&\leq C_{\varrho_\#}\max\{{\varrho_\#}(y^\ast,x^\ast),{\varrho_\#}(x^\ast,z)\}\\
&<C_{\varrho_\#}\max\{C_{\varrho_\#}^2\Lambda^2\,2^{-k},2^{n_x-m_x}\}\\
&<C_{\varrho_\#}\max\{2C_{\varrho_\#}^3\Lambda^2\,2^{n_y-m_y},
8C_{\varrho_\#}^5\Lambda^2\,2^{n_y-m_y}\}<8C_{\varrho_\#}^6\Lambda^2\,2^{n_y-m_y},
\end{align*}
which implies that
$2^{n_x} B_{x^\ast}\subset 8C_{\varrho_\#}^6\Lambda^2\,2^{n_y} B_{y^\ast}$.
From this, $2^{n_y-m_y}<1$, and the doubling and
the measure density conditions of $\mu$, it follows that
$$
\mu(2^{n_x} B_{x^\ast}\cap\Omega)\leq\mu(2^{n_x} B_{x^\ast})\leq\mu(8C_{\varrho_\#}^6\Lambda^2\,2^{n_y} B_{y^\ast})
\lesssim\mu(2^{n_y} B_{y^\ast})\lesssim\mu(2^{n_y} B_{y^\ast}\cap\Omega).$$
Hence,
$$\mu(2^{n_y} B_{y^\ast}\cap\Omega)\approx\mu(2^{n_x} B_{x^\ast}\cap\Omega)\gtrsim\mu(B_{x^\ast}), $$
where in obtaining the last inequality we have used
the measure density condition \eqref{measdens} again
and the fact that $2^{n_x-m_x}<1$.
Going further,
observe that, if $z\in C_{\varrho_\#}2^{n_x} B_{x^\ast}$, then, by $y\in 2^{n_x} B_{x^\ast}$ and \eqref{xcd-475}, it holds true
that
\begin{align*}
{\varrho_\#}(y,z)
&\leq C_{\varrho_\#}\max\{{\varrho_\#}(y,x^\ast),{\varrho_\#}(x^\ast,z)\}
<C_{\varrho_\#}\max\{2^{n_x-m_x},C_{\varrho_\#}2^{n_x-m_x}\}\\
&=C_{\varrho_\#}^22^{n_x-m_x}<C_{\varrho_\#}^6\Lambda^2\,2^{-k+2},
\end{align*}
which implies
\begin{equation*}
C_{\varrho_\#}2^{n_x} B_{x^\ast}\subset B_{\varrho_\#}(y,C_{\varrho_\#}^6\Lambda^22^{-k+2}).
\end{equation*}
By this,    \eqref{xcd-475-trnb}, \eqref{xcd-475}, \eqref{BXPZ-2} in Lemma~\ref{intavgest}
(with $2^{-L}=2^{n_x-m_x}<1=r_\ast$), \eqref{xcd-475-tn},
and the doubling property for $\mu$, we conclude that
\begin{align*}
\widetilde{\rm II}
&=\lf|m_u(2^{n_y}B_{y^\ast}\cap\Omega)-m_u(2^{n_x}B_{x^\ast}\cap\Omega)\r|\noz\\
&\lesssim2^{-(m_x-n_x)\varepsilon}\left[\,\mvint_{C_{\varrho_\#}2^{n_x}B_{x^\ast}}\sup\limits_{j\geq m_x-n_x-k_0}\lf\{2^{-j(s-\varepsilon)t}g_j^t\r\}\, d\mu\right]^{1/t}\noz\\
&\lesssim2^{-k\varepsilon}\left[\,\mvint_{B_{\varrho_\#}(y,C_{\varrho_\#}^6\Lambda^22^{-k+2})}\sup\limits_{j\geq k-2k_0}\lf\{2^{-j(s-\varepsilon)t}g_j^t\r\}\, d\mu\right]^{1/t}\noz\\
&\lesssim2^{-ks}\lf[\mathcal{M}_{\varrho_\#}\lf(\sup\limits_{j\geq k-2k_0}\lf\{2^{(k-j)(s-\varepsilon)t}g_j^t\r\}\r)(y)\r]^{1/t}
\lesssim[\varrho_\#(x,y)]^s h_k(y).
\end{align*}
This, together with \eqref{est-I+III}, \eqref{qwk-384-58},  \eqref{pprb-392-xxn},
and  \eqref{qwk-384}, then finishes the proof of the desired estimate \eqref{fracgradext}
in  Case 3.

\vspace{0.2cm}
\noindent{\bf CASE 4}:  $x,\,y\in V\setminus\Omega$ with ${\varrho_\#}(x,y)<\min\{{\rm dist}_{\varrho_\#}(x,\Omega),{\rm dist}_{\varrho_\#}(y,\Omega)\}$.
\vspace{0.2cm}

Again, in this case, without loss of generality, we may
assume ${\rm dist}_{\varrho_\#}(x,\Omega)
\leq{\rm dist}_{\varrho_\#}(y,\Omega)$.
Clearly, this assumption implies $r(x)\leq r(y)$,
where $r(x)$ and $r(y)$ are defined as in \eqref{est-rx}.
Let $\ell_0\in\mathbb{Z}$ be the least integer satisfying
$C^2_{\varrho_\#}\leq 2^{\ell_0}$ and note that $\ell_0$ is non-negative.
Then
\begin{align}
\label{vxlk-576}
r(y):=&\,\frac{{\rm dist}_{\varrho_\#}(y,\Omega)}{4C_{\varrho_\#}}\leq\frac{C_{\varrho_\#}\max\{{\varrho_\#}(x,y),{\rm dist}_{\varrho_\#}(x,\Omega)\}}{4C_{\varrho_\#}}\noz\\
=&\,\frac{1}{4}{\rm dist}_{\varrho_\#}(x,\Omega)= C_{\varrho_\#}r(x)\leq
C_{\varrho_\#}^2r(x)\leq2^{\ell_0}r(x),
\end{align}
and so $r(x)\approx r(y)$. Consequently, since  $r_i\approx r(x)$
for any $i\in I_x$, and $r_i\approx r(y)$ for any $i\in I_y$ [see \eqref{qwpk-1-radii}],
we have $r_i\approx r(x)$ for any $i\in I_x\cup I_y$; moreover,
by \eqref{qwpk-3}, we have
$$\sum_{i\in I_x\cup I_y}\lf[\psi_i(x)-\psi_i(y)\r]
=\sum_{i\in I_x\cup I_y}\psi_i(x)-\sum_{i\in I_x\cup I_y}\psi_i(y)=1-1=0.
$$
From these and Theorem \ref{MSz7b}(i) with $\beta=\alpha$, we deduce that
\begin{align}
\label{nri-458}
|\mathcal{F}u(x)-\mathcal{F}u(y)|
&=\sum_{i\in I_x\cup I_y}\left|\psi_i(x)-\psi_i(y)\right|\cdot\left|m_u(B_i^\ast\cap\Omega)
-m_u(2^{\ell_0}B_{x^\ast}\cap\Omega)\right|\noz\\
&\lesssim\frac{[{\varrho_\#}(x,y)]^\alpha}{[r(x)]^{\alpha}}\sum_{i\in I_x\cup I_y}\left|m_u(B_i^\ast\cap\Omega)-m_u(2^{\ell_0}B_{x^\ast}\cap\Omega)\right|,
\end{align}
where  $B_x$ and $B_{x^\ast}$ are defined as in
the proof for Case 3.

In order to bound the sum in \eqref{nri-458}, we will once again appeal to Lemma~\ref{intavgest}.
With this idea in mind, fix an $i\in I_x\cup I_y$. Observe that,
if $i\in I_x$, then, by \eqref{qwpk-1}, \eqref{zcpe-285}, and $\ell_0\ge 0$,
we have $B_i^\ast\subset B_{x}\subset
B_{x^\ast}\subset 2^{\ell_0}B_{x^\ast}$; if $i\in I_y$, then, with $B_y$
and $B_{y^\ast}$ maintaining their significance from Case~3, we have $B_i^\ast\subset B_{y}
\subset 2^{\ell_0}B_{x^\ast}$. Indeed, the first inclusion follows from the definition of $I_y$
and the analogous version of \eqref{qwpk-1} for $B_y$. To see the second inclusion, observe
that, if $z\in B_y$, then we can use \eqref{vxlk-576} to obtain
\begin{align*}
{\varrho_\#}(x^\ast,z)&\leq C^2_{\varrho_\#}\max\{{\varrho_\#}(x^\ast,x),{\varrho_\#}(x,y),{\varrho_\#}(y,z)\}\\
&<C^2_{\varrho_\#}\max\{\Lambda^2r(x),{\rm dist}_{\varrho_\#}(x,\Omega),\Lambda^2r(y)\}\\
&=C^2_{\varrho_\#}\max\{\Lambda^2r(x),4C_{\varrho_\#}r(x),\Lambda^2r(y)\}
=C^2_{\varrho_\#}\Lambda^2r(y)\leq C^3_{\varrho_\#} \Lambda^2r(x)< 2^{\ell_0-m_x},
\end{align*}
where the second equality follows from the estimate $\Lambda^2>\theta^2=4C_{\varrho_\#}^4\geq4C_{\varrho_\#}$, and the last inequality is obtained from the choices of $\ell_0$
and  $m_x$.  Hence, $B_{y}\subset  2^{\ell_0}B_{x^\ast}$.
Altogether, we proved that   $B_i^\ast\subset  2^{\ell_0}B_{x^\ast}$ whenever $i\in I_x\cup I_y$.

Next, we show
that $\mu(B_i^\ast\cap\Omega)\approx\mu(2^{\ell_0}B_{x^\ast}\cap\Omega)$ for any $i\in I_x\cup I_y$. Fix $i\in I_x\cup I_y$.
It follows from what we have just shown that $\mu(B_i^\ast\cap\Omega)
\leq\mu(2^{\ell_0}B_{x^\ast}\cap\Omega)$. To see $\mu(2^{\ell_0}B_{x^\ast}
\cap\Omega)\lesssim\mu(B_i^\ast\cap\Omega)$, observe that, if $i\in I_x$, then
the inclusions $B_{x^\ast}\subset 2C_{\varrho_\#}^2B_x$ in \eqref{zcpe-285}  and $B_{x}\subset
C_{\varrho_\#}^2\Lambda^3B_i^\ast$ in \eqref{qwpk-1}, in conjunction with the doubling and the
measure density properties for $\mu$, give
$$
\mu(2^{\ell_0}B_{x^\ast}\cap\Omega)\leq\mu(2^{\ell_0}B_{x^\ast})
\lesssim\mu(B_{x})\lesssim\mu(C_{\varrho_\#}^2\Lambda^3B_i^\ast)\lesssim\mu(B_i^\ast)\lesssim\mu(B_i^\ast\cap\Omega).
$$
On the other hand, if $i\in I_y$, then similar to the proof of
\eqref{qwpk-1}, we find that  $B_y\subset C_{\varrho_\#}^2\Lambda^3B_i^\ast$. Moreover, we have
$B_{x^\ast}\subset 2C_{\varrho_\#}^2B_y$ since, for any $z\in B_{x^\ast}$, there holds
\begin{align*}
{\varrho_\#}(y,z)&\leq C^2_{\varrho_\#}\max\{{\varrho_\#}(y,x),{\varrho_\#}(x,x^\ast),{\varrho_\#}(x^\ast,z)\}\\
&<C^2_{\varrho_\#}\max\{{\rm dist}_{\varrho_\#}(x,\Omega),\Lambda^2r(x),2^{-m_x}\}\\
&<C^2_{\varrho_\#}\max\{4C_{\varrho_\#}r(x),\Lambda^2r(x),2C_{\varrho_\#}\Lambda^2r(x)\}
=2C^3_{\varrho_\#}\Lambda^2r(x)\leq 2C^3_{\varrho_\#}\Lambda^2r(y).
\end{align*}
These inclusions, together with the measure density and doubling conditions imply that
$$
\mu(2^{\ell_0}B_{x^\ast}\cap\Omega)\leq\mu(2^{\ell_0}B_{x^\ast})
\lesssim\mu(B_{y})\lesssim\mu(C_{\varrho_\#}^2\Lambda^3B_i^\ast)\lesssim\mu(B_i^\ast)\lesssim\mu(B_i^\ast\cap\Omega).
$$
This finishes the proof of the fact that $\mu(B_i^\ast\cap\Omega)
\approx\mu(2^{\ell_0}B_{x^\ast}\cap\Omega)$.

To invoke Lemma~\ref{intavgest}, we finally need to uniformly bound the
radius of the ball $2^{\ell_0}B_{x^\ast}$, namely, $2^{-(m_x-\ell_0)}$.
Indeed, the choices of both $\ell_0$ and $m_x$, the definition of $r(x)$ [see \eqref{esrx}],
and the membership of $x\in V$ [see \eqref{neib}]  imply that
$$
2^{-(m_x-\ell_0)}\leq 4C_{\varrho_\#}^3\Lambda^2r(x)=C_{\varrho_\#}^2\Lambda^2{\rm dist}_{\varrho_\#}(x,\Omega)<2C_{\varrho_\#}^3\Lambda^2<2^{k_0}.
$$
From this and   \eqref{BXPZ-2} in Lemma~\ref{intavgest}
[with $2^{-L}=2^{-(m_x-\ell_0)}<2^{k_0}=r_\ast$],
we deduce for any $i\in I_x\cup I_y$ that (keeping in mind $B_{x}\subset C_{\varrho_\#}2^{\ell_0}B_{x^\ast}\subset 2^{\ell_0+1}C_{\varrho_\#}^3B_x$)
\begin{align*}
&\left|m_u(B_i^\ast\cap\Omega)-m_u(2^{\ell_0}B_{x^\ast}\cap\Omega)\right|\noz\\
&\quad\lesssim2^{-(m_x-\ell_0)\varepsilon}\left[\,\mvint_{C_{\varrho_\#}2^{\ell_0}B_{x^\ast}}\sup\limits_{j\geq m_x-\ell_0-k_0}\lf\{2^{-j(s-\varepsilon)t}g_j^t\r\}\, d\mu\right]^{1/t}\noz\\
&\quad\lesssim2^{-m_x\varepsilon}\left[\,\mvint_{2^{\ell_0+1}C_{\varrho_\#}^3B_{x}}\sup\limits_{j\geq m_x-\ell_0-k_0}\lf\{2^{-j(s-\varepsilon)t}g_j^t\r\}\, d\mu\right]^{1/t}\noz\\
&\quad\lesssim2^{-m_x\varepsilon}\lf[\mathcal{M}_{\varrho_\#}\lf(\sup\limits_{j\geq m_x-\ell_0-k_0}\lf\{2^{-j(s-\varepsilon)t}g_j^t\r\}\r)(x)\r]^{1/t},
\end{align*}
which, together with  \eqref{nri-458} and the fact that $\#(I_x\cup I_y)\leq 2M$ (see Theorem~\ref{L-WHIT}),
further implies that
\begin{align}
\label{nri-458-X}
\lf|\mathcal{F}u(x)-\mathcal{F}u(y)\r|
&\lesssim\frac{[{\varrho_\#}(x,y)]^\alpha}{[r(x)]^{\alpha}}\sum_{i\in I_x\cup I_y}\left|m_u(B_i^\ast\cap\Omega)-m_u(2^{\ell_0}B_{x^\ast}\cap\Omega)\right|\noz\\
&\lesssim[{\varrho_\#}(x,y)]^\alpha [r(x)]^{-\alpha}2^{-m_x\varepsilon}
\lf[\mathcal{M}_{\varrho_\#}\lf(\sup\limits_{j\geq m_x-\ell_0-k_0}\lf\{2^{-j(s-\varepsilon)t}g_j^t\r\}\r)(x)\r]^{1/t}.
\end{align}
Since $2^{-m_x}\leq 2C_{\varrho_\#}^2\Lambda^2r(x)$,
${\varrho_\#}(x,y)<\min\{2^{-k},4C_{\varrho_\#}r(x)\}$, by the ranges of $s$ and $\alpha$,
and $\delta\geq0$, we conclude that
\begin{align*}
&[{\varrho_\#}(x,y)]^\alpha [r(x)]^{-\alpha}2^{-m_x\varepsilon}\\
&\quad=[{\varrho_\#}(x,y)]^\alpha [r(x)]^{s+\delta-\alpha}[r(x)]^{-s-\delta}2^{-m_x\varepsilon}\\
&\quad\lesssim[{\varrho_\#}(x,y)]^\alpha [r(x)]^{s+\delta-\alpha}2^{m_x(s-\varepsilon+\delta)}\\
&\quad\lesssim[{\varrho_\#}(x,y)]^{s+\delta}2^{m_x(s-\varepsilon+\delta)}
\lesssim[{\varrho_\#}(x,y)]^{s}2^{(m_x-k)\delta+m_x(s-\varepsilon)},
\end{align*}
which, together with \eqref{nri-458-X}, allows us to write
\begin{equation}
\label{nri-458-X-X}
\lf|\mathcal{F}u(x)-\mathcal{F}u(y)\r|
\lesssim [{\varrho_\#}(x,y)]^{s}
\lf[\mathcal{M}_{\varrho_\#}\lf(\sup\limits_{j\geq m_x-\ell_0-k_0}
\lf\{2^{(m_x-k)\delta t+(m_x-j)(s-\varepsilon)t}g_j^t\r\}\r)(x)\r]^{1/t}.
\end{equation}
To bound the supremum in \eqref{nri-458-X-X}, observe that, by the
choice of $m_x$, the definition of $r(x)$,
the fact that $\Lambda>\theta=2C_{\varrho_\#}^2$ and the assumption of Case 4,
one has
$$2^{-m_x}>C_{\varrho_\#}\Lambda^2r(x)>C_{\varrho_\#}^4
{\rm dist}_{\varrho_\#}(x,\Omega)>{\varrho_\#}(x,y)\geq 2^{-k-1}.
$$
Hence, $m_x<k+1$. As   $\ell_0$ and $k_0$ are non-negative, we further have
 $m_x-\ell_0-k_0<m_x\leq k$. Using this and the fact that
 $m_x\leq j+\ell_0+k_0$, we can bound the supremum
 in \eqref{nri-458-X-X} as
\begin{align*}
&\lf[\mathcal{M}_{\varrho_\#}\lf(\sup\limits_{j\geq m_x-\ell_0-k_0}\lf\{2^{(m_x-k)\delta t+(m_x-j)(s-\varepsilon)t}g_j^t\r\}\r)(x)\r]^{1/t}\noz\\
&\quad=\lf[\mathcal{M}_{\varrho_\#}\lf(\sup\limits_{j\in[m_x-\ell_0-k_0,k-1]}\lf\{2^{(m_x-k)\delta t+(m_x-j)(s-\varepsilon)t}g_j^t\r\}\r)(x)\r]^{1/t}\noz\\
&\quad\quad+\lf[\mathcal{M}_{\varrho_\#}\lf(\sup\limits_{j\geq k}\lf\{2^{(m_x-k)\delta t+(m_x-j)(s-\varepsilon)t}g_j^t\r\}\r)(x)\r]^{1/t}\noz\\
&\quad\lesssim\lf[\mathcal{M}_{\varrho_\#}\lf(\sup\limits_{j\leq k-1}\lf\{2^{(j-k)\delta t}g_j^t\r\}\r)(x)\r]^{1/t}+\lf[\mathcal{M}_{\varrho_\#}\lf(\sup\limits_{j\geq k}\lf\{2^{(k-j)(s-\varepsilon)t}g_j^t\r\}\r)(x)\r]^{1/t}.
\end{align*}
For the first  supremum in the last line above,  since $j<k$,  we easily have
$$
\lf[\mathcal{M}_{\varrho_\#}\lf(\sup\limits_{j\leq k-1}\lf\{2^{(j-k)\delta t}g_j^t\r\}\r)(x)\r]^{1/t}=\lf[\mathcal{M}_{\varrho_\#}
\lf(\sup\limits_{j\leq k-1}\lf\{2^{-|j-k|\delta t}g_j^t\r\}\r)(x)\r]^{1/t}
\leq h_k(x),
$$
while, for the second one, we use the fact
that $k-j\leq0$ and $\delta<s-\varepsilon$ to obtain
\begin{align*}
&\lf[\mathcal{M}_{\varrho_\#}\lf(\sup\limits_{j\geq k}\lf\{2^{(k-j)(s-\varepsilon)t}g_j^t\r\}\r)(x)\r]^{1/t}\\
&\quad\leq\lf[\mathcal{M}_{\varrho_\#}\lf(\sup\limits_{j\geq k}\lf\{2^{(k-j)\delta t}g_j^t\r\}\r)(x)\r]^{1/t}
=\lf[\mathcal{M}_{\varrho_\#}\lf(\sup\limits_{j\geq k}\lf\{2^{-|k-j|\delta t}g_j^t\r\}\r)(x)\r]^{1/t}
\leq h_k(x).
\end{align*}
From these estimates   and \eqref{nri-458-X-X}, we deduce
the desired estimate \eqref{fracgradext}
in  Case 4.

Combining all obtained results in Cases 1-4, we complete
the proof of Lemma~\ref{claim2}.
\end{proof}

Recall that $\vec{h}:=\{h_k\}_{k\in\mathbb{Z}}$
[originally defined in \eqref{localgraddef}] has the following formula:  for any $x\in X$,
\begin{eqnarray}
\label{localgraddef-claim3}
h_k(x):=
\left\{
\begin{array}{ll}
\displaystyle \lf[\mathcal{M}_{\varrho_\#}\lf(\sup_{j\in\mathbb{Z}}
\lf\{2^{-|k-j|\delta t}g_j^t\r\}\r)(x)\r]^{1/t}\quad &\mbox{if $k\geq k_0$,}\\[20pt]
2^{(k+1)s}|\mathcal{F}u(x)| &\mbox{if $k<k_0$.}
\end{array}
\right.
\end{eqnarray}
We then have the following result.

\begin{lemma}
\label{claim3}
Let all the notation be as in the above proof of   Theorem~\ref{measext}.
The sequence $\vec{h}$ defined in \eqref{localgraddef-claim3} satisfies
\begin{equation*}
\lf\Vert\vec{h}\r\Vert_{L^p(V,\ell^q)}\lesssim\lf\|\vec{g}\r\|_{L^p(\Omega,\ell^q)}
+\left\Vert u\right\Vert_{L^p(\Omega,\mu)},
\end{equation*}
where the implicit positive constant is independent of $u$, $\vec g$, and $\vec h$.
\end{lemma}

\begin{proof}
Suppose first that $q<\infty$ and recall that $\delta>0$,
by design, in this case. Dominating the supremum by
the corresponding summation in the definition of $h_k$ in
\eqref{localgraddef-claim3} and using the sublinearity
of $\mathcal{M}_{\varrho_\#}$, we find that
\begin{equation}
\label{tby-386}
\lf[\mathcal{M}_{\varrho_\#}\lf(\sup\limits_{j\in\mathbb{Z}}\lf\{2^{-|k-j|\delta t}g_j^t\r\}\r)\r]^{1/t}\leq
\left[\sum_{j\in\mathbb{Z}}2^{-|k-j|\delta t}\mathcal{M}_{\varrho_\#}\lf(g_j^t\r)\right]^{1/t}.
\end{equation}
Using this and $s>0$, and applying Lemma~\ref{heli-est}   with $a=2^{\delta t}$, $b=q/t$,
and $c_j=\mathcal{M}_{\varrho_\#}(g_j^t)$, one has
\begin{align*}
\left(\sum_{k\in\mathbb{Z}}|h_k|^q\right)^{1/q}
&=\left\{\sum_{k=k_0}^\infty\lf[\sum_{j\in\mathbb{Z}}
2^{-|k-j|\delta t}\mathcal{M}_{\varrho_\#}\lf(g_j^t\r)\r]^{q/t}
+|\mathcal{F}u|^q\sum_{k=-\infty}^{k_0-1}2^{(k+1)sq}\right\}^{1/q}
\\
&\lesssim\left\{\sum_{j\in\mathbb{Z}}
\lf[\mathcal{M}_{\varrho_\#}\lf(g_j^t\r)\r]^{q/t}+2^{k_0sq}|\mathcal{F}u|^q\right\}^{1/q}\\
&\lesssim\left\{\sum_{j\in\mathbb{Z}}\lf[\mathcal{M}_{\varrho_\#}\lf(g_j^t\r)\r]^{q/t}
\right\}^{1/q}+|\mathcal{F}u|,
\end{align*}
which, together with
the Fefferman--Stein inequality (see \cite[p.~4, Theorem~1.2]{Grafk})
as well as  Lemma~\ref{claim1}, further implies that
\begin{align}
\label{abru-572}
\lf\Vert\vec{h}\r\Vert_{L^p(V,\ell^q)}
&=\left\Vert\left(\sum_{k\in\mathbb{Z}}|h_k|^q\right)^{1/q}\right\Vert_{L^p(V,\mu)}
\lesssim\left\Vert\left\{\sum_{j\in\mathbb{Z}}
\lf[\mathcal{M}_{\varrho_\#}\lf(g_j^t\r)\r]^{q/t}\right\}^{t/q}\right\Vert_{L^{p/t}(X,\mu)}^{1/t}+
\left\Vert\mathcal{F}u\right\Vert_{L^p(V,\mu)}
\noz\\
&\lesssim\left\Vert\left(\sum_{j\in\mathbb{Z}}|g_j|^q\right)^{t/q}\right\Vert_{L^{p/t}(X,\mu)}^{1/t}+
\left\Vert\mathcal{F}u\right\Vert_{L^p(V,\mu)}
\noz\\
&\lesssim\left\Vert\left(\sum_{j\in\mathbb{Z}}|g_j|^q\right)^{1/q}\right\Vert_{L^{p}(X,\mu)}+
\left\Vert u\right\Vert_{L^p(\Omega,\mu)}
\lesssim\lf\|\vec{g}\r\|_{L^p(\Omega,\ell^q)}+\left\Vert u\right\Vert_{L^p(\Omega,\mu)},
\end{align}
where we have used the fact that for any $j\in\zz$,  $g_j\equiv0$  on $X\setminus\Omega$ in
obtaining the last inequality of \eqref{abru-572}.
Also, note that we have used the assumption
$\min\{q/t,p/t\}>1$  in applying the Fefferman--Stein inequality.

Suppose next that $q=\infty$. Then $\delta\geq0$ and we
use the fact that $2^{-|k-j|\delta t}\leq1$ to estimate
\begin{equation*}
\sup_{k\in\mathbb{Z}}|h_k|
\leq\lf[\mathcal{M}_{\varrho_\#}\lf(\sup_{j\in\mathbb{Z}}g_j^t\r)\r]^{1/t}+2^{k_0s}|\mathcal{F}u|.
\end{equation*}
Note that, by the definition
of $\dot{M}^s_{p,\infty}(X,\mu)$,
we have $\sup_{j\in\mathbb{Z}}g_j^t\in L^{p/t}(X,\mu)$, where $p/t>1$. From these observations,
Lemma~\ref{claim1}, and the boundedness of $\mathcal{M}_{\varrho_\#}$
on $L^{p/t}(X,\mu)$,  it follows that
\begin{align*}
\lf\Vert\vec{h}\r\Vert_{L^p(V,\ell^\infty)}
&=\left\Vert\sup_{k\in\mathbb{Z}}|h_k|\right\Vert_{L^p(V,\mu)}
\lesssim\left\Vert\mathcal{M}_{\varrho_\#}\lf(\sup_{j\in\mathbb{Z}}g_j^t\r)\right\Vert_{L^{p/t}(X,\mu)}^{1/t}+
\left\Vert\mathcal{F}u\right\Vert_{L^p(V,\mu)}
\noz\\
&\lesssim\left\Vert\sup_{j\in\mathbb{Z}}g_j^t\right\Vert_{L^{p/t}(X,\mu)}^{1/t}+
\left\Vert\mathcal{F}u\right\Vert_{L^p(V,\mu)}
\noz\\
&\lesssim\left\Vert\sup_{j\in\mathbb{Z}}g_j\right\Vert_{L^{p}(X,\mu)}+
\left\Vert u\right\Vert_{L^p(\Omega,\mu)}
\lesssim\lf\|\vec{g}\r\|_{L^p(\Omega,\ell^\infty)}
+\left\Vert u\right\Vert_{L^p(\Omega,\mu)}.
\end{align*}
This finishes the proof of Lemma~\ref{claim3}.
\end{proof}

Combining Lemmas~\ref{claim1}, \ref{claim2}, and \ref{claim3}, we conclude that
\begin{equation}
\label{localextest}
\mathcal{F}u\in M^s_{p,q}(V)\quad\mbox{ and }\quad\Vert\mathcal{F}u\Vert_{M^s_{p,q}(V)}
\lesssim\Vert u\Vert_{M^s_{p,q}(\Omega)},
\end{equation}	
that is, $\mathcal{F}$ serves as a bounded extension operator from $\Omega$ to $V$
for functions in Haj\l asz--Triebel--Lizorkin spaces.

We now define the final extension of $u$ to the entire space $X$. Let $\Psi\colon X\to[0,1]$
be any H\"older continuous function of order $\alpha$ on $X$ such that $\Psi\lfloor_\Omega\equiv1$
and $\Psi\lfloor_{X\setminus V}\equiv0$, where
here and thereafter, for any set $E\subset X$, the symbol $\Psi\lfloor_E$ means
the restriction of $\Psi$ on $E$.
The existence of such a function having this order with
these properties is guaranteed by \cite[Theorem~4.1]{AMM13}
(see also \cite[Theorem~2.6]{AM15}). Define $\widetilde{u}\colon  X\to\mathbb{R}$ by setting
$$\widetilde{u}:=\Psi\mathcal{F}u.$$
Then $\widetilde{u}\lfloor_\Omega\equiv u$ and  it follows from Lemma~\ref{GVa2-prev} and \eqref{localextest} that
$$
\Vert\widetilde{u}\Vert_{M^s_{p,q}(X)}\lesssim\Vert\mathcal{F}u\Vert_{M^s_{p,q}(V)}
\lesssim\Vert u\Vert_{M^s_{p,q}(\Omega)}.
$$
This finishes the proof of \eqref{II+kan} in the statement of Theorem \ref{measext} and,
in turn, the proof of this theorem for  $M^s_{p,q}$ spaces.

Finally, let us turn our attention to
extending Haj\l{}asz--Besov functions
when $s<{\rm ind}\,(X,\rho)$. To construct an
$N^s_{p,q}$-extension operator, we will proceed
just as we did in the Triebel--Lizorkin case with
a few key differences. More specifically,
since $s<{\rm ind}\,(X,\rho)$, we can choose a number $\alpha\in\mathbb{R}$
and a quasi-metric $\varrho$ on $X$
such that $\varrho\approx\rho$ and $s<\alpha\leq(\log_{2}C_\varrho)^{-1}$.
For any $u\in N^s_{p,q}(\Omega,\varrho_\#,\mu)$,
we define the local
extension $\mathcal{F}u$ as in \eqref{localextdef}. Then
Lemma~\ref{claim2} still gives an $s$-fractional gradient
of $\mathcal{F}u$ that is given (up to a multiplicative constant) by \eqref{localgraddef},
where $\vec{g}\in\mathbb{D}^s_{\varrho_\#}(u)$ is such that
$\Vert\vec{g}\Vert_{\ell^q(L^p(\Omega))}\lesssim \Vert u\Vert_{\dot{N}^s_{p,q}(\Omega)}$.
By Proposition~\ref{constant}, such a choice of $\vec{g}$ is possible.
Moreover, we have that $\delta>0$ by its definition, because $\alpha\neq s$.
In place of the norm estimate in Lemma~\ref{claim3}, we use the following lemma.

\begin{lemma}
\label{claim3-besov}
The sequence $\vec{h}:=\{h_k\}_{k\in\mathbb{Z}}$ defined in \eqref{localgraddef} satisfies
\begin{equation*}
\lf\Vert\vec{h}\r\Vert_{\ell^q(L^p(V))}\lesssim \lf\|\vec{g}\r\|_{\ell^q(L^p(\Omega))}+\left\Vert u\right\Vert_{L^p(\Omega,\mu)},
\end{equation*}
where the implicit positive constant is independent of $u$, $\vec g$, and $\vec h$.
\end{lemma}

\begin{proof}
By the definition of
$N^s_{p,q}(X,\mu)$, we have $g_j^t\in L^{p/t}(X,\mu)$ for
any $j\in\mathbb{Z}$. Since $p/t>1$,
using \eqref{tby-386}, the Minkowski inequality, and the boundedness
of $\mathcal{M}_{\varrho_\#}$ on $L^{p/t}(X,\mu)$, one has, for any $k\in\mathbb{Z}$,
\begin{align}
\label{mxiw-547-1}
&\left\Vert\lf[\mathcal{M}_{\varrho_\#}\lf(\sup\limits_{j\in\mathbb{Z}}\lf\{2^{-|k-j|\delta t}g_j^t\r\}\r)\r]^{1/t}\right\Vert_{L^{p}(X,\mu)}\noz\\
&\quad\leq\left\Vert\sum_{j\in\mathbb{Z}}2^{-|k-j|\delta t}\mathcal{M}_{\varrho_\#}\lf(g_j^t\r)\right\Vert_{L^{p/t}(X,\mu)}^{1/t}
\leq\left[\sum_{j\in\mathbb{Z}}2^{-|k-j|\delta t}\lf\Vert\mathcal{M}_{\varrho_\#}\lf(g_j^t\r)\r\Vert_{L^{p/t}(X,\mu)}\right]^{1/t}\noz\\
&\quad\lesssim\left[\sum_{j\in\mathbb{Z}}2^{-|k-j|\delta t}
\,\Vert g_j\Vert_{L^{p}(X,\mu)}^t\right]^{1/t}.
\end{align}
Therefore, if $q<\infty$, then Lemma~\ref{heli-est}
(used here with $a=2^{\delta t}$, $b=q/t$, and $c_j=\Vert g_j\Vert_{L^{p}(X,\mu)}^t$)
and Lemma~\ref{claim1} give
\begin{align*}
\lf\Vert\vec{h}\r\Vert_{\ell^q(L^p(V))}
&=\left\{\sum_{k=k_0}^\infty\left\Vert\lf[\mathcal{M}_{\varrho_\#}
\lf(\sup\limits_{j\in\mathbb{Z}}\lf\{2^{-|k-j|\delta t}g_j^t\r\}\r)\r]^{1/t}\right\Vert_{L^{p}(X,\mu)}^q+\Vert\mathcal{F}u\Vert_{L^{p}(V,\mu)}^q
\sum_{k=-\infty}^{k_0-1}2^{(k+1)sq}\right\}^{1/q}
\\
&\leq\left\{\sum_{k\in\mathbb{Z}}\lf[\sum_{j\in\mathbb{Z}}2^{-|k-j|\delta t}\,\Vert g_j\Vert_{L^{p}(X,\mu)}^t\r]^{q/t}+\Vert\mathcal{F}u\Vert_{L^{p}(V,\mu)}^q2^{k_0sq}\right\}^{1/q}
\\
&\lesssim\left\{\sum_{j\in\mathbb{Z}}\Vert g_j\Vert_{L^{p}(X,\mu)}^q+\Vert\mathcal{F}u\Vert_{L^{p}(V,\mu)}^q\right\}^{1/q}
\lesssim\|\vec{g}\|_{\ell^q(L^p(\Omega))}+\Vert u\Vert_{L^{p}(\Omega,\mu)},
\end{align*}
where we have used the fact that, for any $j\in\zz$, $g_j\equiv0$ on $X\setminus\Omega$. If $q=\infty$, then,
by \eqref{mxiw-547-1}, we find that (keeping in mind that $\delta>0$)
\begin{align*}
\lf\Vert\vec{h}\r\Vert_{\ell^\infty(L^p(V))}
&\leq\sup_{k\geq k_0}\left\Vert\lf[\mathcal{M}_{\varrho_\#}
\lf(\sup\limits_{j\in\mathbb{Z}}\lf\{2^{-|k-j|\delta t}g_j^t\r\}\r)\r]^{1/t}\right\Vert_{L^{p}(X,\mu)}
+\Vert\mathcal{F}u\Vert_{L^{p}(V,\mu)}\sup_{k\leq k_0-1}2^{(k+1)s}
\noz\\
&\leq\sup_{k\geq k_0}\left[\sum_{j\in\mathbb{Z}}2^{-|k-j|\delta t}\,\Vert g_j\Vert_{L^{p}(X,\mu)}^t\right]^{1/t}+\Vert\mathcal{F}u\Vert_{L^{p}(V,\mu)}2^{k_0s}
\noz\\
&\lesssim\|\vec{g}\|_{\ell^\infty(L^p(X))}\sup_{k\geq k_0}\lf[\sum_{j\in\mathbb{Z}}2^{-|k-j|\delta t}\r]^{1/t}+\Vert u\Vert_{L^{p}(\Omega,\mu)}\\
&\lesssim\lf\|\vec{g}\r\|_{\ell^\infty(L^p(\Omega))}+\Vert u\Vert_{L^{p}(\Omega,\mu)}.
\end{align*}
This finishes the proof of Lemma~\ref{claim3-besov}.
\end{proof}

The remainder of the proof of Theorem \ref{measext}  for $N^s_{p,q}$ spaces,
including the definition for the final extension of $u$ to
the entire space $X$, is carried out as
it was in the Triebel--Lizorkin case. This finishes the proof of Theorem~\ref{measext}.
\end{proof}

\section{Equivalency of The Extension and Embedding Properties of The Domain and The Measure Density Condition}
\label{section:measuredensity}
In this section, we  fully characterize the $M^s_{p,q}$ and the $N^s_{p,q}$ extension domains which are locally uniformly perfect
via the so-called measure density condition in general spaces of homogeneous
type, for an optimal range of $s$;
see Theorem~\ref{measdens-ext-sob} below.

\subsection{Main Tools}
Before formulating the main theorems in this section, we  first collect
a few technical results that were established in \cite{agh20} and \cite{AYY21}.
Given the important role that these results play in the proofs of the main theorems in this article,
we include their statements for the convenience of the reader.

We begin with a few definitions. Throughout this section, let $(X,\rho,\mu)$ be a quasi-metric
measure space and suppose that $\Omega\subset X$ is a $\mu$-measurable
set with at least two points and has the property that
$\mu(B_\rho(x,r)\cap\Omega)>0$ for any $x\in\Omega$ and
$r\in(0,1]$ \footnote{The upper bound   1 on the radius can be arbitrarily
replaced by any other threshold $\delta\in(0,\infty)$.}.
Then $(\Omega,\rho,\mu)$ can be naturally viewed
as a quasi-metric measure space, where the quasi-metric $\rho$
and the measure $\mu$ are restricted to the set $\Omega$.
In this context, for any $x\in X$ and $r\in[0,\infty)$, we let
\begin{equation}
\label{radphiomeg}
\mbox{$\vi^x_{\Omega,\rho}(r):=\sup\lf\{s\in [0,r]:\, \mu(B_\rho(x,s)\cap\Omega)\leq\frac{1}{2}\mu(B_\rho(x,r)\cap\Omega)\r\}$.}
\end{equation}
Note that, when $s=0$, $B_\rho(x,s)\cap\Omega=\emptyset$
and so $\varphi^x_{\Omega,\rho}(r)\geq 0$. When $\Omega=X$, we
abbreviate $\vi^x_{\rho}:=\vi^x_{X,\rho}$.

In the Euclidean setting $(\mathbb{R}^n,|\cdot-\cdot|,\mathcal{L}^n)$, where $\mathcal{L}^n$
denotes the $n$-dimensional Lebesgue measure, for any $x\in\Omega$ and $r\in(0,\infty)$,
it is always possible to find a radius $\tilde{r}<r$ with the property that
$\mathcal{L}^n(B(x,\tilde{r})\cap\Omega)=\frac{1}{2}\mathcal{L}^n(B(x,r)\cap\Omega))$,
whenever $\Omega\subset\mathbb{R}^n$ is an open connected set.
In a general quasi-metric measure space there is no guarantee
that there exists such a radius $\tilde{r}$;
however, $\vi^x_{\Omega,\rho}(r)$ will  be a suitable
replacement of $\tilde{r}$ in this more general setting.

In the context above, the quasi-metric space $(\Omega,\rho)$
(or simply the set $\Omega$) is said to be
\textit{locally uniformly perfect}  if there exists a
constant $\lambda\in(0,1)$ with the property that, for any $x\in\Omega$
and   $r\in(0,1]$ \footnote{The upper bound   1 can be replaced by any positive and finite number.},
\begin{equation}
\label{U-perf}
\lf(B_\rho(x,r)\cap\Omega\r)\setminus B_\rho(x,\lambda r)\neq\emptyset\quad
\mbox{ whenever }\quad \Omega\setminus B_\rho(x,r)\neq\emptyset.
\end{equation}
If \eqref{U-perf} holds true
for any $x\in\Omega$ and $r\in(0,\infty)$, then $\Omega$ is
simply referred as \textit{uniformly perfect}.
It is easy to see that every connected quasi-metric space is uniformly perfect;
however, there are very disconnected Cantor-type sets that are also uniformly perfect.
Moreover, every uniformly perfect set is locally uniformly perfect,
but the set $\Omega:=\bigcup_{n=0}^\infty[2^{2^n},2^{2^n}+1)$,
when equipped with the restriction of the standard Euclidean distance on $\mathbb{R}$,
is a locally uniformly perfect set which is not uniformly
perfect (see \cite{gs20} for additional details).
Finally, observe that it follows immediately from \eqref{U-perf} that,
whenever \eqref{U-perf} holds true for some $\lambda\in(0,1)$,
then it holds true for any $\lambda'\in(0,\lambda]$.

The first main result of this subsection is a technical
lemma which is needed in proving that $M^s_{p,q}$
and $N^s_{p,q}$ extension domains necessarily satisfy the measure density condition.

\begin{lemma}
\label{en2-4}
Let $(X,\rho,\mu)$ be a quasi-metric measure space,
where $\mu$ is a doubling measure on $X$, and suppose that $\Omega\subset X$ is a
$\mu$-measurable locally uniformly perfect set.
Let $\lambda\in(0,(C_{\rho\lfloor_\Omega}\widetilde{C}_{\rho\lfloor_\Omega})^{-2})$ be as
in \eqref{U-perf}, where ${\rho\lfloor_\Omega}$ denotes the restriction
of $\rho$ to $\Omega$, and assume that there exists a positive constant $C$ such that
\begin{equation}
\label{eq48omega}
\mu(B_\rho(x,r))\leq C\mu(B_\rho(x,r)\cap\Omega),
\end{equation}
whenever both $x\in \Omega$ and $r\in(0,1]$ satisfy $r\leq C_{\rho\lfloor_\Omega}\varphi^x_{\Omega,\rho}(r)/\lambda^2$.
Then there exists a positive constant $\widetilde{C}$, depending only on $C$, $\rho$, $\lambda$,
and the doubling constant for $\mu$, such that

\begin{equation}
\label{eq48omega-2}
\mu(B_\rho(x,r))\leq \widetilde{C}\mu(B_\rho(x,r)\cap\Omega)
\end{equation}
holds true  for any $x\in X$ and $r\in(0,1]$.
\end{lemma}

We will employ the following result from \cite[Lemma 4.2]{AYY21}
in the proof of Lemma~\ref{en2-4}. Strictly speaking, the result in \cite[Lemma 4.2]{AYY21} was proven for uniformly perfect spaces; however, the same proof is valid for the class of locally uniformly perfect spaces.

\begin{lemma}\label{tx}
Let $(X,\rho,\mu)$ be a locally uniformly perfect quasi-metric measure space and suppose that $\lambda\in (0,(C_\rho\widetilde{C}_\rho)^{-2})$ is the constant appearing in \eqref{U-perf},
where $C_\rho,\,\widetilde{C}_\rho\in[1,\infty)$ are as in \eqref{C-RHO.111} and \eqref{C-RHO.111XXX}, respectively.
If $x\in X$ and $r\in (0,{\rm \diam}_\rho(X)]$ satisfy $C_\rho\varphi^x_{\rho}(r)/\lambda^2<r\leq1$, then there exist a point $\tilde{x}\in X$ and a radius $\tilde{r}\in(0,1]$ such that $\lambda r\leq\tilde{r}\leq\min\{r,C_\rho\varphi^{\tilde{x}}_{\rho}(\tilde{r})/\lambda^2\}$ and $B_\rho(\tilde{x},\tilde{r})\subset B_\rho(x,r)$.
\end{lemma}

\begin{proof}[Proof of Lemma~\ref{en2-4}]
We first remark that, by the definition of a locally uniformly perfect set,
we conclude that $(\Omega,\rho\lfloor_\Omega,\mu\lfloor_\Omega)$ is
a locally uniformly perfect quasi-metric-measure space,
where the quasi-metric $\rho$ and the measure $\mu$ are restricted to
the set $\Omega$.
As such, we can invoke Lemma~\ref{tx} for the quasi-metric measure
space $(\Omega,\rho\lfloor_\Omega,\mu\lfloor_\Omega)$.

With this in mind, fix a point $x\in\Omega$ and a radius $r\in(0,1]$.
If $r\leq C_{\rho\lfloor_\Omega}\varphi^x_{\Omega,\rho}(r)/\lambda^2$,
then we immediately have $C\mu(B_\rho(x,r)\cap\Omega)\geq\mu(B_\rho(x,r))$ by
the assumption of Lemma~\ref{en2-4}.
Thus, it suffices to assume that
$r> C_{\rho\lfloor_\Omega}\varphi^x_{\Omega,\rho}(r)/\lambda^2$.
Observe that \eqref{eq48omega} implies that $C\geq1$. Therefore,
if $r>{\rm diam}_\rho(\Omega)$, then $B_\rho(x,r)=\Omega$ and \eqref{eq48omega-2}
holds true trivially.
Thus, we may further assume that
$r\leq{\rm diam}_\rho(\Omega)$. By Lemma~\ref{tx},
there exist a
point $\tilde{x}\in\Omega$ and a
radius $\tilde{r}\in(0,1]$ such that
$\lambda r\leq\tilde{r}\leq\min\{r,C_{\rho\lfloor_\Omega}
\varphi^{\tilde{x}}_{\Omega,\rho}(\tilde{r})/\lambda^2\}$ and
$B_\rho(\tilde{x},\tilde{r})\cap\Omega\subset B_\rho(x,r)\cap\Omega$.
Since $\tilde{x}\in B_\rho(x,r)$, we infer that
$B_\rho(x,r)\subset B_\rho(\tilde{x},C_{\rho}\widetilde{C}_\rho r)$,
which, together with \eqref{eq48omega} and the doubling property for $\mu$, implies
that
\begin{align*}
C\mu(B_\rho(x,r)\cap\Omega)
&\geq C\mu(B_\rho(\tilde{x},\tilde{r})\cap\Omega)
\geq \mu(B_\rho(\tilde{x},\tilde{r}))\noz\\
&\geq\mu(B_\rho(\tilde{x},\lambda r))\gtrsim\mu(B_\rho(\tilde{x},C_{\rho}\widetilde{C}_\rho r))
\gtrsim \mu(B_\rho(x,r)).
\end{align*}
This finishes the proof of Lemma~\ref{en2-4}.
\end{proof}

The following result was recently established in \cite[Lemma~4.7]{AYY21} and
it concerns the ability to construct certain
families of maximally smooth H\"older continuous `bump' functions belonging to $M^{s}_{p,q}$ and $N^{s}_{p,q}$.

\begin{lemma}\label{HolderBump}
Suppose $(X,\rho,\mu)$ is a quasi-metric measure space and $\varrho$
any quasi-metric on $X$ such that $\varrho\approx\rho$.
Let $C_\varrho\in[1,\infty)$ be as in \eqref{C-RHO.111} and fix
exponents $p\in(0,\infty)$ and $q\in(0,\infty]$, along with
a finite number $s\in(0,(\log_{2}C_\varrho)^{-1}]$,
where the value  $s=(\log_{2}C_\varrho)^{-1}$
is only permissible when $q=\infty$. Also suppose
that $\varrho_{\#}$ is the regularized quasi-metric
given by Theorem~\ref{DST1}. Then there exists a
number $\delta\in(0,1)$, depending only on $s$, $\varrho$, and the proportionality
constants in $\varrho\approx\rho$, such that, for any $x\in X$ and any finite
$r\in(0,{\rm diam}_\rho(X)]$,
there exist a sequence $\{r_j\}_{j\in\mathbb{N}}$ of radii
and a collection $\{u_j\}_{j\in\mathbb{N}}$ of functions  such that, for any $j\in\mathbb{N}$,
\begin{enumerate}[label=\rm{(\alph*)}]
\item $B_{\varrho_{\#}}(x,r_j)\subset B_{\rho}(x,r)$ \,\,\mbox{and }\,\,$\delta r<r_{j+1}<r_j<r$;
\item $0\leq u_j\leq 1$ pointwise on $X$;
\item $u_j\equiv 0$ on $X\setminus B_{\varrho_{\#}}(x,r_j)$;
\item $u_j\equiv 1$ on $\overline{B}_{\varrho_{\#}}(x,r_{j+1})$;
\item $u_j$ belongs to both $M^s_{p,q}(X,\rho,\mu)$ and $N^s_{p,q}(X,\rho,\mu)$, and
$$0<\Vert u_j\Vert_{\dot{M}^s_{p,q}(X,\rho,\mu)}\leq C 2^{j}r^{-s}\lf[\mu(B_{\varrho_{\#}}(x,r_j))\r]^{1/p}$$
and
$$0<\Vert u_j\Vert_{\dot{N}^s_{p,q}(X,\rho,\mu)}\leq C 2^{j}r^{-s}\lf[\mu(B_{\varrho_{\#}}(x,r_j))\r]^{1/p}
$$
for some positive constant $C$ which is independent of $x$, $r$, and $j$.
\end{enumerate}

Furthermore, if there exists a constant $c_0\in(1,\infty)$
such that $r\leq c_0\varphi^x_{\rho}(r)$, where
$\varphi^x_{\rho}(r)$ is as in \eqref{radphiomeg},
then the radii $\{r_j\}_{j\in\mathbb{N}}$ and the
functions $\{u_j\}_{j\in\mathbb{N}}$ can be chosen to
have the following property in addition
to  {\rm(a)}-{\rm(e)}  above:
\begin{enumerate}
\item[{\rm(f)}] for any fixed $j\in\mathbb{N}$ and $\gamma\in(0,\infty)$, there exists a $\mu$-measurable set $E^\gamma_j\subset B_\rho(x,r)$ such that $\mu(E^\gamma_j)\geq\mu(B_{\varrho_{\#}}(x,r_{j+1}))$ and $|u_j-\gamma\,|\geq\frac{1}{2}$ pointwise on $E^\gamma_j$.
\end{enumerate}
In this case, the number $\delta$   depends also on $c_0$.
\end{lemma}

\begin{remark}
In the context of Lemma~\ref{HolderBump}, we  say that both
the sequence $\{r_j\}_{j\in\mathbb{N}}$ of radii and the collection $\{u_j\}_{j\in\mathbb{N}}$ of functions  are associated to the ball $B_\rho(x,r)$.
\end{remark}

The last result in this subsection is an abstract
iteration scheme that will be applied many times in the proofs of our main results (see also \cite[Lemma 4.9]{AYY21}).
A version of this lemma in the setting of metric-measure spaces was proven in \cite[Lemma~16]{agh20},
which itself is an abstract version of an argument used in \cite{korobenkomr}.

\begin{lemma}
\label{iteration}
Let $(X,\rho,\mu)$ be a quasi-metric measure space.
Suppose that $a,\,b,\,p,\,t,\,\theta\in(0,\infty)$ satisfy $a<b$ and $p<t$.
Let  $x\in X$ and $\{r_j\}_{j\in\mathbb{N}}$  be  a sequence of radii  such that
\begin{equation*}
a\leq r_j\leq b
\quad
\text{and}
\quad
\lf[\mu(B_\rho(x,r_{j+1}))\r]^{1/t}\leq\theta 2^j \lf[\mu(B_\rho(x,r_j))\r]^{1/p},
\qquad
\forall\ j\in\mathbb{N}.
\end{equation*}
Then
$$\mu(B_\rho(x,r_{1}))\geq \theta^{-pt/(t-p)}\,2^{-pt^2/(t-p)^2}.$$
\end{lemma}

\subsection{Principal Results}

The stage is now set to prove our next main result which, among other things, includes a converse to Theorem~\ref{measext}.
The reader is first reminded of the
following pieces of terminology: Given a quasi-metric measure space
$(X,\rho,\mu)$ and an exponent $Q\in(0,\infty)$,
a measure $\mu$ is said to be \emph{$Q$-doubling (on $X$)} provided that there exists a positive constant $\kappa$
such that
\begin{equation}\label{Doub-2-eh}
\kappa\lf(\frac{r}{R}\r)^{Q}\leq\frac{\mu(B_\rho(x,r))}{\mu(B_\rho(y,R))},
\end{equation}
whenever $x,\,y\in X$ and $0<r\leq R<\infty$ satisfy $B_\rho(x,r)\subset B_\rho(y,R)$,
and $\mu$ is said to be \emph{$Q$-Ahlfors regular (on $X$)}  if there exists a positive  constant $\kappa$
such that
\begin{equation*}
\kappa^{-1}\,r^Q\leq\mu\lf(B_\rho(x,r)\r)\leq\kappa r^Q
\end{equation*}
for any $x\in X$ and any finite $r\in (0,{\rm diam}_\rho(X)]$.

\begin{theorem}
\label{measdens-ext-sob}
Let $(X,\rho,\mu)$ be a quasi-metric
measure space where $\mu$ is a Borel
regular $Q$-doubling measure for some $Q\in(0,\infty)$,
and suppose that $\Omega\subset X$ is a
$\mu$-measurable locally uniformly perfect
set. Then  the following statements are equivalent.
\begin{enumerate}[label={\rm(\alph*)}]
\item The measure $\mu$ satisfies the following  measure
density condition on $\Omega$:\ there exists a positive constant $C_\mu$ such that
\begin{equation}
\label{measdensity}
\mu(B_\rho(x,r))\leq C_\mu\,\mu(B_\rho(x,r)\cap\Omega)\,\,\,\mbox{ for any }\,x\in \Omega\mbox{ and  }r\in(0,1].
\end{equation}
		
\item The set $\Omega$ is an $M^s_{p,q}$-extension domain
for some $p\in(0,\infty)$, $q\in(0,\infty]$, and $s\in(0,\infty)$
satisfying $s\preceq_q{\rm ind}\,(X,\rho)$, in the sense that there exists a positive constant $C$
such that,
\begin{eqnarray}
\label{ext1}
\begin{array}{c}
\mbox{for any }\,u\in M^s_{p,q}(\Omega,\rho,\mu),
\,\,\mbox{ there exists a\ }\ \widetilde{u}\in M^s_{p,q}(X,\rho,\mu)
\\[6pt]
\mbox{ for which }\,\,\, u=\widetilde{u}|_{\Omega}\,\,
\mbox{ and }\,\,\,\,\|\widetilde{u}\|_{M^s_{p,q}(X,\rho,\mu)}
\leq C\|u\|_{M^s_{p,q}(\Omega,\rho,\mu)}.
\end{array}
\end{eqnarray}

\item There exist $s,\,p\in(0,\infty)$ and a $q\in(0,\infty]$, satisfying
$s\preceq_q{\rm ind}\,(X,\rho)$ and $p,\,q>Q/(Q+s)$,
and a linear and bounded operator $\mathscr{E}\colon  M^s_{p,q}(\Omega,\rho,\mu)\to M^s_{p,q}(X,\rho,\mu)$
with the property that $(\mathscr{E}u)|_{\Omega}=u$ for any $u\in M^s_{p,q}(\Omega,\rho,\mu)$.

\item There exist an $s\in(0,\infty)$, a $p\in(0,Q/s)$, a $q\in(0,\infty]$,
and a $C_S\in(0,\infty)$  satisfying $s\preceq_q{\rm ind}(\Omega,\rho)$ such that, for any ball $B_0:=B_\rho(x_0,R_0)$ with $x_0\in\Omega$
and $R_0\in(0,1]$,
\begin{equation}
\label{ptc}
\Vert u\Vert_{L^{p^\ast}(B_0\cap\Omega)}\leq
\frac{C_S}{[\mu(B_0)]^{s/Q}}
\left[R_0^s\Vert u\Vert_{\dot{M}^s_{p,q}(\Omega,\rho,\mu)}+\Vert u\Vert_{L^{p}(\Omega)}\right],
\end{equation}
whenever $u\in\dot{M}^s_{p,q}(\Omega,\rho,\mu)$. Here, $p^*:=Qp/(Q-sp)$.

\item There exist an $s\in(0,\infty)$, a $p\in(0,Q/s)$, a $q\in(0,\infty]$,
and a $C_P\in(0,\infty)$ satisfying $s\preceq_q{\rm ind}(\Omega,\rho)$ such that,
for any ball $B_0:=B_\rho(x_0,R_0)$ with $x_0\in\Omega$ and $R_0\in(0,1]$,
\begin{equation}
\label{ptc-poin}
\inf_{\gamma\in\mathbb{R}}\Vert u-\gamma\Vert_{L^{p^\ast}(B_0\cap\Omega)}\leq
\frac{C_PR_0^s}{[\mu(B_0)]^{s/Q}}\,\Vert u\Vert_{\dot{M}^s_{p,q}(\Omega,\rho,\mu)},
\end{equation}
whenever $u\in\dot{M}^s_{p,q}(\Omega,\rho,\mu)$. Here, $p^*:=Qp/(Q-sp)$.

\item There exist  constants $C_1,\,C_2,\,\omega\in(0,\fz)$, $q\in(0,\infty]$,
and $s\in(0,\infty)$ satisfying $s\preceq_q{\rm ind}(\Omega,\rho)$ such that
\begin{equation*}
\int_{B_0\cap\Omega} {\rm exp}\left(C_1\frac{[\mu(B_0)]^{s/Q}|u-u_{B_0\cap\Omega}|}{R_0^s\Vert u\Vert_{\dot{M}^s_{Q/s,q}(\Omega,\rho,\mu)}}\right)^\omega\,d\mu\leq C_2\mu(B_0),
\end{equation*}
whenever $B_0$ is a $\rho$-ball centered in $\Omega$ having radius $R_0\in(0,1]$, and $u\in \dot{M}^{s}_{p,q}(\Omega,\rho,\mu)$ with $\Vert u\Vert_{\dot{M}^s_{p,q}(\Omega,\rho,\mu)}>0$.

\item There exist an $s\in(0,\infty)$, a $p\in(Q/s,\infty)$, a $q\in(0,\infty]$, and a $C_H\in(0,\infty)$ satisfying
$s\preceq_q{\rm ind}(\Omega,\rho)$
such that, for any ball $B_0:=B_\rho(x_0,R_0)$ with $x_0\in\Omega$ and $R_0\in(0,1]$,
every function $u\in \dot{M}^s_{p,q}(\Omega,\rho,\mu)$
has a H\"older continuous representative of order $s-Q/p$ on $B_0\cap \Omega$, denoted by $u$ again, satisfying
\begin{equation}
\label{ptc-hold}
|u(x)-u(y)|\leq C_H\,[\rho(x,y)]^{s-Q/p}\frac{R_0^{Q/p}}{[\mu(B_0)]^{1/p}}\,\Vert u\Vert_{\dot{M}^s_{p,q}(\Omega,\rho,\mu)},
\quad\forall\,x,\,y\in B_0\cap \Omega.
\end{equation}
\end{enumerate}
If the measure $\mu$ is actually $Q$-Ahlfors-regular on $X$ for some $Q\in(0,\infty)$,
then the following statements are also equivalent to each of {\rm(a)-(g)}.
\begin{enumerate}[label={\rm(\alph*)}]\addtocounter{enumi}{7}
\item There exist an $s\in(0,\infty)$, a $p\in(0,Q/s)$, a $q\in(0,\infty]$, and a $c_S\in(0,\infty)$
satisfying $s\preceq_q{\rm ind}(\Omega,\rho)$
such that
\begin{equation*}
\Vert u\Vert_{L^{p^\ast}(\Omega)}\leq
c_S\Vert u\Vert_{{M}^s_{p,q}(\Omega,\rho,\mu)},
\end{equation*}
whenever $u\in{M}^s_{p,q}(\Omega,\rho,\mu)$. Here, $p^*:=Qp/(Q-sp)$.

\item There exist an $s\in(0,\infty)$, a $p\in(0,Q/s)$, a $q\in(0,\infty]$,
and a $c_P\in(0,\infty)$ satisfying $s\preceq_q{\rm ind}(\Omega,\rho)$ such that
\begin{equation*}
\inf_{\gamma\in\mathbb{R}}\Vert u-\gamma\Vert_{L^{p^\ast}(\Omega)}\leq
c_P\Vert u\Vert_{{M}^s_{p,q}(\Omega,\rho,\mu)},
\end{equation*}
whenever $u\in{M}^s_{p,q}(\Omega,\rho,\mu)$. Here, $p^*:=Qp/(Q-sp)$.

\item There exist constants $c_1,\,c_2,\,\omega\in(0,\infty)$,
$q\in(0,\infty]$, and $s\in(0,\infty)$ satisfying $s\preceq_q{\rm ind}(\Omega,\rho)$ such that
\begin{equation*}
\int_{B_0\cap\Omega} {\rm exp}\left(c_1\frac{|u-u_{B_0\cap\Omega}|}{\Vert u\Vert_{{M}^s_{Q/s,q}(\Omega,\rho,\mu)}}\right)^\omega\,d\mu\leq c_2R_0^Q
\end{equation*}
for any $\rho$-ball $B_0$ centered in $\Omega$ with finite radius $R_0\in(0,{\rm diam}_\rho(X)]$,
and  any nonzero function $u\in {M}^{s}_{p,q}(\Omega,\rho,\mu)$.

\item There exist an $s\in(0,\infty)$, a $p\in(Q/s,\infty)$, a $q\in(0,\infty]$, and a
$c_H\in(0,\infty)$ satisfying $s\preceq_q{\rm ind}(\Omega,\rho)$ such that  every function $u\in{M}^s_{p,q}(\Omega,\rho,\mu)$
has a H\"older continuous representative of order $s-Q/p$ on $\Omega$, denoted by $u$ again, satisfying
\begin{equation*}
|u(x)-u(y)|\leq c_H\,[\rho(x,y)]^{s-Q/p}\,\Vert u\Vert_{{M}^s_{p,q}(\Omega,\rho,\mu)},
\quad\forall\ x,\,y\in \Omega.
\end{equation*}
\end{enumerate}
\end{theorem}

\begin{remark}
Theorem~\ref{measdens-ext-sob} asserts that in particular, if just one of the statements in (b) - (k) holds for some $p,q,s$ (in their respective ranges), then all of the statements (b) - (k) hold for \textit{all} $p,q,s$ (in their respective ranges).
\end{remark}

\begin{remark}
In the context of Theorem~\ref{measdens-ext-sob}, if we
include the additional demand that $q\leq p$ in statements (d)-(k),
and require that $s<{\rm ind}\,(X,\rho)$ in statements (b) and (c), then
all of the statements in Theorem~\ref{measdens-ext-sob} remain equivalent with $\dot{M}^s_{p,q}$
and ${M}^s_{p,q}$ replaced by $\dot{N}^s_{p,q}$ and ${N}^s_{p,q}$, respectively.
The case for the full range of $q\in(0,\infty]$ is addressed in Theorem~\ref{measdens-ext-sob-besov} below. Note that an upper bound on
the exponent $q$ for the Besov spaces is to be expected; see \cite[Remark~4.17]{AYY21}.
\end{remark}

\begin{remark}
Recall that we assume that $\Omega$ satisfies
$\mu(B_\rho(x,r)\cap\Omega)>0$ for any
$x\in \Omega$ and   $r\in(0,1]$,
and so $(\Omega,\rho,\mu)$ in Theorem~\ref{measdens-ext-sob}
is a well-defined locally uniformly perfect quasi-metric-measure space.
 This is only used in proving that each of {(b)-(k)} $\Longrightarrow$ {(a)}.
\end{remark}

\begin{proof}[Proof of Theorem \ref{measdens-ext-sob}]
We prove the present theorem  in seven steps.
\medskip

\noindent\textbf{Step 1:\ Proof that  {(a)} implies  {(d)-(g)}:}
Assume that the measure density condition in \eqref{measdensity} holds true. Our plan
is to use the embeddings in Theorem~\ref{DOUBembedding} with $r_\ast:=C_\rho$
for the induced quasi-metric-measure space $(\Omega,\rho,\mu)$. To justify our
use of Theorem~\ref{DOUBembedding} with this choice of $r_\ast$,
we need to show that $\mu$ is $Q$-doubling on $\Omega$ up to scale $r_\ast$. That is, we need to show that
\begin{equation}
\label{gcq5-xt}
\lf(\frac{r}{R}\r)^{Q}\lesssim\frac{\mu(B_\rho(x,r)\cap\Omega)}{\mu(B_\rho(y,R)\cap\Omega)}
\end{equation}
for any balls $B_\rho(x,r)\cap\Omega\subset B_\rho(y,R)\cap\Omega$
with $x,\,y\in\Omega$ and $0<r\leq R\leq C_\rho$.
Indeed, by $C_\rho\geq1$,
we have $C_\rho^{-1}r\leq r\leq C_\rho R$ and $C_\rho^{-1}r\leq C_\rho^{-1}R\leq1$.
Moreover, since $x\in B_\rho(y,R)$, we dduce that $B_\rho(x,C_\rho^{-1}r)\subset B_\rho(y,C_\rho R)$.
Then the $Q$-doubling property \eqref{Doub-2-eh} and the
measure density condition \eqref{measdensity} imply
\begin{equation*}
\frac{\mu\big(B_\rho(x,r)\cap\Omega\big)}{\mu\big(B_\rho(y,R)\cap\Omega\big)}
\geq\frac{\mu\big(B_\rho(x,C_\rho^{-1}r)\cap\Omega\big)}{\mu\big(B_\rho(y,R)\big)}\geq
\frac{C_\mu^{-1}\,\mu\big(B_\rho(x,C_\rho^{-1}r)\big)}{\mu\big(B_\rho(y,C_\rho R)\big)}\geq\kappa C_\mu^{-1}\left(\frac{C_\rho^{-1}r}{C_\rho R}\right)^{Q},
\end{equation*}
and so \eqref{gcq5-xt} holds true. With \eqref{gcq5-xt}
in hand, applying Theorem~\ref{DOUBembedding} with $r_\ast:=C_\rho$
implies that the embeddings in \eqref{eq18-DOUB}, \eqref{eq19-DOUB}, \eqref{eq20-DOUB}, and \eqref{eq30-DOUB} hold true
for any domain ball $B_\rho(x_0,R_0)\cap\Omega$, where $x_0\in\Omega$ and $R_0\in(0,1]$.

Moving on, fix exponents $s,\,p\in(0,\infty)$ and
$q\in(0,\infty]$. Consider a $\rho$-ball $B_0:=B_\rho(x_0,R_0)$,
where $x_0\in\Omega$ and $R_0\in(0,1]$, and
suppose that $u\in \dot{M}^s_{p,q}(\Omega,\rho,\mu)$.
Note that the pointwise restriction of $u$ to $\sigma B_0\cap\Omega$, also denoted by $u$,
belongs to $\dot{M}^s_{p,q}(\sigma B_0\cap\Omega,\rho,\mu)$. As such,   Theorem~\ref{DOUBembedding}(a)
implies that, if $p\in(0,Q/s)$, then there exists a positive constant $C$ such that
\begin{align}
\label{eq18-DOUB-last}
\Vert u\Vert_{L^{p^\ast}(B_0\cap\Omega)}&\leq
\frac{C}{[\mu(\sigma B_0\cap\Omega)]^{s/Q}}\left[
R_0^s\Vert u\Vert_{\dot{M}^s_{p,q}(\sigma B_0\cap\Omega)}
+\Vert u\Vert_{L^{p}(\sigma B_0\cap\Omega)}\right]
\noz\\
&\leq\frac{C  C_\mu^{s/Q}}{[\mu(B_0)]^{s/Q}}\left[
R_0^s\Vert u\Vert_{\dot{M}^s_{p,q}(\Omega)}
+\Vert u\Vert_{L^{p}(\Omega)}\right],
\end{align}
where, in obtaining the second inequality in \eqref{eq18-DOUB-last}, we have used the measure density condition \eqref{measdensity} and the fact that $\sigma\geq1$ to write
$$
\mu(\sigma B_0\cap\Omega)\geq\mu(B_0\cap\Omega)\geq C_\mu^{-1}\mu(B_0).
$$
This is the desired inequality in \eqref{ptc} and hence, (d) holds  true.
A similar line of reasoning will  show that (a) also implies each of {(e)-(g)}.
This finishes Step 1.

\medskip

\noindent \textbf{Step 2:\ Proof that each of {(d)-(g)} implies {(a)}:} As a preamble,
we make a few comments to clarify for the reader how we will employ Lemma~\ref{HolderBump}
for the space $(\Omega,\rho,\mu)$ in the sequel. First, observe that the assumptions on
$\Omega$ ensure that $(\Omega,\rho)$ is a well-defined locally uniformly perfect
quasi-metric measure space. Moreover, if $s\preceq_q{\rm ind}(\Omega,\rho)$,
then we can always choose a quasi-metric $\varrho$ on $\Omega$
such that $\varrho\approx\rho\lfloor_\Omega$  and $s\leq(\log_{2}C_\varrho)^{-1}$,
where $C_\varrho\in[1,\infty)$ is as in \eqref{C-RHO.111}, and the value  $s=(\log_{2}C_\varrho)^{-1}$
can only occur when $q=\infty$. Throughout this step,
we let $\varrho$ denote such a quasi-metric whenever
$s\preceq_q{\rm ind}(\Omega,\rho)$, and let $\varrho_{\#}$
be the regularized quasi-metric given by Theorem~\ref{DST1}.
Also, $\lambda\in(0,1)$ will stand for the local uniformly perfect
constant as in \eqref{U-perf} for the quasi-metric space $(\Omega,\rho)$.
Recall that there is no loss in generality by assuming that $\lambda<(C_{\rho\lfloor_\Omega}\widetilde{C}_{\rho\lfloor_\Omega})^{-2}$.

We now prove that (d)  implies (a). Fix a ball $B:=B_\rho(x,r)$
with $x\in\Omega$ and $r\in(0,1]$. By assumption, there exist a $p\in(0,Q/s)$,
a $q\in(0,\infty]$,  an $s\in(0,\infty)$ satisfying $s\preceq_q{\rm ind}(\Omega,\rho)$, and a $C_S\in(0,\infty)$
such that \eqref{ptc} holds true whenever $u\in\dot{M}^s_{p,q}(\Omega,\rho,\mu)$.
Note that, if $r>{\rm diam}_\rho(\Omega)$, then $B=\Omega$
and \eqref{measdensity} is trivially satisfied with any $C_\mu\in[1,\infty)$.
As such, we assume that $r\leq{\rm diam}_\rho(\Omega)$. Let
$\{r_j\}_{j\in\mathbb{N}}$ and $\{u_j\}_{j\in\mathbb{N}}$ be as in
Lemma~\ref{HolderBump} associated to the domain ball $B\cap\Omega$ in the induced quasi-metric measure space $(\Omega,\rho,\mu)$. Then,
for any $j\in\mathbb{N}$, we have $u_j\in M^s_{p,q}(\Omega,\rho,\mu)$ which,
in turn, implies that  $u_j$ satisfies \eqref{ptc} with $B_0$ replaced by $B$.
Moreover, $B_{\varrho_{\#}}(x,r_j)\cap\Omega\subset B\cap\Omega$.
As such, it follows from the properties listed in {(b)-(e)} of Lemma~\ref{HolderBump}
that, for any $j\in\mathbb{N}$,
\begin{equation}\label{xu-12-X}
\|u_j\|_{\dot{M}^s_{p,q}(\Omega)}\lesssim r^{-s}\,2^j\lf[\mu(B_{\varrho_\#}(x,r_j)\cap\Omega)\r]^{1/p}
\end{equation}
and
\begin{equation}\label{xu-13-X}
\Vert u_j\Vert_{L^{p}(\Omega)}\lesssim \lf[\mu(B_{\varrho_\#}(x,r_j)\cap\Omega)\r]^{1/p}.
\end{equation}
Therefore, we have
\begin{align}
\label{djq-476}
r^{s}\|u_j\|_{\dot{M}^s_{p,q}(\Omega)}+\Vert u_j\Vert_{L^{p}(\Omega)}
&\lesssim (2^j+1)\,\lf[\mu(B_{\varrho_\#}(x,r_j)\cap\Omega)\r]^{1/p}\noz
\\
&\lesssim 2^{j+1}\,\lf[\mu(B_{\varrho_\#}(x,r_j)\cap\Omega)\r]^{1/p}.
\end{align}
Moreover, since $u_j\equiv1$ on the set $B_{\varrho_\#}(x,r_{j+1})\cap\Omega$, one deduces that
\begin{equation}\label{xu-14-X}
\Vert u_j\Vert_{L^{p^\ast}(B\cap\Omega)}\ge \Vert u_j\Vert_{L^{p^\ast}(B_{\varrho_\#}(x,r_{j+1})\cap\Omega)}=\lf[\mu(B_{\varrho_{\#}}(x,r_{j+1})\cap\Omega)\r]^{1/p^*}.
\end{equation}
Combining \eqref{ptc}, \eqref{djq-476}, and \eqref{xu-14-X} gives
\begin{equation}
\label{xu-15-X}
\lf[\mu(B_{\varrho_{\#}}(x,r_{j+1})\cap\Omega)\r]^{1/p^*}
\leq C
[\mu(B)]^{-s/Q} 2^{j}\lf[\mu(B_{\varrho_{\#}}(x,r_j)\cap\Omega)\r]^{1/p},
\quad\forall\,j\in\mathbb{N},
\end{equation}
where $C\in(0,\infty)$ is independent of both $x$ and $r$. Since $0<\delta r<r_j<r<\infty$ for
any $j\in\mathbb{N}$, where $\delta\in(0,1)$ is as in Lemma~\ref{HolderBump},
we  invoke Lemma~\ref{iteration} with the quasi-metric $\varrho_\#$, and
$$
p:=p,\quad
t:=p^*,\quad
\text{and}
\quad\theta:=C\,[\mu(B)]^{-s/Q},
$$
to conclude that
\begin{equation}
\label{xu-16-X}
\mu(B\cap\Omega)\geq\mu(B_{\varrho_{\#}}(x,r_1)\cap\Omega)\gtrsim 2^{-\frac{Q^2}{s^2p}}
\left[\mu(B)^{-s/Q}\right]^{-Q/s},
\end{equation}
which implies the  measure density condition \eqref{measdensity}. This finishes the proof of
the statement that {(d)} implies {(a)}.

We next prove that {(e)} implies {(a)}.
Fix a ball $B:=B_\rho(x,r)$ with $x\in\Omega$ and $r\in(0,1]$
and recall that we can assume that $r\leq{\rm diam}_\rho(\Omega)$.
By the statement in  {(e)},   there exist an $s\in(0,\infty)$
satisfying $s\preceq_q{\rm ind}(\Omega,\rho)$, a $p\in(0,Q/s)$, a $q\in(0,\infty]$,
and a $C_P\in(0,\infty)$ such that \eqref{ptc-poin} holds true
whenever $u\in{M}^s_{p,q}(\Omega,\rho,\mu)$, which further leads to the
following  weaker version of
\eqref{ptc-poin} (stated here with $B_0=B$ for the arbitrarily fixed ball $B$)
\begin{equation}
\label{ptc-poin-weak}
\inf_{\gamma\in\mathbb{R}}\Vert u-\gamma\Vert_{L^{p^\ast}(B\cap\Omega)}\leq
\frac{C_Pr^s}{[\mu(B)]^{s/Q}}\,\Vert{u}\Vert_{M^s_{p,q}(\Omega,\rho,\mu)}.
\end{equation}
We next show that \eqref{ptc-poin-weak} implies \eqref{measdensity}.
In light  of Lemma~\ref{en2-4}, we may
assume that $$r\leq C_{\rho\lfloor_\Omega}\varphi^x_{\Omega,\rho}(r)/\lambda^2.$$
Then, by Lemma~\ref{HolderBump}, we find that there exist $\{r_j\}_{j\in\mathbb{N}}\subset(0,\fz)$ and
a collection  $\{u_j\}_{j\in\mathbb{N}}$ of functions  that are associated to
$B\cap\Omega$ and have the properties {(a)-(f)} listed in Lemma~\ref{HolderBump}, with the choice $c_0:=C_{\rho\lfloor_\Omega}/\lambda^2\in(1,\infty)$.
In particular,  for any $j\in\mathbb{N}$,  $u_j\in M^s_{p,q}(\Omega,\rho,\mu)$.
Hence, for any $j\in\nn$,  $u_j$ satisfies \eqref{ptc-poin-weak} and the estimates in
both
\eqref{xu-12-X} and \eqref{xu-13-X}. In particular, we obtain
\begin{equation}
\label{Gkw-924-2-p}
\begin{split}
\|u_j\|_{{M}^s_{p,q}(\Omega)}
&\lesssim (r^{-s}2^j+1)\,\lf[\mu(B_{\varrho_\#}(x,r_j)\cap\Omega)\r]^{1/p}
\lesssim r^{-s}2^{j}\,\lf[\mu(B_{\varrho_\#}(x,r_j)\cap\Omega)\r]^{1/p},
\end{split}
\end{equation}
where the last inequality follows from the fact that $1\leq r^{-s}2^{j}$. On the other hand,
Lemma~\ref{HolderBump}(f) gives
\begin{align}
\label{Gkw-925-2-p}
\inf_{\gamma\in\mathbb{R}}\Vert u_j-\gamma\Vert_{L^{p^\ast}(B\cap\Omega)}
\geq\inf_{\gamma\in\mathbb{R}}\frac{1}{2}\lf[\mu\lf(E_j^\gamma\r)\r]^{1/p^*}
\geq\frac{1}{2}\lf[\mu\lf(B_{\varrho_{\#}}(x,r_{j+1})\cap\Omega\r)\r]^{1/p^*},\quad \forall\,j\in\mathbb{N}.
\end{align}
In concert, \eqref{ptc-poin-weak}, \eqref{Gkw-925-2-p}, and \eqref{Gkw-924-2-p} give
\begin{equation*}
\lf[\mu\lf(B_{\varrho_{\#}}(x,r_{j+1})\cap\Omega\r)\r]^{1/p^*}
\lesssim
 \frac{1}{[\mu(B)]^{s/Q}}  2^{j}\,\lf[\mu(B_{\varrho_\#}(x,r_j)\cap\Omega)\r]^{1/p},
\quad\forall\,j\in\mathbb{N}.
\end{equation*}
At this stage, we have arrived at the inequality  \eqref{xu-15-X}. Therefore, arguing as in the proofs of
both
\eqref{xu-15-X} and \eqref{xu-16-X}  leads to \eqref{measdensity}. This finishes
the proof of the statement that  {(e)} implies {(a)}.

Next, we turn our attention to showing that
{(f)} implies {(a)}  by proving that the measure density
condition \eqref{measdensity} can be deduced from the
following weaker version of {(f)}: There exist constants $C_1,\,C_2,\,\omega\in(0,\infty)$,
$q\in(0,\infty]$, and $s\in(0,\infty)$ satisfying $s\preceq_q{\rm ind}(\Omega,\rho)$ such that,
for any $\rho$-ball $B:=B_\rho(x,r)$ with $x\in\Omega$ and $r\in(0,1]$, one has
\begin{equation}
\label{ptc-trud-weak}
\inf_{\gamma\in\mathbb{R}}\int_{B\cap\Omega} {\rm exp}\left(C_1\frac{[\mu(B)]^{s/Q}|u-\gamma|}{r^s\Vert{u}\Vert_{M^s_{Q/s,q}(\Omega,\rho,\mu)}}\right)^\omega\,d\mu\leq C_2\mu(B),
\end{equation}
whenever $u\in M^s_{Q/s,q}(\Omega,\rho,\mu)$ is nonconstant. To this end, fix a ball $B:=B_\rho(x,r)$ with $x\in\Omega$ and $r\in(0,1]$. As before, it suffices to consider the case when  $r\leq{\rm diam}_\rho(\Omega)$ and $r\leq C_{\rho\lfloor_\Omega}\varphi^x_{\Omega,\rho}(r)/\lambda^2$.
Now, let $\{r_j\}_{j\in\mathbb{N}}$ and $\{u_j\}_{j\in\mathbb{N}}$ be as in Lemma~\ref{HolderBump}
associated to the domain ball $B\cap\Omega$.
Then, by Lemma~\ref{HolderBump}(e), for any $j\in\mathbb{N}$,
$u_j\in M^s_{p,q}(\Omega,\rho,\mu)$ is nonconstant in $\Omega$
and hence satisfies \eqref{ptc-trud-weak} with
$B$ and the estimate in \eqref{Gkw-924-2-p}. Therefore, for any $\gamma\in\mathbb{R}$ and any $j\in\mathbb{N}$ we have
$$
\frac{[\mu(B)]^{s/Q}|u_j-\gamma|}{r^s\Vert{u}\Vert_{M^s_{p,q}(\Omega,\rho,\mu)}}
\geq\frac{C[\mu(B)]^{s/Q}|u_j-\gamma|}{2^{j}
\,[\mu(B_{\varrho_{\#}}(x,r_{j})\cap\Omega)]^{s/Q}},
$$
for some positive constant $C$ independent of $x$, $r$, $\gamma$, and $j$.
Combining this with   Lemma~\ref{HolderBump}(f) and \eqref{ptc-trud-weak}, we conclude that
\begin{align}
\label{Iue.5-2-X-52-p}
&\mu(B_{\varrho_{\#}}(x,r_{j+1})\cap\Omega)\,
{\rm exp}\left(\frac{C[\mu(B)]^{s/Q}}{2^{j+1}
\,[\mu(B_{\varrho_{\#}}(x,r_{j})\cap\Omega)]^{s/Q}}\right)^\omega\nonumber\\
&
\quad\leq\inf_{\gamma\in\mathbb{R}}\int_{E_j^\gamma} {\rm exp}\left(C_1\frac{[\mu(B)]^{s/Q}|u_j-\gamma|}{r^s\Vert{u_j}\Vert_{M^s_{Q/s,q}(\Omega,\rho,\mu)}}\right)^\omega\,d\mu\leq C_2\mu(B).
\end{align}
By increasing the positive constant $C_2$,
we may assume that $C_2>1$. As such, using the elementary
estimate $\log(z)\leq 2Q(s\omega)^{-1}\,z^{s\omega/(2Q)}$ for any $z\in(0,\infty)$,
a rewriting of \eqref{Iue.5-2-X-52-p} implies that
\begin{align*}
\frac{C[\mu(B)]^{s/Q}}{2^{j+1}\,[\mu(B_{\varrho_{\#}}(x,r_{j})\cap\Omega)]^{s/Q}}
&\leq\lf[\log\lf(C_2\frac{\mu(B)}{\mu(B_{\varrho_{\#}}(x,r_{j+1})\cap\Omega)}\r)\r]^{1/\omega}\\
&\leq \lf[2Q(s\omega)^{-1}\r]^{1/\omega}\,C_2^{s/(2 Q)}\left[\frac{\mu(B)}{\mu(B_{\varrho_{\#}}(x,r_{j+1})\cap\Omega)}\right]^{s/(2Q)}.
\end{align*}
Therefore,
\begin{equation*}
\lf[\mu(B_{\varrho_{\#}}(x,r_{j+1})\cap\Omega)\r]^{s/(2Q)}\leq
 \frac{C}{[\mu(B_0)]^{s/(2Q)}} \,2^{j}\,[\mu(B_{\varrho_{\#}}(x,r_{j})\cap\Omega)]^{s/Q}.
\end{equation*}
Now, applying Lemma~\ref{iteration} with the quasi-metric $\varrho_\#$,
$$
p:=Q/s,\quad
t:=2Q/s,\quad
\text{and}
\quad
\theta:=\frac{C}{[\mu(B)]^{s/(2Q)}},
$$
we obtain
\begin{equation}
\label{ptc-trud-weak-last}
\mu(B\cap\Omega)\geq\mu(B_{\varrho_{\#}}(x,r_1)\cap\Omega)\gtrsim
\left\{\frac{C}{[\mu(B)]^{s/(2Q)}}\right\}^{\frac{-2Q}{s}}\,
2^{\frac{-4Q}{s^2}},
\end{equation}
which proves \eqref{measdensity}.
The proof of the statement that {(e)} implies {(a)} is now complete.

As concerns the implication {(g)}\,$\Longrightarrow$\,{(a)},
note that, by  {(g)}, we find that
there exist an $s\in(0,\infty)$ satisfying $s\preceq_q{\rm ind}(\Omega,\rho)$, a
$p\in(Q/s,\infty)$, a $q\in(0,\infty]$, and a $C_H\in(0,\infty)$ such that
the following weaker version of \eqref{ptc-hold} holds true,
whenever $B:=B_\rho(x,r)$ with $x\in\Omega$ and $r\in(0,1]$
is a $\rho$-ball and $u\in M^s_{p,q}(\Omega,\rho,\mu)$,
\begin{eqnarray}
\label{ptc-hold-weak}
|u(x)-u(y)|\leq C_H\,[\rho(x,y)]^{s-Q/p}\frac{r^{Q/p}}{[\mu(B)]^{1/p}}\,\Vert{u}\Vert_{M^s_{p,q}(\Omega,\rho,\mu)},
\quad\forall\,x,\,y\in B\cap \Omega.
\end{eqnarray}
We  now show that  {(a)} follows from \eqref{ptc-hold-weak}.
To this end, fix an $x\in\Omega$ and an $r\in(0,1]$ again.
If $B_\rho(x,r)=\Omega$, then \eqref{measdensity} easily follows
with any choice of $C_\mu\in[1,\infty)$. Thus, in what follows, we
assume that $\Omega\setminus B_\rho(x,r)\neq\emptyset$.
Granted this, using the local uniform perfectness of $(\Omega,\rho)$,
we can choose a point $x_0\in [B_{\rho}(x,r)\cap\Omega]\setminus B_{\rho}(x,\lambda r)$.
By Lemma~\ref{HolderBump}, we know that there exist a sequence  $\{r_j\}_{j\in\mathbb{N}}$ of
radii and a family $\{u_j\}_{j\in\mathbb{N}}$ of functions that are associated to
the domain ball $B_\rho(x,\lambda r)\cap\Omega$.
Consider the function $u_1$ from this collection. According to Lemma~\ref{HolderBump}, we have
$u_1\in M^s_{p,q}(\Omega,\rho,\mu)$, $u_1\equiv 1$ on $B_{\varrho_\#}(x,r_2)\cap\Omega$
and $u_1\equiv 0$ on $\Omega\setminus B_{\varrho_\#}(x,r_1)$. Moreover,
by appealing to \eqref{Gkw-924-2-p}, we obtain the following estimate
\begin{equation}
\label{fna-487}
\|u_1\|_{{M}^s_{p,q}(\Omega)}\lesssim r^{-s}\,\lf[\mu(B_{\varrho_\#}(x,r_1)\cap\Omega)\r]^{1/p}
\lesssim r^{-s}\,\lf[\mu(B_{\rho}(x,r)\cap\Omega)\r]^{1/p}.
\end{equation}
Note that we have simply used the inclusion $B_{\varrho_\#}(x,r_1)\cap\Omega\subset B_{\rho}(x,\lambda r)\cap\Omega$
[see Lemma~\ref{HolderBump}(a)] and the fact that $\lambda<1$ in obtaining the last inequality in \eqref{fna-487}.
Now, the choice of the point $x_0$ ensures that $x_0\not\in B_{\varrho_\#}(x,r_1)$ and so $u_1(x_0)=0$.
Thus, it follows from \eqref{ptc-hold-weak} (used with $u=u_1$) and \eqref{fna-487} that 	
\begin{align*}
1&=|u_1(x)-u_1(x_0)|
\leq C_H [\rho(x,x_0)]^{s-Q/p}\frac{r^{Q/p}}{[\mu(B_\rho(x, r))]^{1/p}}\,\|u_1\|_{{M}^s_{p,q}(\Omega)}\\
&\lesssim r^{s-Q/p}\frac{r^{Q/p-s}}{[\mu(B_\rho(x, r))]^{1/p}}\,[\mu(B_{\rho}(x,r)\cap\Omega)]^{1/p}\\
&\lesssim\frac{[\mu(B_\rho(x, r)\cap\Omega)]^{1/p}}{[\mu(B_\rho(x,r))]^{1/p}}.
\end{align*}
The desired estimate in \eqref{measdensity} now follows.
This finishes the proof that {(g)}\,$\Longrightarrow$\,{(a)} and concludes Step~2.	

\medskip

\noindent\textbf{Step 3:\ Proof that {(a)} implies {(c)}:} This implication follows immediately from
Theorem~\ref{measext}.

\medskip

\noindent\textbf{Step 4:\ Proof that {(c)} implies {(b)}:} This trivially holds true.

\medskip

\noindent \textbf{Step 5:\ Proof that {(b)} implies {(a)}:} Assume that the statement in  {(b)} holds
true. That is, $\Omega$ is an $M^s_{p,q}$-extension domain in the sense of \eqref{ext1}
for some exponents $p\in(0,\infty)$, $q\in(0,\infty]$, and $s\in(0,\infty)$
satisfying $s\preceq_q{\rm ind}\,(X,\rho)$. Since $C_{\rho\lfloor_\Omega}\leq C_{\rho}$ [due to their definitions in \eqref{C-RHO.111}], it follows from  \eqref{index}
that ${\rm ind}\,(X,\rho)\leq{\rm ind}\,(\Omega,\rho)$.
Hence, we have $s\preceq_q{\rm ind}\,(\Omega,\rho)$.
To ease the notation in what follows, we will omit the notation `$\lfloor_\Omega$' on the quasi-metric. In what follows, we let $\sigma:=C_\rho$.
We now proceed by considering three cases depending on the size of $p$.

\medskip

\noindent\texttt{CASE 1:\ $p\in(0,Q/s)$.}
In this case, we first show that inequality \eqref{ptc-poin-weak} holds true
for any ball $B:=B_\rho(x,r)$ with $x\in\Omega$ and $r\in(0,1]$. Fix such a ball $B$ and a function $u\in M^s_{p,q}(\Omega,\rho,\mu)$. By \eqref{ext1},
we know that there exists a $\widetilde{u}\in M^s_{p,q}(X,\rho,\mu)$ such that $u=\widetilde{u}|_{\Omega}$ and $\|\widetilde{u}\|_{M^s_{p,q}(X,\rho,\mu)}\ls\|u\|_{M^s_{p,q}(\Omega,\rho,\mu)}$. Since $\mu$ is $Q$-doubling on $X$
and $\widetilde{u}\in M^s_{p,q}(\sigma B)$, \eqref{eq19-DOUB} in Theorem~\ref{DOUBembedding}
(applied with $B_0$ therein replaced by $B$) gives
\begin{align*}
\|\widetilde{u}\|_{M^s_{p,q}(X,\rho,\mu)}
&\geq \|\widetilde{u}\|_{\dot{M}^s_{p,q}(\sigma B,\rho,\mu)}
\gtrsim\frac{[\mu(\sigma B)]^{s/Q}}{r^s}\inf_{\gamma\in\mathbb{R}}\Vert
\widetilde{u}-\gamma\Vert_{L^{p^\ast}(B)}\\
&\gtrsim\frac{[\mu(B)]^{s/Q}}{r^s}\inf_{\gamma\in\mathbb{R}}\Vert u-\gamma\Vert_{L^{p^\ast}(B\cap\Omega)},
\end{align*}	
where $p^\ast:=Qp/(Q-sp)$. As such, we have
\begin{align*}
\inf_{\gamma\in\mathbb{R}}\Vert u-\gamma\Vert_{L^{p^\ast}(B\cap\Omega)}
&\lesssim\frac{r^s}{[\mu(B)]^{s/Q}}\|\widetilde{u}\|_{M^s_{p,q}(X,\rho,\mu)}
\lesssim\frac{r^s}{[\mu(B)]^{s/Q}}\|u\|_{M^s_{p,q}(\Omega,\rho,\mu)}.
\end{align*}
Hence, \eqref{ptc-poin-weak} holds true whenever $u\in M^s_{p,q}(\Omega,\rho,\mu)$ which, in turn, implies that the measure density condition \eqref{measdensity}, as we have proved in Step 2.
\medskip

\noindent\texttt{CASE 2:\ $p=Q/s$.} In this case,
we claim that \eqref{ptc-trud-weak} holds true.
To see this, let  $B:=B_\rho(x,r)$ with $x\in\Omega$ and $r\in(0,1]$,
and fix a nonconstant function $u\in M^s_{p,q}(\Omega,\rho,\mu)$.
By \eqref{ext1}, we know that there exists a $\widetilde{u}\in M^s_{p,q}(X,\rho,\mu)$ such
that $u=\widetilde{u}|_{\Omega}$ and $\|\widetilde{u}\|_{M^s_{p,q}(X,\rho,\mu)}\le C
\|u\|_{M^s_{p,q}(\Omega,\rho,\mu)}$
for some positive constant $C$ independent of $u$. If $\|\widetilde{u}\|_{\dot{M}^s_{p,q}(\sigma B)}=0$, then, by
Proposition~\ref{constant}, we know that $\widetilde{u}$ is constant $\mu$-almost everywhere in $\sigma B$.
Hence, $u=\widetilde{u}$ is constant $\mu$-almost everywhere in
$B\cap\Omega$ and \eqref{ptc-trud-weak} is trivial.
If $\|\widetilde{u}\|_{\dot{M}^s_{p,q}(\sigma B)}>0$,
since $\mu$ is $Q$-doubling on $X$, applying Theorem~\ref{DOUBembedding}(b) with $B_0=B$,
we find that
\begin{align*}
\mu(B)&\gtrsim\int_{B} {\rm exp}\left(C_1\frac{[\mu(\sigma B)]^{s/Q}|\widetilde{u}-\widetilde{u}_{B}|}{r^s\|\widetilde{u}\|_{\dot{M}^s_{p,q}(\sigma B,\rho,\mu)}}\right)^\omega\,d\mu
\\
&\gtrsim\inf_{\gamma\in\mathbb{R}}\int_{B\cap\Omega} {\rm exp}\left(C_1\frac{[\mu(B)]^{s/Q}|u-\gamma|}{r^s\|\widetilde{u}\|_{{M}^s_{p,q}(X,\rho,\mu)}}\right)^\omega\,d\mu
\\
&\gtrsim\inf_{\gamma\in\mathbb{R}}\int_{B\cap\Omega} {\rm exp}\left(C_1\frac{[\mu(B)]^{s/Q}|u-\gamma|}{Cr^s\|u\|_{M^s_{p,q}(\Omega,\rho,\mu)}}\right)^\omega\,d\mu.
\end{align*}
Thus, \eqref{ptc-trud-weak} holds true, and the measure density condition \eqref{measdensity}
now follows from the argument presented in \eqref{ptc-trud-weak}-\eqref{ptc-trud-weak-last}.
\bigskip

\noindent\texttt{CASE 3:\ $p\in(0,Q/s)$.} In this case, fix a ball $B:=B_\rho(x,r)$ with $x\in\Omega$ and $r\in(0,1]$, and let $u\in M^s_{p,q}(\Omega,\rho,\mu)$.
Again, by \eqref{ext1}, we know that
there exists a $\widetilde{u}\in M^s_{p,q}(X,\rho,\mu)$ such that
$u=\widetilde{u}|_{\Omega}$ and $\|\widetilde{u}\|_{M^s_{p,q}(X,\rho,\mu)}
\ls\|u\|_{M^s_{p,q}(\Omega,\rho,\mu)}$. Since $\mu$ is $Q$-doubling on $X$, applying
Theorem~\ref{DOUBembedding}(c) with $B_0=B$, we conclude that, if $x,\,y\in B\cap\Omega$, then
\begin{align*}
|u(x)-u(y)|
&=|\widetilde{u}(x)-\widetilde{u}(y)|
\\
&\lesssim[\rho(x,y)]^{s-Q/p}\frac{r^{Q/p}}{[\mu(\sigma B)]^{1/p}}\,\|\widetilde{u}\|_{\dot{M}^s_{p,q}(\sigma B,\rho,\mu)}
\\
&\lesssim[\rho(x,y)]^{s-Q/p}\frac{r^{Q/p}}{[\mu(B)]^{1/p}}\,\|u\|_{M^s_{p,q}(\Omega,\rho,\mu)}.
\end{align*}
This implies that \eqref{ptc-hold-weak} holds true, which further implies
the measure density condition \eqref{measdensity}, as proved in Step 2.
This finishes the proof of the statement that  {(b)}\,$\Longrightarrow$\,{(a)}, and hence  of Step~5.

\medskip

There remains to show that {(a)-(k)} are equivalent under
the assumption that $\mu$ is $Q$-Ahlfors regular on $X$. Since a
$Q$-Ahlfors regular measure is also a $Q$-doubling measure (with the same value of $Q$),
we immediately have that {(a)-(g)} are equivalent in light of Steps~1-5.
As such, there are only two more steps left in the proof of the present theorem.

\medskip

\noindent \textbf{Step 6:\ Proof that each of {(h)-(k)} implies {(a)}:}
Note that the $Q$-Ahlfors regularity condition ensures that $\mu(B(x,r))\approx r^Q$ for
any $x\in X$ and $r\in(0,1]$. With this observation in hand, it is straightforward to check
that the statements  (h), {(i)}, {(j)}, and {(k)} of Theorem~\ref{measdens-ext-sob} imply the statements (d), (e), (f), and (g), respectively. Hence, {(a)} follows as a consequence of {(i)}-{(k)},
given what we have already established in Step~2. This finishes the proof of  Step~6.

\medskip

\noindent \textbf{Step 7:\ Proof that  {(b)} implies each of {(h)-(k)}:} This step   proceeds along a similar line of reasoning as the one in Step~5 where, in place of the local embeddings for $Q$-doubling measures in Theorem~\ref{DOUBembedding}, we use the global embeddings for $Q$-Ahlfors regular measures in Theorem~\ref{GlobalEmbeddCor}. This finishes the proof of Step~7 and, in turn, the proof of Theorem~\ref{measdens-ext-sob}.
\end{proof}

The following theorem is a version of Theorem~\ref{measdens-ext-sob} adapted to $N^s_{p,q}$ spaces which, to our knowledge, is a brand new result even in the Euclidean setting.

\begin{theorem}
\label{measdens-ext-sob-besov}
Let $(X,\rho,\mu)$ be a quasi-metric measure space
where $\mu$ is a Borel regular $Q$-doubling measure for
some $Q\in(0,\infty)$, and suppose that $\Omega\subset X$
is a $\mu$-measurable locally uniformly perfect set. Then
the following statements are equivalent.
\begin{enumerate}[label={\rm(\alph*)}]
\item The measure $\mu$ satisfies the following  measure density condition on $\Omega$:\ there exists some constant $C_\mu\in(0,\infty)$ with the property that
\begin{equation*}
\mu(B_\rho(x,r))\leq C_\mu\,\mu(B_\rho(x,r)\cap\Omega)\,\,\,\mbox{for any }\,x\in \Omega\mbox{ and  }r\in(0,1].
\end{equation*}
		
\item The set $\Omega$ is an $N^s_{p,q}$-extension domain for some $p\in(0,\infty)$,
$q\in(0,\infty]$, and $s\in(0,{\rm ind}\,(X,\rho))$  in the sense that
there exists a positive constant $C$ such that,
\begin{eqnarray*}
\begin{array}{c}
\mbox{for any }\,u\in N^s_{p,q}(\Omega,\rho,\mu),
\,\,\mbox{ there exists a }\ \widetilde{u}\in N^s_{p,q}(X,\rho,\mu)
\\[6pt]
\mbox{ for which }\,\,\, u=\widetilde{u}|_{\Omega}\,\,
\mbox{ and }\,\,\,\,\|\widetilde{u}\|_{N^s_{p,q}(X,\rho,\mu)}
\leq C\|u\|_{N^s_{p,q}(\Omega,\rho,\mu)}.
\end{array}
\end{eqnarray*}

\item There exist an $s\in(0,{\rm ind}\,(X,\rho))$, a $p\in(0,\infty)$, and a $q\in(0,\infty]$ with $p,\,q>Q/(Q+s)$,
and a bounded linear   operator $\mathscr{E}\colon  N^s_{p,q}(\Omega,\rho,\mu)\to N^s_{p,q}(X,\rho,\mu)$ with the property that $(\mathscr{E}u)|_{\Omega}=u$ for any $u\in N^s_{p,q}(\Omega,\rho,\mu)$.

\item There exist  $\varepsilon,\,s\in(0,\infty)$
satisfying $\varepsilon<s\preceq_q{\rm ind}(\Omega,\rho)$, a $p\in(0,Q/\varepsilon)$,
a $q\in(0,\infty]$, and a $C_S\in(0,\infty)$ such that, for any
ball $B_0:=B_\rho(x_0,R_0)$ with $x_0\in\Omega$ and $R_0\in(0,1]$, one has
\begin{equation*}
\Vert u\Vert_{L^{p^\ast}(B_0\cap\Omega)}\leq
\frac{C_S}{[\mu(B_0)]^{\varepsilon/Q}}\left[R_0^s\Vert u\Vert_{\dot{N}^s_{p,q}(\Omega,\rho,\mu)}+\Vert u\Vert_{L^{p}(\Omega)}\right],
\end{equation*}
whenever $u\in\dot{N}^s_{p,q}(\Omega,\rho,\mu)$. Here, $p^*:=Qp/(Q-\varepsilon p)$.

\item There exist $\varepsilon,\,s\in(0,\infty)$, a
$p\in(0,Q/\varepsilon)$, a $q\in(0,\infty]$, and a $C_P\in(0,\infty)$
satisfying $\varepsilon<s\preceq_q{\rm ind}(\Omega,\rho)$
such that, for any ball $B_0:=B_\rho(x_0,R_0)$ with $x_0\in\Omega$ and $R_0\in(0,1]$, one has
\begin{equation*}
\inf_{\gamma\in\mathbb{R}}\Vert u-\gamma\Vert_{L^{p^\ast}(B_0\cap\Omega)}\leq
\frac{C_PR_0^s}{[\mu(B_0)]^{\varepsilon/Q}}\,\Vert u\Vert_{\dot{N}^s_{p,q}(\Omega,\rho,\mu)},
\end{equation*}
whenever $u\in\dot{N}^s_{p,q}(\Omega,\rho,\mu)$. Here, $p^*:=Qp/(Q-\varepsilon p)$.
		
\item There exist $c_1,\,c_2,\,\omega\in(0,\infty)$, a $q\in(0,\infty]$,
and  $\varepsilon,\,s\in(0,\infty)$ satisfying
$\varepsilon<s\preceq_q{\rm ind}(\Omega,\rho)$ such that
\begin{equation*}
\int_{B_0\cap\Omega} {\rm exp}
\left(c_1\frac{[\mu(B_0)]^{\varepsilon/Q}|u-u_{B_0\cap\Omega}|}{R_0^s\Vert u\Vert_{\dot{N}^s_{Q/\varepsilon,q}(\Omega,\rho,\mu)}}\right)^\omega\,d\mu\leq c_2\mu(B_0),
\end{equation*}
whenever $B_0$ is a $\rho$-ball centered in $\Omega$ having radius $R_0\in(0,1]$, and $u\in \dot{N}^{s}_{Q/\varepsilon,q}(\Omega,\rho,\mu)$ with $\Vert u\Vert_{\dot{N}^s_{Q/\varepsilon,q}(\Omega,\rho,\mu)}>0$.
		
\item There exist $\varepsilon,\,s\in(0,\infty)$, a $p\in(Q/\varepsilon,\infty)$, a $q\in(0,\infty]$, and a $C_H\in(0,\infty)$ satisfying
$\varepsilon<s\preceq_q{\rm ind}(\Omega,\rho)$ such that,
for any ball $B_0:=B_\rho(x_0,R_0)$ with $x_0\in\Omega$ and $R_0\in(0,1]$, every function
$u\in \dot{N}^s_{p,q}(\Omega,\rho,\mu)$ has a H\"older continuous representative of order $s-Q/p$ on $B_0\cap \Omega$, denoted by $u$ again,
satisfying
\begin{equation*}
|u(x)-u(y)|\leq C_H\,[\rho(x,y)]^{s-Q/p}\frac{R_0^{Q/p}}{[\mu(B_0)]^{1/p}}\,\Vert u\Vert_{\dot{N}^s_{p,q}(\Omega,\rho,\mu)},
\quad\forall\,x,\,y\in B_0\cap \Omega.
\end{equation*}
\end{enumerate}
\end{theorem}

The proof of this theorem follows along a similar line of reasoning as in the proof of Theorem~\ref{measdens-ext-sob} where, in place of the $\dot{M}^s_{p,q}$-embeddings from Theorem~\ref{DOUBembedding}, we use the $\dot{N}^s_{p,q}$-embeddings from Theorem~\ref{mainembedding-epsilon}.
We omit the details.

\bigskip

\noindent Ryan Alvarado (Corresponding author)

\medskip

\noindent Department of Mathematics and Statistics, Amherst College, Amherst, MA, USA

\smallskip

\noindent{\it E-mail:} \texttt{rjalvarado@amherst.edu}
\bigskip

\noindent Dachun Yang  and Wen Yuan

\medskip

\noindent Laboratory of Mathematics and Complex Systems (Ministry of Education of China),
School of Mathematical Sciences, Beijing Normal University, Beijing 100875, People's Republic of China

\smallskip

\noindent{\it E-mails:} \texttt{dcyang@bnu.edu.cn} (D. Yang)

\noindent\phantom{{\it E-mails:} }\texttt{wenyuan\@@bnu.edu.cn} (W. Yuan)

\end{document}